\pdfoutput=1
\documentclass[12pt,a4paper,reqno]{amsart}
\usepackage[cp1250]{inputenc}
\usepackage{xcolor}

\input xy 
\usepackage[color,matrix,arrow]{xy} 
\usepackage{comment} 
\xyoption{all}   

\usepackage[margin=1in]{geometry}  
\usepackage{hyperref}     
\hypersetup{colorlinks,citecolor=blue,filecolor=black,linkcolor=blue,urlcolor=blue}

\newtheorem{theorem}{Theorem}[section]
\newtheorem{proposition}[theorem]{Proposition}
\newtheorem{lemma}[theorem]{Lemma}

\newtheorem{example}[theorem]{Example}
\newtheorem{remark}[theorem]{Remark}
\newtheorem{conjecture}[theorem]{Conjecture}
\numberwithin{equation}{section}

\usepackage{mathtools}
\usepackage{pgf,tikz}
\usetikzlibrary{chains}
\usepackage{graphics}
\usepackage{epic,eepic}

\usetikzlibrary{arrows}
\usetikzlibrary{arrows.meta,decorations.markings,quotes}
\usetikzlibrary{arrows.new}

\newcommand*\tcircled[1]{{\normalfont \footnotesize\raisebox{2pt}{\textcircled{\raisebox{-1pt}{#1}}}} }

\newcommand*\circled[1]{\footnotesize\textcircled{\raisebox{-1pt}{#1}}}

\newcommand{\beq}{\begin{equation}}
\newcommand{\eeq}{\end{equation}}
\newcommand{\beqn}{\begin{equation*}}
\newcommand{\eeqn}{\end{equation*}}

\newcommand{\ch}{{\normalfont \textrm{ch}}}
\newcommand{\wt}{{\normalfont \textrm{wt}}}

\author{M. Butorac}
\address{Marijana Butorac:\newline\indent
Faculty of Mathematics, University of Rijeka,\newline\indent Radmile Matej\v{c}i\'{c} 2, 51000 Rijeka, Croatia}
\email{mbutorac@math.uniri.hr}

\author{S. Ko\v zi\' c}
\address{Slaven Ko\v{z}i\'{c}:\newline\indent  Department of Mathematics, Faculty of Science, University of Zagreb,\newline\indent   Bijeni\v{c}ka cesta 30, 10000 Zagreb, Croatia}
\email{kslaven@math.hr}

\author{A. Meurman}\address{Arne Meurman:\newline\indent  Department of Mathematics, University of Lund, \newline\indent Box 118, 22100 Lund, Sweden}
\email{arne.meurman@math.lu.se}

\author{M. Primc}\address{Mirko Primc:\newline\indent  Department of Mathematics, Faculty of Science, University of Zagreb,\newline\indent   Bijeni\v{c}ka cesta 30, 10000 Zagreb, Croatia}
\email{primc@math.hr}

\begin{document}

\title[Lepowsky's and Wakimoto's product   formulas for $C_l^{(1)}$]
{Lepowsky's and Wakimoto's product formulas for the affine Lie algebras $C_l^{(1)}$}

\begin{abstract} 
In this paper, we  recall Lepowsky's and Wakimoto's product character formulas formulated in a new way by using arrays of specialized weighted crystals of negative roots for affine Lie algebras of type $C_l^{(1)}$,  $D_{l+1}^{(2)}$ and $A_{2l}^{(2)}$.
Lepowsky--Wakimoto's infinite periodic products appear as one side of (conjectured) Rogers--Ramanujan-type combinatorial identities for affine Lie algebras of type $C_l^{(1)}$.
\end{abstract}

\maketitle

\tableofcontents

\section{Introduction}\label{section_01}

In the last several decades, numerous applications of Rogers--Ramanujan-type identities have been extensively studied. For example, in the early eighties,  they emerged  in the areas of statistical mechanics \cite{ABF,Bax} and orthogonal polynomials \cite{AI,Br}.  On the other hand, in the last decade,  in addition to their well-known role in combinatorics and number theory (see, e.g., \cite{DK,GOW,W}), their close connections with modular forms \cite{BCFK}, algebraic geometry \cite{BMS,GOR} and double affine Hecke algebras \cite{CF}  were investigated. In this paper, we are interested in a line of research going back to Lepowsky and Milne \cite{LM}, which connected the  product sides
of
Gordon--Andrews--Bressoud's generalization \cite{Go,A1,A2,Br1,Br2}
of the Rogers--Ramanujan identities
 with principally specialized characters of
integrable highest weight modules for the affine Lie algebra $\widehat{\mathfrak{sl}}_2$. Their seminal paper \cite{LM} motivated a    research   of Lepowsky and Wilson \cite{LW1}--\cite{LW4}, which led to  discovery of vertex operators in the principal picture of  integrable highest weight  $\widehat{\mathfrak{sl}}_2$-modules and   bases of vacuum spaces for the
principal Heisenberg subalgebra which are parametrized by partitions satisfying certain difference  conditions. 

J. Lepowsky proved that the principally specialized characters of standard modules for affine Lie algebras can be written as infinite periodic products. M. Wakimoto extended Lepowsky's argument for some other specializations of characters. The aim of this paper is to write explicitly all Lepowsky--Wakimoto's product formulas for affine Lie algebras of types $C_l^{(1)}$, $l\geq 2$, by using arrays of specialized weighted crystals of negative roots for  affine Lie algebras of type $C_l^{(1)}$,  $D_{l+1}^{(2)}$ and $A_{2l}^{(2)}$. These weighted crystals parametrize root vectors for negative roots in terms of crystal bases of adjoint representations of $C_l$ and $B_l$, and  $B_l$-modules $L_{B_l}(\omega_1)$ and $L_{B_l}(2\omega_1)$. Although inspired by the crystal bases theory, we use only the combinatorial notion of weighted crystals for Cartan matrices, and the realization of affine Kac--Moody Lie algebras.

 As an  illustration of the above described approach, we remark that,
by using Wakimoto's product formula for $\ch^{(2,1,1)}L_{C^{(1)}_2}(1,0,0)$, and the basis for the basic module $L_{C^{(1)}_2}(1,0,0)$, we get an analogue of Capparelli's identity:
\newtheorem*{theorem*}{Theorem}
\begin{theorem*}
 The number of colored partitions satisfying level $1$ difference conditions
on the array  
$$
\begin{matrix}
& & & & \textbf{\upshape 7}& & \textbf{\upshape 9}& & \textbf{\upshape 11}&   \\
& & &\textit{\textcolor{red}{5}} & &\textbf{\upshape 8}& &\textbf{\upshape 10} & & \textbf{\upshape 12}\\
& &\textit{\textcolor{red}{4}} & &\textit{\textcolor{red}{6}} & &\textbf{\upshape 9} & &\textbf{\upshape 11} & \\
&\textit{\textcolor{red}{3}} & &\textit{\textcolor{red}{5}} & &\textit{\textcolor{red}{7}} & &\textbf{\upshape 10} & &\textit{\textcolor{red}{13}} \\
\textit{\textcolor{red}{2}}& &\textit{\textcolor{red}{4}} & &\textit{\textcolor{red}{6}} & &\textit{\textcolor{red}{8}} & &\textit{\textcolor{red}{12}} &
\end{matrix}
$$
  equals the number of partitions of $n$ into parts
$$
j\not\equiv 0,\pm1,\pm6,\pm7, 8\mod 16. 
$$
\end{theorem*}

This paper, and the above Theorem in particular, was motivated by the results on Poincar\'{e}--Birkhoff--Witt-type bases for standard modules over $C_l^{(1)}$ in \cite{CMPP,PS1,PS2}.
It is organized as follows. Sections \ref{section_02}--\ref{section_04} serve as an introduction. In Section \ref{section_02}, we provide some preliminary definitions and results on the affine Kac--Moody Lie algebras and their representation theory; see  \cite{Kac}  for more details. Next, in Section \ref{section_03}, we recall the specialized character formulas of Lepowsky \cite{Lepowsky} and Wakimoto \cite{W1}. Finally, in Section \ref{section_04}, we consider certain weighted crystals which arise as connected components of tensor squares of vector representation crystals associated with the complex simple Lie algebras of types $B_l$ and $C_l$.
An introduction to the theory of crystal bases, which, in particular, contains all   notions and results  used in this manuscript, can be found in \cite{HK}.

In Section \ref{section_05}, we use the   weighted crystals from Section \ref{section_04} to explicitly describe the arrays of negative root vectors for the affine Lie algebras of type $C_l^{(1)}$,  $D_{l+1}^{(2)}$ and $A_{2l}^{(2)}$. By assigning   degrees to the negative Chevalley generators  with respect to   certain specializations, in Section \ref{section_06},
we turn these arrays into tables of integers, which we refer to as specialized weighted arrays. Moreover, we express their generating functions in terms of product formulas. Finally, in Section \ref{section_07}, we employ the  product formulas to establish a connection with the aforementioned specialized character formulas of Lepowsky   and Wakimoto.

In Section \ref{section_08}, we use the generating functions from Section \ref{section_06} to generalize Borcea's correspondence  \cite{B} between specialized characters of certain standard modules for the affine Lie algebras $C_l^{(1)}$ and $A_{2l}^{(2)}$ to an arbitrary positive integer level. At the end, in Section \ref{section_09}, we introduce the notion of {\em level $k$-difference conditions}  for colored partitions on arrays connected with a basis of negative root vectors. Using this we recall a conjecture by \cite{CMPP} on combinatorial identities, that served as one of the motivations for the present paper.

\section{Affine Lie algebras of type $C_l^{(1)}$,  $D_{l+1}^{(2)}$ and $A_{2l}^{(2)}$}\label{section_02}

\subsection{Realization of affine Lie algebras}

Let $\mathfrak{g}$ be a complex simple Lie algebra of type $X_{l}$ and let $\mu$ be a diagram automorphism  of $\mathfrak{g}$ of order $r \  (= 1,2,\text{or\ }3)$,
see \cite[Section 7.9]{Kac}.
Let $\langle\cdot,\cdot\rangle$ be an invariant symmetric bilinear form on $\mathfrak{g}$, normalized so that the square length of a long root is 2.
The untwisted affine Lie algebra of type $X_l^{(1)}$ can then be realized as
$$
\mathfrak{g}(X_l^{(1)}) = \mathfrak{g}\otimes\mathbb{C}[t,t^{-1}] \oplus \mathbb{C}c \oplus \mathbb{C}d
$$
with the following commutation relations: denote by $x(n):=x\otimes t^n$ for $x\in\mathfrak{g}, n\in\mathbb{Z}$, then define 
\begin{equation*}
	\begin{aligned}
&[x(m), y(n)] = [x,y](m+n) + m\delta_{m+n,0}\langle x,y \rangle c,\cr
&c\in\operatorname{center}(\mathfrak{g}(X_l^{(1)})),\cr
&[d,x(n)] = nx(n),
\end{aligned}
\end{equation*}
for $x,y \in\mathfrak{g}$, $m,n\in\mathbb{Z}$.

Following \cite{Kac} we shall realize the (twisted) affine Lie algebra of type $X_{l}^{(r)}$ as a subalgebra of the affine Lie algebra $\mathfrak{g}(X_l^{(1)})$.  
For $j\in\mathbb{Z}$ we set $\bar j = j+ r\mathbb{Z}\in\mathbb{Z}/r\mathbb{Z}$ and let $\mathfrak{g}_{\bar j}$ denote the eigenspace of $\mu$ on $\mathfrak{g}$ 
with eigenvalue $e^{2\pi ij/r}$. Define
$$
\mathcal{L}(\mathfrak{g}, \mu) = \bigoplus_{j\in\mathbb{Z}} \mathfrak{g}_{\bar j} \otimes t^j \oplus \mathbb{C}c.
$$
Then $\mathcal{L}(\mathfrak{g}, \mu)$ is a realization of the (twisted) affine Lie algebra of type $X_l^{(r)}$, see \cite[Theorem 8.3]{Kac}.


\subsection{Affine Lie algebras of type $C_l^{(1)}$,  $D_{l+1}^{(2)}$ and $A_{2l}^{(2)}$}

We denote the Kac--Moody Lie algebra  with generalized Cartan matrix $A$ of type $X_l^{(r)}$ as $\mathfrak{g}(X_l^{(r)})$. In this paper we are interested in three types of affine Lie algebras with Dynkin diagrams in Figure \ref{Dynkin diagrams}.
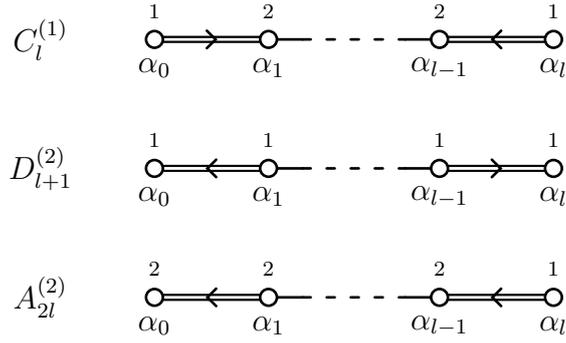
\begin{figure}[h!]
\begin{tikzpicture}[line cap=round,line join=round,>=triangle 45,x=1.50cm,y=1.50cm]
\clip(-0.5,-2.03) rectangle (5.0,1.75);

\draw (0.0,1.0) node {$C_{l}^{(1)}$};

\draw [line width=1.pt] (1.,1.020)-- (2.,1.020);
\draw [line width=1.pt] (1.,0.980)-- (2.,0.980);
\draw [line width=1.pt,color=white,postaction={decorate},decoration={markings,
		mark=at position 0.50 with {\arrow[black,xshift=2.25pt,arrow head=4.5pt]{angle 90 new}}}] (1.,1.)-- (2.,1.);
  
\draw [line width=1.pt] (2.,1.)-- (2.3,1.);
\draw [line width=1.pt,loosely dashed] (2.3,1.)-- (3.2,1.);
\draw [line width=1.pt] (3.2,1.)-- (3.5,1.);

\draw [line width=1.pt] (3.5,1.020)-- (4.5,1.020);
\draw [line width=1.pt] (3.5,0.980)-- (4.5,0.980);
\draw [line width=1.pt,color=white,postaction={decorate},decoration={markings,
		mark=at position 0.50 with {\arrow[black,xshift=2.25pt,arrow head=4.5pt]{angle 90 new reversed}}}] (3.5,1.)-- (4.5,1.);

\draw [line width=1.pt,fill=white] (1.,1.) circle (3.15pt);
\draw [line width=1.pt,fill=white] (2.,1.) circle (3.15pt);
\draw [line width=1.pt,fill=white] (3.5,1.) circle (3.15pt);
\draw [line width=1.pt,fill=white] (4.5,1.) circle (3.15pt);

\draw (1.0,1.0+0.25) node {$\scriptstyle 1$};
\draw (2.0,1.0+0.25) node {$\scriptstyle 2$};
\draw (3.5,1.0+0.25) node {$\scriptstyle 2$};
\draw (4.5,1.0+0.25) node {$\scriptstyle 1$};
\draw (1.0,1.0-0.25) node {$\alpha_0$};
\draw (2.0,1.0-0.25) node {$\alpha_1$};
\draw (3.5,1.0-0.25) node {$\alpha_{l-1}$};
\draw (4.5,1.0-0.25) node {$\alpha_l$};

\begin{scope}[yshift={-1.71cm}]

\draw (0.0,1.0) node {$D_{l+1}^{(2)}$};

\draw [line width=1.pt] (1.,1.020)-- (2.,1.020);
\draw [line width=1.pt] (1.,0.980)-- (2.,0.980);
\draw [line width=1.pt,color=white,postaction={decorate},decoration={markings,
		mark=at position 0.50 with {\arrow[black,xshift=2.25pt,arrow head=4.5pt]{angle 90 new reversed}}}] (1.,1.)-- (2.,1.);
  
\draw [line width=1.pt] (2.,1.)-- (2.3,1.);
\draw [line width=1.pt,loosely dashed] (2.3,1.)-- (3.2,1.);
\draw [line width=1.pt] (3.2,1.)-- (3.5,1.);

\draw [line width=1.pt] (3.5,1.020)-- (4.5,1.020);
\draw [line width=1.pt] (3.5,0.980)-- (4.5,0.980);
\draw [line width=1.pt,color=white,postaction={decorate},decoration={markings,
		mark=at position 0.50 with {\arrow[black,xshift=2.25pt,arrow head=4.5pt]{angle 90 new}}}] (3.5,1.)-- (4.5,1.);

\draw [line width=1.pt,fill=white] (1.,1.) circle (3.15pt);
\draw [line width=1.pt,fill=white] (2.,1.) circle (3.15pt);
\draw [line width=1.pt,fill=white] (3.5,1.) circle (3.15pt);
\draw [line width=1.pt,fill=white] (4.5,1.) circle (3.15pt);

\draw (1.0,1.0+0.25) node {$\scriptstyle 1$};
\draw (2.0,1.0+0.25) node {$\scriptstyle 1$};
\draw (3.5,1.0+0.25) node {$\scriptstyle 1$};
\draw (4.5,1.0+0.25) node {$\scriptstyle 1$};
\draw (1.0,1.0-0.25) node {$\alpha_0$};
\draw (2.0,1.0-0.25) node {$\alpha_1$};
\draw (3.5,1.0-0.25) node {$\alpha_{l-1}$};
\draw (4.5,1.0-0.25) node {$\alpha_l$};

\end{scope}

\begin{scope}[yshift={-3.42cm}]

\draw (0.0,1.0) node {$A_{2l}^{(2)}$};

\draw [line width=1.pt] (1.,1.020)-- (2.,1.020);
\draw [line width=1.pt] (1.,0.980)-- (2.,0.980);
\draw [line width=1.pt,color=white,postaction={decorate},decoration={markings,
		mark=at position 0.50 with {\arrow[black,xshift=2.25pt,arrow head=4.5pt]{angle 90 new reversed}}}] (1.,1.)-- (2.,1.);
  
\draw [line width=1.pt] (2.,1.)-- (2.3,1.);
\draw [line width=1.pt,loosely dashed] (2.3,1.)-- (3.2,1.);
\draw [line width=1.pt] (3.2,1.)-- (3.5,1.);

\draw [line width=1.pt] (3.5,1.020)-- (4.5,1.020);
\draw [line width=1.pt] (3.5,0.980)-- (4.5,0.980);
\draw [line width=1.pt,color=white,postaction={decorate},decoration={markings,
		mark=at position 0.50 with {\arrow[black,xshift=2.25pt,arrow head=4.5pt]{angle 90 new reversed}}}] (3.5,1.)-- (4.5,1.);

\draw [line width=1.pt,fill=white] (1.,1.) circle (3.15pt);
\draw [line width=1.pt,fill=white] (2.,1.) circle (3.15pt);
\draw [line width=1.pt,fill=white] (3.5,1.) circle (3.15pt);
\draw [line width=1.pt,fill=white] (4.5,1.) circle (3.15pt);

\draw (1.0,1.0+0.25) node {$\scriptstyle 2$};
\draw (2.0,1.0+0.25) node {$\scriptstyle 2$};
\draw (3.5,1.0+0.25) node {$\scriptstyle 2$};
\draw (4.5,1.0+0.25) node {$\scriptstyle 1$};
\draw (1.0,1.0-0.25) node {$\alpha_0$};
\draw (2.0,1.0-0.25) node {$\alpha_1$};
\draw (3.5,1.0-0.25) node {$\alpha_{l-1}$};
\draw (4.5,1.0-0.25) node {$\alpha_l$};

\end{scope}

\end{tikzpicture}
\caption{Dynkin diagrams of types $C_{l}^{(1)}$,  $D_{{l}+1}^{(2)}$ and $A_{2{l}}^{(2)}$}
\label{Dynkin diagrams}
\end{figure}
	
	In the case $X_l = C_l$ we have
	\begin{equation}\label{E: affine g(Cl(1))}
		\mathfrak{g}(C_l^{(1)}) = \mathbb{C}c \oplus \mathbb{C}d \oplus \coprod_{j\in\mathbb{Z}}\mathfrak{g}(C_l) \otimes t^j.
	\end{equation}

	In the case $X_l = D_{l+1}$, the fixed point subalgebra is $\mathfrak{g}(B_l)$ \cite[Proposition 7.9(b)]{Kac}, and the $-1$-eigenspace 
	is the $\mathfrak{g}(B_l)$-module $L(\omega_1)$ of highest weight $\omega_1$ (a fundamental weight) \cite[Proposition 7.9(g)]{Kac}. Hence
	\begin{equation}\label{E: affine g(Dl+1(2))}
		\mathfrak{g}(D_{l+1}^{(2)}) = \mathbb{C}c \oplus \mathbb{C}d \oplus \coprod_{j\in\mathbb{Z}}\mathfrak{g}(B_l) \otimes t^{2j} \oplus 
		\coprod_{j\in\mathbb{Z}} L_{B_l}(\omega_1) \otimes t^{2j+1}.
	\end{equation}
	
	In the case $X_l = A_{2l}$, the fixed point subalgebra is $\mathfrak{g}(B_l)$ \cite[Proposition 7.10(b)]{Kac}, and the $-1$-eigenspace is 
	the $\mathfrak{g}(B_l)$-module $L(2\omega_1)$ of highest weight $2\omega_1$ \cite[Proposition 7.10(g)]{Kac}. Hence
	\begin{equation}\label{E: affine g(A2l(2))}
		\mathfrak{g}(A_{2l}^{(2)}) = \mathbb{C}c \oplus \mathbb{C}d \oplus \coprod_{j\in\mathbb{Z}}\mathfrak{g}(B_l) \otimes t^{2j} \oplus 
		\coprod_{j\in\mathbb{Z}} L_{B_l}(2\omega_1) \otimes t^{2j+1}.
	\end{equation}
	
	Fix a Cartan subalgebra $\mathfrak{h}$ in $\mathfrak{g}_{\bar 0}$, with $\mathfrak{g}_{\bar 0}$ the fixed point subalgebra of $\mu$ on
	 $\mathfrak{g}(X_l)$ as above. For a root vector $b = x\otimes t^{i}\in \mathfrak{g}(X_l^{(r)})$ we define the $\wt_\mathfrak{h}(b)$ 
	 as the weight of $x$ with respect to $\mathfrak{h}$ i.e. $\wt_\mathfrak{h}(b) = \nu$ if
	 $$
	 [h,x] = \nu(h) x
	 $$
	for all $h\in \mathfrak{h}$. By identifying $\mathfrak{h}^*$ with the subspace of $(\mathfrak{h} \oplus \mathbb{C}c \oplus \mathbb{C}d)^*$ 
	of weights that take the value $0$ on $c$ and $d$, we then have
	$$
	\wt(b) = \wt_{\mathfrak{h}}(b) + i\delta. 
	$$

For affine Lie algebras of the types $C_l^{(1)}$,  $D_{l+1}^{(2)}$ and $A_{2l}^{(2)}$, $l\geq 2$, we have the following relations for the imaginary root $\delta$ and the canonical central element $c$ of the affine Lie algebras (see \cite{Kac}):
$$
\begin{aligned}
C_{l}^{(1)}\colon\qquad &\delta=\alpha_0+2\alpha_1+\dots+2\alpha_{l-1}+\alpha_l,\quad c=h_0+h_1+\dots+h_{l-1}+h_l,\\
D_{l+1}^{(2)}\colon\qquad &\delta=\alpha_0+\alpha_1+\dots+\alpha_{l-1}+\alpha_l,\quad c=h_0+2h_1+\dots+2h_{l-1}+h_l,\\
A_{2l}^{(2)}\colon\qquad &\delta=2\alpha_0+2\alpha_1+\dots+2\alpha_{l-1}+\alpha_l,\quad c=h_0+2h_1+\dots+2h_{l-1}+2h_l.\\
\end{aligned}
$$

     \subsection{Standard modules and the Weyl--Kac character formula}

Let $A$ be a symmetrizable generalized Cartan matrix and let $\mathfrak{g}(A)$ be the associated Kac--Moody Lie algebra, see \cite[Section 1.3]{Kac}. Let $h_i, e_i, f_i$ be the usual Kac--Moody generators and $\{\alpha_1,\dots,\alpha_l\}$ the set of simple roots. The root system $\Delta$ then decomposes
$$
\Delta = \Delta^+ \cup -\Delta^+
$$
where $\Delta^+ = \Delta \cap \{\sum_i \mathbb{Z}_{\ge 0}\alpha_i\}$. With
$$
\mathfrak{n}_{+} = \bigoplus_{\alpha\in\Delta^+} \mathfrak{g}_\alpha, \qquad \mathfrak{n}_{-} = \bigoplus_{\alpha\in-\Delta^+} \mathfrak{g}_\alpha
$$
one then has the triangular decomposition
$$
\mathfrak{g}(A) = \mathfrak{n}_{+} \oplus \mathfrak{h} \oplus \mathfrak{n}_{-}, 
$$
where $\mathfrak{h} = \bigoplus_i \mathbb{C} h_i$ is the Cartan subalgebra. 

Let $\Lambda\in\mathfrak{h}^*$ be a dominant integral weight i.e. $\Lambda(h_i)\in\mathbb{Z}_{\ge 0}$ for all $i$.
The \emph{standard module} $L(\Lambda)$ is then the unique up to isomorphism irreducible highest weight $\mathfrak{g}(A)$-module with highest weight $\Lambda$ i.e. $L(\Lambda)$ contains a nonzero vector
$v_\Lambda$ of weight $\Lambda$, which is annihilated by $\mathfrak{n}_+$, and generates $L(\Lambda)$ as a $\mathfrak{g}(A)$-module, cf. \cite[Section 9.3]{Kac}. 

The Weyl group $W$ of $\mathfrak{g}(A)$ is the group generated by the simple reflections $r_{\alpha_i}$. The length $\ell(w)$ of an element $w\in W$ is the minimal $\ell$ such that $w$ has an expression $w = r_{\alpha_{i_1}}\cdots r_{\alpha_{i_\ell}}$ as a product of simple reflections. Fix an element $\rho\in\mathfrak{h}^*$ such that $\rho(h_i) = 1$ for all $i$.

To formulate the Weyl--Kac character formula we also need a certain ring of formal series. Consider formal series of the form $s = \sum_{\mu\in\mathfrak{h}^*}c_\mu e^\mu$, $c_\mu\in\mathbb{C}$. The \emph{support} of $s$ is then $\{\mu; c_\mu \neq 0\}$. The set of series with support contained in a finite union of sets $D(\lambda)$, where $D(\lambda) = \lambda + \sum_i \mathbb{Z}_{\ge 0}(-\alpha_i)$, forms a ring under the multiplication determined by $e^\lambda e^\mu = e^{\lambda + \mu}$. Formula \eqref{WKcharacter} below takes place in this ring.

Define the \emph {character} of $L(\Lambda)$ by
$$
\operatorname{ch}L(\Lambda) = \sum_{\mu\in\mathfrak{h}^*} \dim(L(\Lambda)_\mu) e^\mu,
$$
which is a series of the above form.

\begin{theorem} [{Weyl--Kac character formula \cite[Theorem 10.4]{Kac}}] 
	The character of the standard module $L(\Lambda)$ satisfies
	\begin{equation}\label{WKcharacter}
		\operatorname{ch}L(\Lambda) = \frac{\sum_{w\in W} (-1)^{\ell(w)}e^{w(\Lambda + \rho)-\rho}}{\prod_{\alpha\in\Delta^+}(1-e^{-\alpha})^{\dim(\mathfrak{g}_{-\alpha})}}.
		\end{equation}
	\end{theorem}

	\section{Lepowsky's and Wakimoto's theorems on specializations for $C_l^{(1)}$}\label{section_03}
\noindent
Let $s = (s_0,\dots,s_l)$ be a sequence of positive integers. 
Following \cite{W1}, denote by $F_s$ the homomorphism
$$
F_s: \mathbb{C}[[e^{-\alpha_0},\dots,e^{-\alpha_l}]] \to \mathbb{C}[[q]]
$$
determined by $F_s(e^{-\alpha_j}) = q^{s_j}$, called the {\em $s$-specialization}.

Let $A$ be an affine GCM, and let
$$
\mathfrak{g}(A) = \coprod_{\alpha\in\Delta} \mathfrak{g}_\alpha \oplus \mathfrak{h} \oplus \mathbb{C}c \oplus \mathbb{C}d
$$
be the root space decomposition of $\mathfrak{g}(A)$. For $j\in\mathbb{Z}$ set
$$
\Delta_{j,s} = \{\alpha = \sum k_i\alpha_i\in \Delta\mid \sum k_is_i = j\},
$$
and using these define
$$
\mathfrak{g}_j(s;A) = \coprod_{\alpha\in\Delta_{j,s}} \mathfrak{g}_\alpha.
$$
Then
$$
\mathfrak{g}(A) = \coprod_{j\in\mathbb{Z}}\mathfrak{g}_j(s;A)
$$
is a $\mathbb{Z}$-gradation of $\mathfrak{g}(A)$, where $\mathfrak{g}_0 = \mathfrak{h}\oplus\mathbb{C}c\oplus\mathbb{C}d$.
We call this the $s$-gradation of $\mathfrak{g}(A)$. 

Following \cite{W1}, define
$$
Q(s;A) = F_s\left(\prod_{\alpha\in\Delta^+}(1-e^{-\alpha})^{\dim \mathfrak{g}_{-\alpha}}\right) = \prod_{j=1}^\infty (1-q^j)^{\dim \mathfrak{g}_j(s;A)},
$$
the $s$-specialization of the denominator in the Weyl--Kac character formula.

For $L(\Lambda)$ an integrable $\mathfrak{g}(A)$-module, we denote by
$$
\operatorname{ch}^{(s;A)} L(\Lambda) = F_s(e^{-\Lambda} \operatorname{ch} L(\Lambda))
$$
the $s$-specialized character. For $\Lambda = \sum k_i\Lambda_i$, set $s_\Lambda = (k_0, \dots,k_l)$.

\begin{theorem}[{Lepowsky's product formula \cite[Theorem 2.6.]{Lepowsky}}]
	 Let $\Lambda$ be a dominant integral weight, and set ${\bf 1} = (1,\dots,1)$. Then
$$
\operatorname{ch}^{(1,\dots,1;A)}L(\Lambda) = \frac{Q(s_\Lambda + {\bf 1};A^T)}{Q({\bf 1};A)}.
$$
\end{theorem}

\begin{remark}
	In \cite{Lepowsky} this is proved more generally in the setting where $A$ is a symmetrizable GCM.
\end{remark}

The case $A = C_l^{(1)}$ gives

\begin{theorem}
	For a dominant integral weight $\Lambda$,
	\begin{equation}
		\operatorname{ch}^{(1,\dots,1;C_l^{(1)})}L(\Lambda) = \frac{Q(s_\Lambda + {\bf 1};D_{l+1}^{(2)})}{Q({\bf 1};C_l^{(1)})} 
		\end{equation}
		(cf. Theorem \ref{T: Lepowskys formula}).
\end{theorem}

Let $(k_0,\dots,k_l)$ be the coordinates of a dominant integral weight $\Lambda = \sum k_i\Lambda_i$.
Set
 $$
\phi(q) = \prod_{j=1}^\infty (1-q^j).
$$

\begin{theorem} [{\cite[Section 4]{W1}}] The following specializations of characters of $\mathfrak{g}(C_l^{(1)})$-modules hold:
	
\begin{equation}\label{W83_1}
\operatorname{ch}^{(2,1,\dots,1;C_l^{(1)})}L(\Lambda) = \phi(q)^{-l} Q(k_l+1,\dots,k_1+1,2(k_0+1);A_{2l}^{(2)}),
\end{equation}
(cf. Theorem \ref{T: Wakimotos formula for 21...1}),

\begin{equation}\label{W83_2}
\operatorname{ch}^{(1,\dots,1,2;C_l^{(1)})}L(\Lambda) = \phi(q)^{-l} Q(k_0+1,\dots,k_{l-1}+1,2(k_l+1);A_{2l}^{(2)}),
\end{equation}
(cf. Theorem \ref{T: Wakimotos formula for 1...12}),

\begin{equation}\label{W83_3}
\begin{split}
&\operatorname{ch}^{(2,1,\dots,1,2;C_l^{(1)})}L(\Lambda) \cr
&= \frac{Q(2(k_0+1),k_1+1,\dots,k_{l-1}+1,2(k_l+1);C_{l}^{(1)})}{Q(2,1,\dots,1,2;C_l^{(1)})},
\end{split}
\end{equation}
(cf. Theorem \ref{T: Wakimotos formula for 21...12}),

\begin{equation}\label{W83_4}
\begin{split}
&\operatorname{ch}^{(s; C_l^{(1)})}L(\sum_{i=0}^{l-1}(2n-1)\Lambda_i + (n-1)\Lambda_l) \cr
&= Q(ns_l,2ns_{l-1},\dots,2ns_0;A_{2l}^{(2)})/Q(s;C_l^{(1)}),
\end{split}
\end{equation}
(cf. Theorem \ref{T: Wakimotos formula for 2n-1,...,2n-1,n-1}),

\begin{equation}\label{W83_5}
\begin{split}
&\operatorname{ch}^{(s; C_l^{(1)})}L((n-1)\Lambda_0 +\sum_{i=1}^{l}(2n-1)\Lambda_i) \cr
&= Q(ns_0,2ns_1,\dots,2ns_l;A_{2l}^{(2)})/Q(s;C_l^{(1)}),
\end{split}
\end{equation}
(cf. Theorem \ref{T: Wakimotos formula for n-1,2n-1,...,2n-1}),

\begin{equation}\label{W83_6}
\begin{split}
&\operatorname{ch}^{(s; C_l^{(1)})}L((n-1)\Lambda_0 +\sum_{i=1}^{l-1}(2n-1)\Lambda_i + (n-1)\Lambda_l) \cr
&= Q(ns_0,2ns_1,\dots,2ns_{l-1},ns_l;D_{l+1}^{(2)})/Q(s;C_l^{(1)}),
\end{split}
\end{equation}
(cf. Theorem \ref{T: Wakimotos formula for n-1,2n-1,...,2n-1,n-1}).
\end{theorem}

\section{Weighted crystals for  $C_l$ and  $B_l$}\label{section_04}

In this section, we introduce some weighted crystals for Cartan matrices of types $C_l$ and  $B_l$ and their tensor products---see \cite{HK, K90, K91, KKMMNN} for the definitions and results we use. Here we shall not distinguish properly between $U_q(\mathfrak g)$-modules $L(\lambda)$ for quantum groups and $\mathfrak g$-modules $L(\lambda)$ for simple Lie algebras---what we really want is a parametrization of the basis $\mathcal B(\lambda)$ of the $\mathfrak g$-module $L(\lambda)$ ``suggested'' by the crystal bases theory.

\subsection{Weighted crystals $\mathcal B_{C_l}(\omega_1)$ and $\mathcal B_{C_l}(\theta)$ for $C_l$, $l\geq2$}

We start with $l=2$ (see Figure \ref{F: crystal of vector representation for C2}). Kashiwara's \cite{K90, K91}   tensor product of crystals $\mathcal B_{C_2}(\omega_1)\otimes \mathcal B_{C_2}(\omega_1)$ is the union  $\mathcal B_{C_2}(\theta)\cup \mathcal B_{C_2}(\omega_2)\cup \mathcal B_{C_2}(0)$ of $C_2$-crystals (see Figure \ref{F: crystal of the tensor square of vector representation for C2}). 
\smallskip
\begin{figure}[h!]
$$\xymatrix {
 1 \ar@[blue][r]^{1} & 
2 \ar@[green][r]^{2} & 
\bar{2}\ar@[blue][r]^{1} & 
\bar{1}  
}.$$\caption{The crystal of the vector representation for $C_2$}
\label{F: crystal of vector representation for C2}
\end{figure}
\begin{figure}[h!]
$$\xymatrix {
1 \ar@[blue][r]^{1} & 
2 \ar@[green][r]^{2} & 
\bar{2}\ar@[blue][r]^{1} & 
\bar{1} & \\
11\ar@[blue][r]^{1} &
21\ar@[blue][d]^{1}\ar@[green][r]^{2} & 
\bar{2}1\ar@[blue][r]^{1} &
\bar{1}1\ar@[blue][d]^{1} &
1\ar@[blue][d]^{1}\\ 
12\ar@[green][d]^{2} &
22\ar@[green][r]^{2} & 
\bar{2}2\ar@[green][d]^{2} &
\bar{1}2\ar@[green][d]^{2} &
2\ar@[green][d]^{2} \\
1\bar{2}\ar@[blue][r]^{1} &
2\bar{2}\ar@[blue][d]^{1} & 
\bar{2}\bar{2}\ar@[blue][r]^{1} &
\bar{1}\bar{2}\ar@[blue][d]^{1} &
\bar{2}\ar@[blue][d]^{1}\\
1\bar{1} &
2\bar{1}\ar@[green][r]^{2} & 
\bar{2}\bar{1} &
\bar{1}\bar{1} &
\bar{1}      
}$$\caption{The crystal of the tensor square of vector representation for $C_2$}
\label{F: crystal of the tensor square of vector representation for C2}
\end{figure}

Note that the crystal $\mathcal B_{C_2}(\theta)$ parametrizes a weight basis of the adjoint $10$-dimensional representation $L_{C_2}(\theta)$ of the simple Lie algebra $\mathfrak{g}=\mathfrak{g}(C_2)$ of type $C_2$ with the highest weight vector $11$;  
the crystal $\mathcal B_{C_2}(\omega_2)$ parametrizes a weight basis of the $5$-dimensional representation $L_{C_2}(\omega_2)$ with highest weight vector $12$; 
the crystal $\mathcal B_{C_2}(0)=\{1\bar{1}\}$ parametrizes a weight basis of the $1$-dimensional trivial representation.

We can parametrize the weights of $\mathcal B_{C_2}(\omega_1)=\{1,2,\bar{2},\bar{1}\}$ and the root system $\Delta$ of $\mathfrak{g}$ as 
$$
\{\epsilon_1,\epsilon_2,-\epsilon_2,-\epsilon_1\}\quad\text{and}\quad \Delta=
\{\pm(\epsilon_i\pm\epsilon_j)\mid 1\leq i\leq j\leq 2\}\backslash\{0\},
$$
where, as usual, $\{\epsilon_1,\epsilon_2\}$ is the canonical basis of $\mathbb R^2$. Then $11$ is the root vector for the maximal root $\theta=2\epsilon_1$, and the root vectors $\bar{1}2$ and $\bar{2}\bar{2}$ for the negative simple roots $-\alpha_1=-\epsilon_1+\epsilon_2$ and $-\alpha_2=-2\epsilon_2$ are proportional to Chevalley generators $f_1$ and $f_2$. The elements $\bar{i}i$ are proportional to simple coroots $h_i=\alpha_i^\vee$,  $i=1,2$ in the Cartan subalgebra $\mathfrak h\subset\mathfrak g$ (cf. \cite{Bou1, Bou2, Car, Hum}). Moreover, the $i$-arrow $\overset{i}\longrightarrow$, $i=1,2$, denotes the action of the Kashiwara operator $\tilde f_i$ (which is, in this case, proportional to the action of the Chevalley generator  $f_i$, $i=1,2$). 

In Figure \ref{F: crystal of the tensor square of vector representation for C2}, the crystal for the adjoint representation has the shape of a right-angled triangle with vertices $11$, $\bar{1}1$ and $\bar{1}\bar{1}$. The hypotenuse $\{11,22,\bar{2}\bar{2}, \bar{1}\bar{1}\}$ consists of root vectors corresponding to the long roots $\{2\epsilon_1, 2\epsilon_2, -2\epsilon_2, -2\epsilon_1\}$, and the arrows on both catheti are congruent to the arrows of the crystal for the vector representation. 

For $l=3$, the crystal $\mathcal B_{C_3}(\omega_1)$ for the vector representation is shown in Figure \ref{F: crystal of vector representation for C3}.
\begin{figure}[h!]
$$\xymatrix {
1 \ar@[blue][r]^{1} & 
2 \ar@[green][r]^{2} & 
3 \ar@[red][r]^{3} & 
\bar{3} \ar@[green][r]^{2} & 
\bar{2}\ar@[blue][r]^{1} & 
\bar{1}  
}$$\caption{The crystal of the vector representation for $C_3$}
\label{F: crystal of vector representation for C3}
\end{figure}
The   tensor product of crystals $\mathcal B_{C_3}(\omega_1)\otimes \mathcal B_{C_3}(\omega_1)$ is the union  $\mathcal B_{C_3}(\theta)\cup \mathcal B_{C_3}(\omega_2)\cup \mathcal B_{C_3}(0)$ of $C_3$-crystals (see Figure \ref{F: crystal of the tensor square of vector representation for C3}). 
\begin{figure}[h!]
$$\xymatrix {
1 \ar@[blue][r]^{1} & 
2 \ar@[green][r]^{2} & 
3 \ar@[red][r]^{3} & 
\bar{3} \ar@[green][r]^{2} & 
\bar{2}\ar@[blue][r]^{1} & 
\bar{1}  &\\
11\ar@[blue][r]^{1} &
21\ar@[blue][d]^{1}\ar@[green][r]^{2} &
31\ar@[blue][d]^{1}\ar@[red][r]^{3} &
\bar{3}1\ar@[blue][d]^{1}\ar@[green][r]^{2} &
\bar{2}1\ar@[blue][r]^{1} &
\bar{1}1\ar@[blue][d]^{1} &
1\ar@[blue][d]^{1}\\
12\ar@[green][d]^{2} &
22\ar@[green][r]^{2} &
32\ar@[red][r]^{3}\ar@[green][d]^{2} &
\bar{3}2\ar@[green][r]^{2} &
\bar{2}2\ar@[green][d]^{2} &
\bar{1}2\ar@[green][d]^{2} &
2\ar@[green][d]^{2}\\
13\ar@[blue][r]^{1}\ar@[red][d]^{3} &
23\ar@[red][d]^{3} &
33\ar@[red][r]^{3} &
\bar{3}3\ar@[red][d]^{3} &
\bar{2}3\ar@[blue][r]^{1}\ar@[red][d]^{3} &
\bar{1}3\ar@[red][d]^{3} &
3\ar@[red][d]^{3}\ar@[red][d]^{3}\\
1\bar{3}\ar@[blue][r]^{1}\ar@[green][d]^{2} &
2\bar{3}\ar@[green][r]^{2} &
3\bar{3}\ar@[green][d]^{2} &
\bar{3}\bar{3}\ar@[green][r]^{2} &
\bar{2}\bar{3}\ar@[blue][r]^{1}\ar@[green][d]^{2} &
\bar{1}\bar{3}\ar@[green][d]^{2} &
\bar{3}\ar@[green][d]^{2}\\
1\bar{2}\ar@[blue][r]^{1} &
2\bar{2}\ar@[blue][d]^{1} &
3\bar{2}\ar@[blue][d]^{1}\ar@[red][r]^{3} &
\bar{3}\bar{2}\ar@[blue][d]^{1} &
\bar{2}\bar{2}\ar@[blue][r]^{1} &
\bar{1}\bar{2}\ar@[blue][d]^{1} &
\bar{2}\ar@[blue][d]^{1}\\
1\bar{1} &
2\bar{1}\ar@[green][r]^{2} &
3\bar{1}\ar@[red][r]^{3} &
\bar{3}\bar{1}\ar@[green][r]^{2} &
\bar{2}\bar{1} &
\bar{1}\bar{1} &
\bar{1}
}$$\caption{The crystal of the tensor square of vector representation for $C_3$}
\label{F: crystal of the tensor square of vector representation for C3}
\end{figure}
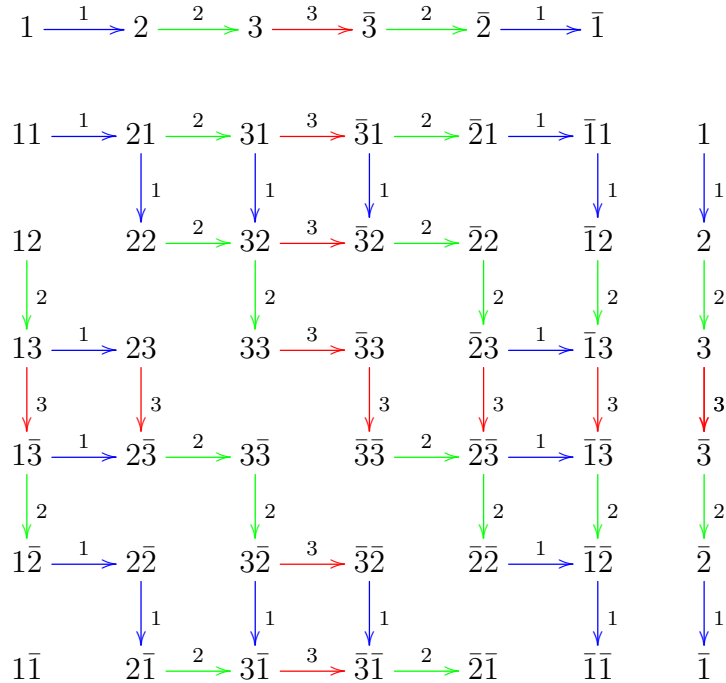
Except for more complicated crystals, we see again that the crystal $\mathcal B_{C_3}(\theta)$ for the adjoint representation has the shape of a right-angled triangle with vertices $11$, $\bar{1}1$ and $\bar{1}\bar{1}$ and that the sequences of arrows on both catheti are congruent to the sequence of arrows in the crystal $\mathcal B_{C_3}(\omega_1)$.
In general, for $l\geq 3$, we have the $C_l$-crystal $\mathcal B_{C_l}(\omega_1)$ for the vector representation (see Figure \ref{F: crystal of vector representation for Cl}) 
and the   tensor product of crystals $\mathcal B_{C_l}(\omega_1)\otimes \mathcal B_{C_l}(\omega_1)$ is the union  $\mathcal B_{C_l}(\theta)\cup \mathcal B_{C_l}(\omega_2)\cup \mathcal B_{C_l}(0)$ of $C_l$-crystals. 
The crystal $\mathcal B_{C_l}(\theta)$ for the adjoint representation of $\mathfrak g=\mathfrak g(C_l)$ of the type $C_l$ has the shape of a right-angled triangle with vertices $11$, $\bar{1}1$ and $\bar{1}\bar{1}$ and the sequence of arrows on both catheti are congruent to the sequence of arrows in the crystal $\mathcal B_{C_l}(\omega_1)$. 
\begin{figure}[h!]
$$\xymatrix {
1 \ar@[black][r]^{1} & 
2 \ar@[black][r]^{2} & 
\cdots&
\cdots \ar@[black][r]^{l-1} & 
l \ar@[black][r]^{l} & 
\bar{l} \ar@[black][r]^{l-1} & 
\cdots &  
\cdots \ar@[black][r]^{2} & 
\bar{2}\ar@[black][r]^{1} & 
\bar{1}  
}$$\caption{The crystal of the vector representation for $C_l$}
\label{F: crystal of vector representation for Cl}
\end{figure}

\subsection{Weighted crystals $\mathcal B_{B_l}(\omega_1)$ and $\mathcal B_{B_l}(\theta)$ for $B_l$, $l\geq2$}\label{ss: vector and adjoint crystals for B_ell}

We start with $l=2$ and the crystal $\mathcal B_{B_2}(\omega_1)$ for the Cartan matrix $B_2$---it is the crystal of the $5$-dimensional vector representation shown in Figure \ref{F: crystal of vector representation for B2}.
\begin{figure}[h!]
$$\xymatrix {
 1 \ar@[blue][r]^{1} & 
2 \ar@[green][r]^{2} & 
0 \ar@[green][r]^{2} & 
\bar{2}\ar@[blue][r]^{1} & 
\bar{1}.  
}$$\caption{The crystal of the vector representation for $B_2$}
\label{F: crystal of vector representation for B2}
\end{figure}
The   tensor product of crystals $\mathcal B_{B_2}(\omega_1)\otimes \mathcal B_{B_2}(\omega_1)$ is the union  $\mathcal B_{B_2}(2\omega_1)\cup \mathcal B_{B_2}(\theta)\cup \mathcal B_{B_2}(0)$ of $B_2$-crystals (see Figure \ref{F: crystal of the tensor square of vector representation for B2}). 
The crystal $\mathcal B_{B_2}(2\omega_1)$ parametrizes a weight basis of the $14$-dimensional representation $L_{B_2}(2\omega_1)$ with highest weight vector $11$; 
the crystal $\mathcal B_{B_2}(0)=\{1\bar{1}\}$ parametrizes a weight basis of the $1$-dimensional trivial representation. We analyze $\mathcal B_{B_2}(2\omega_1)$ in the next subsection.

Note that the crystal $\mathcal B_{B_2}(\theta)$ parametrizes a weight basis of the adjoint $10$-dimensional representation $L_{B_2}(\theta)$ of the simple Lie algebra $\mathfrak{g}=\mathfrak{g}(B_2)$ of type $B_2$ with highest weight vector $12$. 
By translating the diagonal points $00$ and $2\bar{2}$   down along the secondary diagonal $\left\{\bar{1}1, \bar{2}2,00,2\bar{2}, 1\bar{1}\right\} $, we
rectify this crystal to become a triangle (see Figure \ref{F: rectifying the crystal for adjoint B2 module}). Note that we did not change the sequence of arrows going down or going to the right. 
\begin{figure}[h!]
$$\xymatrix@=2em {
1 \ar@[blue][r]^{1} & 
2 \ar@[green][r]^{2} & 
0 \ar@[green][r]^{2} & 
\bar{2}\ar@[blue][r]^{1} & 
 \bar{1} & \\
11\ar@[blue][r]^{1} &
21\ar@[blue][d]^{1}\ar@[green][r]^{2} & 
01\ar@[green][r]^{2}\ar@[blue][d]^{1} &
\bar{2}1\ar@[blue][r]^{1} &
\bar{1}1\ar@[blue][d]^{1}&
1\ar@[blue][d]^{1}\\ 
12\ar@[green][d]^{2} &
22\ar@[green][r]^{2} & 
02\ar@[green][r]^{2} &
\bar{2}2\ar@[green][d]^{2} &
\bar{1}2\ar@[green][d]^{2}&
2\ar@[green][d]^{2} \\
10\ar@[blue][r]^{1}\ar@[green][d]^{2} &
20\ar@[green][r]^{2}&
00\ar@[green][d]^{2}&
\bar{2}0\ar@[blue][r]^{1}\ar@[green][d]^{2}&
\bar{1}0\ar@[green][d]^{2}&
0\ar@[green][d]^{2}\\
1\bar{2}\ar@[blue][r]^{1} &
2\bar{2}\ar@[blue][d]^{1} & 
0\bar{2}\ar@[blue][d]^{1} &
\bar{2}\bar{2}\ar@[blue][r]^{1} &
\bar{1}\bar{2}\ar@[blue][d]^{1}&
\bar{2}\ar@[blue][d]^{1}\\
1\bar{1} &
2\bar{1}\ar@[green][r]^{2} & 
0\bar{1}\ar@[green][r]^{2} &
\bar{2}\bar{1} &
\bar{1}\bar{1} &
\bar{1}      
}$$\caption{The crystal of the tensor square of vector representation for $B_2$}
\label{F: crystal of the tensor square of vector representation for B2}
\end{figure}
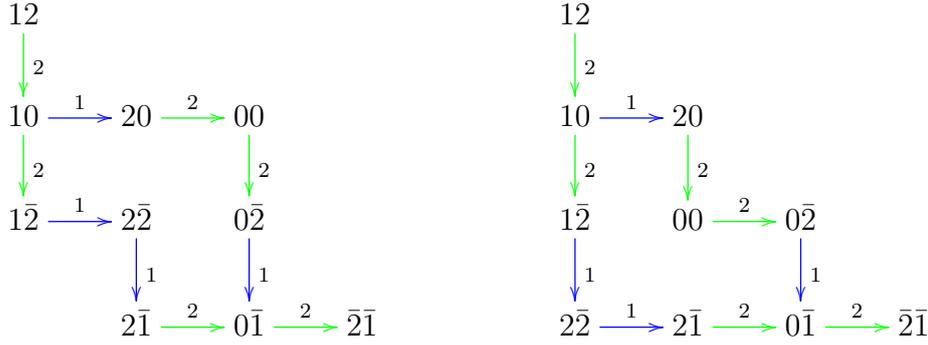
\begin{figure}[h!]
$$\xymatrix@=2em {
12\ar@[green][d]^{2} &&&&& \\
10\ar@[blue][r]^{1}\ar@[green][d]^{2} &
20\ar@[green][r]^{2}&
00\ar@[green][d]^{2}&&& \\
1\bar{2}\ar@[blue][r]^{1} &
2\bar{2}\ar@[blue][d]^{1} & 
0\bar{2}\ar@[blue][d]^{1} &&& \\
&2\bar{1}\ar@[green][r]^{2} & 
0\bar{1}\ar@[green][r]^{2} &
\bar{2}\bar{1} &&   
}
\xymatrix@=2em {
12\ar@[green][d]^{2} &&&&& \\
10\ar@[blue][r]^{1}\ar@[green][d]^{2} &
20\ar@[green][d]^{2}&&&& \\
1\bar{2}\ar@[blue][d]^{1} &
00\ar@[green][r]^{2}& 
0\bar{2}\ar@[blue][d]^{1} &&& \\
2\bar{2}\ar@[blue][r]^{1} &
2\bar{1}\ar@[green][r]^{2} & 
0\bar{1}\ar@[green][r]^{2} &
\bar{2}\bar{1} && 
}$$\caption{Rectifying the crystal for the adjoint $B_2$-module}
\label{F: rectifying the crystal for adjoint B2 module}
\end{figure}
We parametrize the weights of $\mathcal B_{B_2}(\omega_1)=\{1,2,0,\bar{2},\bar{1}\}$ and the root system $\Delta$ of $\mathfrak{g}$ as 
$$
\{\epsilon_1,\epsilon_2,0,-\epsilon_2,-\epsilon_1\}\quad\text{and}\quad \Delta=
\{ \pm(\epsilon_i\pm\epsilon_j)\mid 1\leq i< j\leq 2\}\cup \{\pm\epsilon_i\mid i=1, 2 \}.
$$
Then $12\in \mathcal B_{B_2}(\theta)$ is the root vector for the maximal root $\theta=\epsilon_1+\epsilon_2$, and the root vectors $2\bar{1}$ and $0\bar{2}$ for the negative simple roots $-\alpha_1=-\epsilon_1+\epsilon_2$ and $-\alpha_2=-\epsilon_2$ are proportional to the Chevalley generators $f_1$ and $f_2$. The elements $2\bar{2}$ and $00$ are proportional to the simple coroots $h_i=\alpha_i^\vee$,  $i=1,2$ in the Cartan subalgebra $\mathfrak h\subset\mathfrak g$ (cf. \cite{Bou1, Bou2, Car, Hum}). Moreover, the $i$-arrow, $i=1,2$, denotes the action of the Kashiwara  operator $\tilde f_i$ (which is, in this case, proportional to $f_i$). 
In Figure \ref{F: rectifying the crystal for adjoint B2 module}, the (rectified) crystal for the adjoint representation has the shape of a right-angled triangle with vertices $12$, $2\bar{2}$ and $\bar{2}\bar{1}$. The hypotenuse $\{12,20,0\bar{2}, \bar{2}\bar{1}\}$ consists of root vectors corresponding to the roots $\{\epsilon_1+\epsilon_2, \epsilon_2, -\epsilon_2, -\epsilon_1-\epsilon_2\}$. Unlike the $C_2$-case, the arrows on both catheti are not mutually congruent, nor congruent to the arrows of the crystal for the vector representation. Instead, the arrows on catheti $[12,2\bar 2]$ are congruent to the arrows on the segment $[2,\bar 1]$ of $\mathcal B_{B_2}(\omega_1)$, and the arrows on catheti $[2\bar2,\bar2\bar 1]$ are congruent to the arrows on the segment $[1,\bar 2]$ of $\mathcal B_{B_2}(\omega_1)$.

For $l=3$, the tensor product $\mathcal B_{B_3}(\omega_1)\otimes \mathcal B_{B_3}(\omega_1)$ is the union  $\mathcal B_{B_3}(2\omega_1)\cup \mathcal B_{B_3}(\theta)\cup \mathcal B_{B_3}(0)$ of $B_3$-crystals (see Figure \ref{F: crystal of the tensor square of vector representation for B3}). By translating the diagonal points $00$, $3\bar{3}$ and $2\bar{2}$ down along the secondary diagonal $\left\{\bar{1}1, \bar{2}2,\bar{3}3,00,3\bar{3},2\bar{2}, 1\bar{1}\right\} $, we rectify crystal $\mathcal B_{B_3}(\theta)$ to become a right-angled triangle with vertices $12$, $2\bar{2}$ and $\bar{2}\bar{1}$. Again the arrows on catheti $[12,2\bar 2]$ are congruent to the arrows on the segment $[2,\bar 1]$ of $\mathcal B_{B_3}(\omega_1)$, and the arrows on catheti $[2\bar2,\bar2\bar 1]$ are congruent to the arrows on the segment $[1,\bar 2]$ of $\mathcal B_{B_3}(\omega_1)$.
\begin{figure}[h!]
$$\xymatrix {
1 \ar@[blue][r]^{1} & 
2 \ar@[green][r]^{2} & 
3 \ar@[red][r]^{3} & 
0 \ar@[red][r]^{3} & 
\bar{3} \ar@[green][r]^{2} & 
\bar{2}\ar@[blue][r]^{1} & 
\bar{1}  &\\
11\ar@[blue][r]^{1} &
21\ar@[blue][d]^{1}\ar@[green][r]^{2} &
31\ar@[blue][d]^{1}\ar@[red][r]^{3} &
01 \ar@[blue][d]^{1}\ar@[red][r]^{3} & 
\bar{3}1\ar@[blue][d]^{1}\ar@[green][r]^{2} &
\bar{2}1\ar@[blue][r]^{1} &
\bar{1}1\ar@[blue][d]^{1} &
1\ar@[blue][d]^{1}\\
12\ar@[green][d]^{2} &
22\ar@[green][r]^{2} &
32\ar@[red][r]^{3}\ar@[green][d]^{2} &
02\ar@[red][r]^{3}\ar@[green][d]^{2} & 
\bar{3}2\ar@[green][r]^{2} &
\bar{2}2\ar@[green][d]^{2} &
\bar{1}2\ar@[green][d]^{2} &
2\ar@[green][d]^{2}\\
13\ar@[blue][r]^{1}\ar@[red][d]^{3} &
23\ar@[red][d]^{3} &
33\ar@[red][r]^{3} &
03 \ar@[red][r]^{3} & 
\bar{3}3\ar@[red][d]^{3} &
\bar{2}3\ar@[blue][r]^{1}\ar@[red][d]^{3} &
\bar{1}3\ar@[red][d]^{3} &
3\ar@[red][d]^{3}\ar@[red][d]^{3}\\
10\ar@[blue][r]^{1}\ar@[red][d]^{3} &
20\ar@[red][d]^{3}\ar@[green][r]^{2} &
30\ar@[red][r]^{3} &
00 \ar@[red][d]^{3} & 
\bar{3}0\ar@[red][d]^{3} \ar@[green][r]^{2}&
\bar{2}0\ar@[blue][r]^{1}\ar@[red][d]^{3} &
\bar{1}0\ar@[red][d]^{3} &
0\ar@[red][d]^{3}\ar@[red][d]^{3}\\
1\bar{3}\ar@[blue][r]^{1}\ar@[green][d]^{2} &
2\bar{3}\ar@[green][r]^{2} &
3\bar{3}\ar@[green][d]^{2} &
0\bar{3}  \ar@[green][d]^{2}& 
\bar{3}\bar{3}\ar@[green][r]^{2} &
\bar{2}\bar{3}\ar@[blue][r]^{1}\ar@[green][d]^{2} &
\bar{1}\bar{3}\ar@[green][d]^{2} &
\bar{3}\ar@[green][d]^{2}\\
1\bar{2}\ar@[blue][r]^{1} &
2\bar{2}\ar@[blue][d]^{1} &
3\bar{2}\ar@[blue][d]^{1}\ar@[red][r]^{3} &
0 \bar{2}\ar@[blue][d]^{1}\ar@[red][r]^{3} & 
\bar{3}\bar{2}\ar@[blue][d]^{1} &
\bar{2}\bar{2}\ar@[blue][r]^{1} &
\bar{1}\bar{2}\ar@[blue][d]^{1} &
\bar{2}\ar@[blue][d]^{1}\\
1\bar{1} &
2\bar{1}\ar@[green][r]^{2} &
3\bar{1}\ar@[red][r]^{3} &
0 \bar{1}\ar@[red][r]^{3} & 
\bar{3}\bar{1}\ar@[green][r]^{2} &
\bar{2}\bar{1} &
\bar{1}\bar{1} &
\bar{1}
}$$\caption{The crystal of the tensor square of vector representation for $B_3$}
\label{F: crystal of the tensor square of vector representation for B3}
\end{figure}
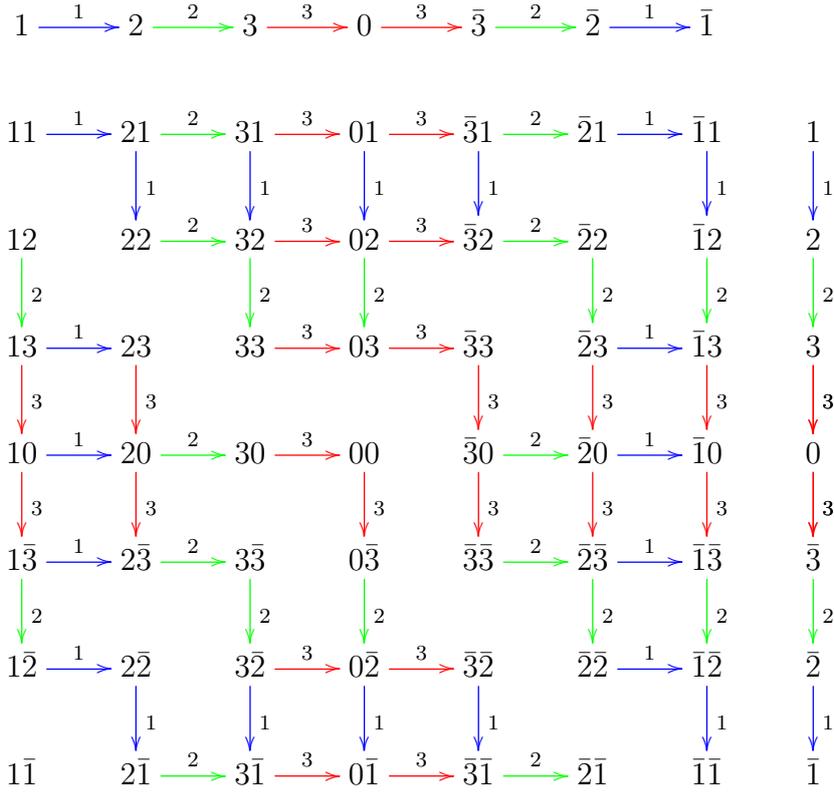
In general, for $l\geq 3$, we have the $B_l$ crystal $\mathcal B_{B_l}(\omega_1)$ for the vector representation (see Figure \ref{F: crystal of vector representation for Bl}) 
and the   tensor product of crystals $\mathcal B_{B_l}(\omega_1)\otimes \mathcal B_{B_l}(\omega_1)$ is the union  $\mathcal B_{B_l}(2\omega_1)\cup \mathcal B_{B_l}(\theta)\cup \mathcal B_{B_l}(0)$ of $B_l$-crystals. 
By translating the diagonal points diagonally down and to the left, we rectify crystal $\mathcal B_{B_l}(\theta)$ to become a right-angled triangle with vertices $12$, $2\bar{2}$ and $\bar{2}\bar{1}$. The arrows on catheti $[12,2\bar 2]$ are congruent to the arrows on the segment $[2,\bar 1]$ of $\mathcal B_{B_l}(\omega_1)$, and the arrows on catheti $[2\bar2,\bar2\bar 1]$ are congruent to the arrows on the segment $[1,\bar 2]$ of $\mathcal B_{B_l}(\omega_1)$.
\begin{figure}[h!]
$$\xymatrix {
1 \ar@[black][r]^{1} & 
2 \ar@[black][r]^{2} & 
\cdots&
\cdots \ar@[black][r]^{l-1} & 
l \ar@[black][r]^{l} & 
0 \ar@[black][r]^{l} & 
\bar{l} \ar@[black][r]^{l-1} & 
\cdots &  
\cdots \ar@[black][r]^{2} & 
\bar{2}\ar@[black][r]^{1} & 
\bar{1}  
}$$\caption{The crystal of the vector representation for $B_l$}
\label{F: crystal of vector representation for Bl}
\end{figure}

\subsection{Weighted crystals $\mathcal B_{B_l}(2\omega_1)$ and $\mathcal B_{B_l}(\theta)$ for $B_l$, $l\geq2$}\label{ss: vector and adjoint crystals for B_ell}

The crystal $\mathcal B_{B_l}(2\omega_1)$ parametrizes a weight basis of the irreducible representation $L_{B_l}(2\omega_1)$ of the simple Lie algebra $\mathfrak{g}=\mathfrak{g}(B_l)$ of type $B_l$ with the highest weight vector $11\in \mathcal B_{B_l}(\omega_1)\otimes \mathcal B_{B_l}(\omega_1)$ (see Figures \ref{F: crystal of the tensor square of vector representation for B2} and \ref{F: crystal of the tensor square of vector representation for B3} for $l=2$ and $l=3$). This crystal is a right-angled triangle with vertices $11$, $\bar{1}1$ and $\bar{1}\bar{1}$. Note that the midpoint on hypotenuse is missing (i.e. $00\in \mathcal B_{B_l}(\theta)$). The arrows on both catheti are congruent to the arrows on $\mathcal B_{B_l}(\omega_1)$.

\section{Arrays of negative root vectors for $C_l^{(1)}$, $D_{l+1}^{(2)}$ and $A^{(2)}_{2l}$}\label{section_05}

\subsection{The array of negative root vectors for $C_l^{(1)}$}

Since the crystal $\mathcal B_{C_l}(\lambda)$ parametrizes a weight basis of the $\mathfrak g(C_l)$-module $L_{C_l}(\lambda)$, from \eqref{E: affine g(Cl(1))} we get (a parametrization of) a weight basis of ${\mathfrak g}(C_l^{(1)})$
\begin{equation}\label{E: a basis of g(Cl(1))}
\{c, d\}\cup {\mathcal B}_{C_l^{(1)}}, \quad\text{where}\quad
{\mathcal B}_{C_l^{(1)}} =\bigcup_{j\in\mathbb Z}  {\mathcal B}_{C_l}(\theta)\otimes t^{j}.
\end{equation}
For $b\in {\mathcal B}_{C_l}(\theta)$ and $i\in\mathbb Z$, set $b_i=b\otimes t^i$. Note that this is a weight vector in the affine Lie algebra ${\mathfrak g}(C_l^{(1)})$ of weight 
$$
\wt(b\otimes t^i)=\wt_\mathfrak h (b)+i\delta.
$$
In each triangle ${\mathcal B}_{C_l}(\theta)\subset{\mathcal B}_{C_l}(\omega_1)\otimes{\mathcal B}_{C_l}(\omega_1)$ (see Figures \ref{F: crystal of the tensor square of vector representation for C2} and \ref{F: crystal of the tensor square of vector representation for C3} for $l=2$ and $l=3$), the $\mathfrak h$-weight changes by a negative simple root $-\alpha_i$, $i\in\{1, \dots, l\}$, when passing from one column to the next, or from one row to the next, according to the sequence of arrows in the crystal of the vector representation  in Figure \ref{F: crystal of vector representation for Cl}. We can ``glue'' the triangles ${\mathcal B}_{C_l}(\theta)\otimes t^{j}$ in ${\mathcal B}_{C_l^{(1)}}$ in such a way that this rule also holds for the negative simple root 
$$
-\alpha_0=\theta-\delta=2\epsilon_1-\delta,
$$
where ``gluing'' means that we place together the catheti 
of the triangles ${\mathcal B}_{C_2}(\theta)\otimes t^{i+1}$ and ${\mathcal B}_{C_2}(\theta)\otimes t^{i}$
along the congruent arrows (see Figures \ref{F: gluing the catheti of triangles for C2}  and  \ref{F: the crystal structure of basis for C2} for $l=2$ and Figure \ref{F: gluing the catheti of triangles for Cl} in general). In this way, ${\mathcal B}_{C_l^{(1)}}$ obtains the structure of a crystal graph (i.e. weighted oriented graph), where dashed   $0$-arrows represent the action of the Kashiwara  operator $\tilde f_0$ (or Chevalley's generator $f_0$). 
\begin{remark}
In the theory of crystal bases, the previous construction is called the affinization of ${\mathcal B}_{C_l}$ with the notation ${\mathcal B}_{C_l^{(1)}}=\bar {\mathcal B}_{C_l} = {\mathcal B}_{C_l}^\text{aff}$ (cf. \cite{HK}). Since the notion of a crystal has many layers, from a simple combinatorial notion of colored directed graph to the notion of crystal basis for a quantum group $U_q(\mathfrak g)$, for ${\mathcal B}_{C_l}$ we prefer the term {\em the array of root vectors for $C_l^{(1)}$} or {\em the arrangement of root vectors for $C_l^{(1)}$}. For us, ${\mathcal B}_{C_l^{(1)}}$ is a part of the basis $\{c, d\}\cup {\mathcal B}_{C_l^{(1)}}$ of ${\mathfrak g}(C_l^{(1)})$ and the $i$-arrows, $i=0, 1, \dots, l$, indicate how the action of Chevalley's generators $f_i$ changes the weight of a basis element, i.e.
$$
\wt(f_i b)=\wt(b)-\alpha_i,\qquad i=0, 1, \dots, l, \quad b\in{\mathcal B}_{C_l^{(1)}}.
$$ 
\end{remark}
\begin{figure}[h!]
$$\xymatrix {
 \bar{1}1_{i+1} \ar@[blue][r]^{1} \ar@{-->}@[red][d]^{0}& 
\bar{1}2_{i+1} \ar@[green][r]^{2}\ar@{-->}@[red][d]^{0} & 
\bar{1}\bar{2}_{i+1}\ar@[blue][r]^{1}\ar@{-->}@[red][d]^{0} & 
\bar{1}\bar{1}_{i+1} \ar@{-->}@[red][d]^{0} \\
11_{i} \ar@[blue][r]^{1} & 
21_{i} \ar@[green][r]^{2} & 
\bar{2}1_{i}\ar@[blue][r]^{1} & 
\bar{1}1_{i}  
}$$\caption{Gluing  the catheti of triangles in ${\mathcal B}_{C_2^{(1)}}$}
\label{F: gluing the catheti of triangles for C2}
\end{figure}
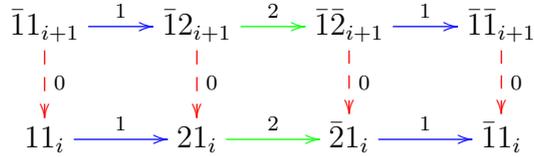
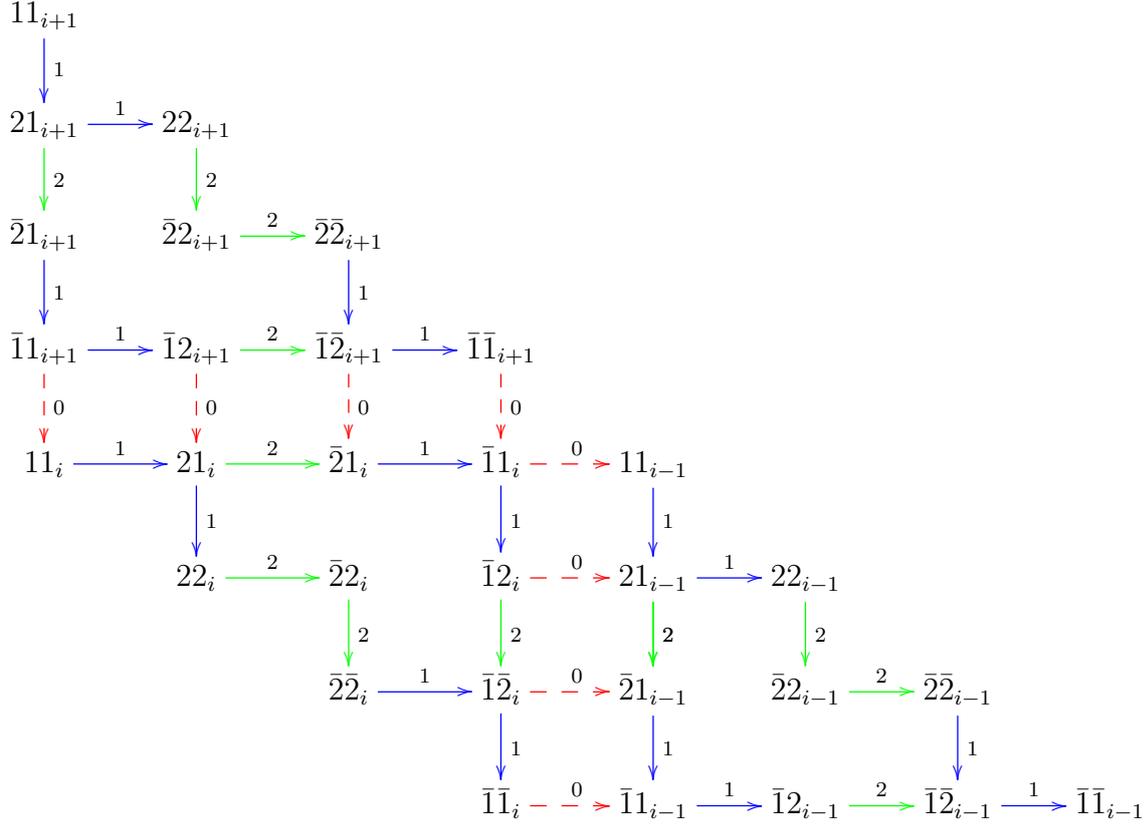
\begin{figure}[h!]
$$\xymatrix {
11_{i+1} \ar@[blue][d]^{1} &&&&&&&\\
21_{i+1} \ar@[blue][r]^{1}\ar@[green][d]^{2} & 
22_{i+1} \ar@[green][d]^{2} &&&&&&\\
\bar{2}1_{i+1} \ar@[blue][d]^{1} & 
\bar{2}2_{i+1} \ar@[green][r]^{2} & 
\bar{2}\bar{2}_{i+1}\ar@[blue][d]^{1} &&&&&\\
\bar{1}1_{i+1} \ar@[blue][r]^{1}\ar@{-->}@[red][d]^{0} & 
\bar{1}2_{i+1} \ar@[green][r]^{2}\ar@{-->}@[red][d]^{0} & 
\bar{1}\bar{2}_{i+1}\ar@[blue][r]^{1}\ar@{-->}@[red][d]^{0} & 
\bar{1}\bar{1}_{i+1}\ar@{-->}@[red][d]^{0} &&&&\\
11_{i}\ar@[blue][r]^{1} &
21_{i}\ar@[blue][d]^{1}\ar@[green][r]^{2} & 
\bar{2}1_{i}\ar@[blue][r]^{1} &
\bar{1}1_{i}\ar@[blue][d]^{1} \ar@{-->}@[red][r]^{0}&
11_{i-1}\ar@[blue][d]^{1}&&&\\ 
&22_{i}\ar@[green][r]^{2} & 
\bar{2}2_{i}\ar@[green][d]^{2} &
\bar{1}2_{i}\ar@[green][d]^{2}\ar@{-->}@[red][r]^{0} &
21_{i-1}\ar@[green][d]^{2} 
\ar@[blue][r]^{1}\ar@[green][d]^{2} & 
22_{i-1} \ar@[green][d]^{2} &&\\
&&\bar{2}\bar{2}_{i}\ar@[blue][r]^{1} &
\bar{1}\bar{2}_{i}\ar@[blue][d]^{1}\ar@{-->}@[red][r]^{0} &
\bar{2}1_{i-1}\ar@[blue][d]^{1}& 
\bar{2}2_{i-1} \ar@[green][r]^{2} & 
\bar{2}\bar{2}_{i-1}\ar@[blue][d]^{1} &&\\
&&&\bar{1}\bar{1}_{i}\ar@{-->}@[red][r]^{0}&\bar{1}1_{i-1}    \ar@[blue][r]^{1} & 
\bar{1}2_{i-1} \ar@[green][r]^{2} & 
\bar{1}\bar{2}_{i-1}\ar@[blue][r]^{1} & 
\bar{1}\bar{1}_{i-1} & 
}$$\caption{Gluing triangles in ${\mathcal B}_{C_2^{(1)}}$}
\label{F: the crystal structure of basis for C2}
\end{figure}
\begin{figure}[h!]
$$\xymatrix@=1.6em {
\bar{1}1_{i+1} \ar@[black][r]^{1}\ar@{-->}@[red][d]^{0} & 
\bar{1}2_{i+1} \ar@[black][r]^{2}\ar@{-->}@[red][d]^{0} &&
\cdots \ar@[black][r]^{l-1} & 
\bar{1}l_{i+1} \ar@[black][r]^{l} \ar@{-->}@[red][d]^{0}& 
\bar{1}\bar{l}_{i+1} \ar@[black][r]^{l-1}\ar@{-->}@[red][d]^{0} &&  
\cdots \ar@[black][r]^{2} & 
\bar{1}\bar{2}_{i+1}\ar@[black][r]^{1}\ar@{-->}@[red][d]^{0} & 
\bar{1}\bar{1} _{i+1}\ar@{-->}@[red][d]^{0}
\\
11_{i} \ar@[black][r]^{1} & 
21_{i} \ar@[black][r]^{2} &&
\cdots \ar@[black][r]^{l-1} & 
l 1_{i}\ar@[black][r]^{l} & 
\bar{l}1_{i} \ar@[black][r]^{l-1} &&  
\cdots \ar@[black][r]^{2} & 
\bar{2}1_{i}\ar@[black][r]^{1}& 
\bar{1}1 _{i}  
}$$\caption{Gluing the catheti of triangles ${\mathcal B}_{C_l}(\theta)\otimes t^{i+1}$ and ${\mathcal B}_{C_l}(\theta)\otimes t^{i}$}
\label{F: gluing the catheti of triangles for Cl}
\end{figure}
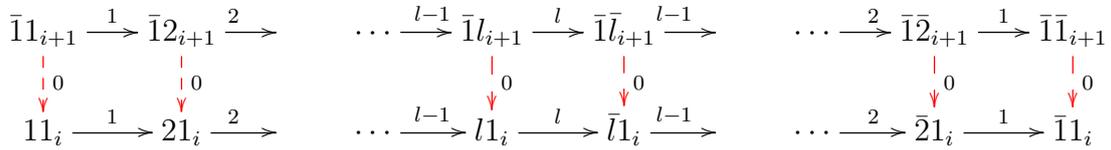
\begin{remark}
We shall usually write the array ${\mathcal B}_{C_l^{(1)}}$ rotated by $\pi/4$. We shall mainly write only negative root vectors ${\mathcal B}_{C_l^{(1)}}^-$, starting from the left column with $f_0=11\otimes t^{-1}$, $f_1=\bar{1}2\otimes  t^0$, \dots, $f_l=\bar{l}\bar{l}\otimes  t^0$ (for $l=2$ see Figure \ref{AC2}), with $f_0$ in the bottom row. Finally, for our purposes, it is not necessary to specify precisely the elements $b_i\in {\mathcal B}_{C_l^{(1)}}$, as all the information needed is encoded in the ``colored'' arrows and ``positions of triangles'' ${\mathcal B}_{C_l}(\theta)\otimes t^{i}$ (for $l=2$ see Figure \ref{F: the arrangement of negative root vectors  for C2(1)} where the positions of the triangles are denoted by  ``circles'' and ``bullets''). 
\end{remark}
\begin{proposition}
The {\em array   of  negative root vectors for $C_l^{(1)}$}, denoted by $\mathcal{B}^{-}_{C_l^{(1)}}$, is a colored directed graph.
Its  nodes, which represent the basis vectors of $\mathfrak g(C_l^{(1)})^{-}$, are organized into $2l+1$ rows and two sequences of diagonals with $2l+1$ (or fewer) nodes. Its arrows  indicate the action of Chevalley's generators $f_0,f_1,\ldots,f_l$ on the negative root subspaces and they are colored by  $0,1,\ldots ,l$, respectively. Removing the arrows of color $0$, which correspond to the action of  $f_0=11\otimes t^{-1}$, the graph decomposes into an infinite union of connected subgraphs. We shall refer to these subgraphs as {\em triangles}. The first triangle, positioned in the upper left corner,   corresponds to the root vector basis of  $l^2$-dimensional Lie subalgebra $\mathfrak{n}_{-}\otimes t^{0}\cong \mathfrak{n}_{-}$ of the simple Lie algebra $\mathfrak g(C_l)$ of type $C_l $. The remaining triangles
are crystal graphs ${\mathcal B}_{C_l}(\theta)$ of the adjoint representation of $\mathfrak g(C_l)$ given over $L_{C_l}(\theta) \otimes t^{i}$ with $i< 0$. 
The weights of nodes in $\mathcal{B}^{-}_{C_l^{(1)}}$ are periodic with  period $\delta$. We place the Chevalley generator $f_0$ in the bottom row.

The main property of  $\mathcal{B}^{-}_{C_l^{(1)}}$ is that the weights of the corresponding points on two adjacent diagonals differ by $-\alpha_i$ if there is an $i$-arrow between these two diagonals. The sequence of arrows between diagonals is determined by the sequence of arrows in ${\mathcal B}_{C_l}(\omega_1)$.
\end{proposition}
\begin{example}
The array $\mathcal{B}^{-}_{C_2^{(1)}}$ is given in   Figure \ref{AC2}.  The triangle with vertices $\bar{2}\bar{2}_{0}$, $\bar{1}2_{0}$ and $\bar{1}\bar{1}_{0}$ corresponds to the basis of the $4$-dimensional Lie subalgebra $\mathfrak{n}_{-}\otimes t^{0}$ isomorphic to the nilpotent subalgebra $\mathfrak{n}_{-}$ of the Lie algebra of type $C_2$, while the remaining triangles ${\mathcal B}_{C_l}(\theta)\otimes t^i$, $i<0$, possess catheti consisting of two $1$-arrows and one $2$-arrow, which represent the action of  $f_1$ and $f_2$, respectively. Finally, the triangles are connected by the action of $f_0 $ which is indicated by the dashed $0$-arrows. The same array $\mathcal{B}^{-}_{C_2^{(1)}}$ without specified basis elements $b_i=b\otimes t^i$ is given in Figure \ref{F: the arrangement of negative root vectors  for C2(1)}.
\end{example}
\begin{figure}[h!]
$$\xymatrix@=
0.8em {
\bar{2}\bar{2}_{0}\ar@[blue][rd]^{1} &&
\bar{1}\bar{1}_{0}\ar@{-->}@[red][rd]^{0}&&
11_{-2}\ar@[blue][rd]^{1}&&
22_{-2}\ar@[green][rd]^{2}&&
\bar{2}\bar{2}_{-2}\ar@[blue][rd]^{1}&&
\bar{1}\bar{1}_{-2}&
\hspace{-20pt}\cdots
\\
&\bar{1}\bar{2}_{0}\ar@[blue][ru]^{1}\ar@{-->}@[red][rd]^{0}&&
\bar{1}1_{-1}\ar@[blue][rd]^{1}\ar@{-->}@[red][ru]^{0}&&
21_{-2}\ar@[blue][ru]^{1}\ar@[green][rd]^{2}&&
\bar{2}2_{-2}\ar@[green][ru]^{2}&&
\bar{1}\bar{2}_{-2}\ar@[blue][ru]^{1}\ar@{-->}@[red][rd]^{0}&&
 \\
\bar{1}2_{0}\ar@[green][ru]^{2}\ar@{-->}@[red][rd]^{0} &&
\bar{2}1_{-1}\ar@[blue][ru]^{1}&&
\bar{1}2_{-1}\ar@[green][rd]^{2}\ar@{-->}@[red][ru]^{0}&&
\bar{2}1_{-2}\ar@[blue][rd]^{1}&&
\bar{1}2_{-2}\ar@[green][ru]^{2}\ar@{-->}@[red][rd]^{0}&&
\bar{2}1_{-3}&
\hspace{-20pt}\cdots
\\
&21_{-1}\ar@[green][ru]^{2}\ar@[blue][rd]^{1}&&
\bar{2}2_{-1}\ar@[green][rd]^{2}&&
\bar{1}\bar{2}_{-1}\ar@[blue][rd]^{1}\ar@{-->}@[red][ru]^{0}&&
\bar{1}1_{-2}\ar@[blue][ru]^{1}\ar@{-->}@[red][rd]^{0}&&
21_{-3}\ar@[blue][rd]^{1}\ar@[green][ru]^{2}&&
 \\
11_{-1}\ar@[blue][ru]^{1} &&
22_{-1}\ar@[green][ru]^{2}&&
\bar{2}\bar{2}_{-1}\ar@[blue][ru]^{1}&&
\bar{1}\bar{1}_{-1}\ar@{-->}@[red][ru]^{0}&&
11_{-3}\ar@[blue][ru]^{1}&&
22 _{-3}&
\hspace{-20pt}\cdots
}$$\caption{The arrangement of negative root vectors with elements $b_i$ for $C_2^{(1)}$}\label{AC2}
\end{figure}
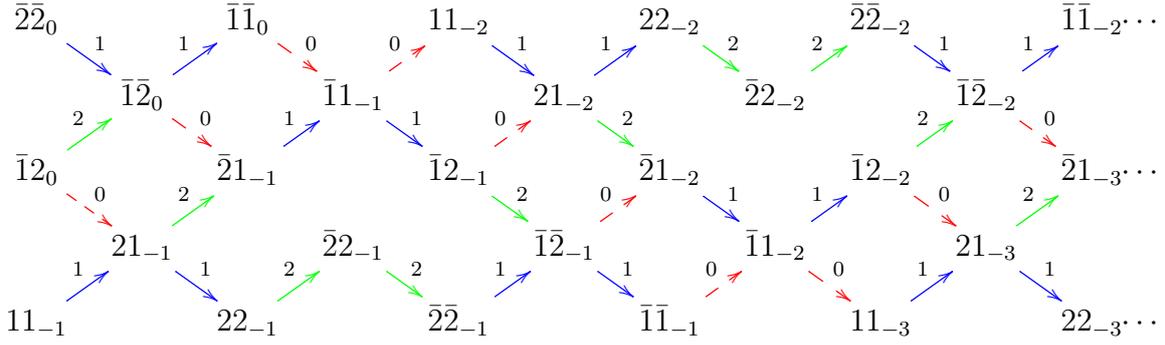
\begin{figure}[h!]
$$\xymatrix@=1.5em {
f_2\ar@[blue][rd]^{1} &&
\circ\ar@{-->}@[red][rd]^{0}&&
\circ\ar@[blue][rd]^{1}&&
\circ\ar@[green][rd]^{2}&&
\circ\ar@[blue][rd]^{1}&&
\circ&
\hspace{-20pt}\cdots
\\
&\circ\ar@[blue][ru]^{1}\ar@{-->}@[red][rd]^{0}&&
\bullet\ar@[blue][rd]^{1}\ar@{-->}@[red][ru]^{0}&&
\circ\ar@[blue][ru]^{1}\ar@[green][rd]^{2}&&
\circ\ar@[green][ru]^{2}&&
\circ\ar@[blue][ru]^{1}\ar@{-->}@[red][rd]^{0}&&
\\
f_1\ar@[green][ru]^{2}\ar@{-->}@[red][rd]^{0} & &
\bullet\ar@[blue][ru]^{1}&&
\bullet\ar@[green][rd]^{2}\ar@{-->}@[red][ru]^{0}&&
\circ\ar@[blue][rd]^{1}&&
\circ\ar@[green][ru]^{2}\ar@{-->}@[red][rd]^{0}&&
\bullet&
\hspace{-20pt}\cdots
\\
&\bullet\ar@[green][ru]^{2}\ar@[blue][rd]^{1}&&
\bullet\ar@[green][rd]^{2}&&
\bullet\ar@[blue][rd]^{1}\ar@{-->}@[red][ru]^{0}&&
\circ\ar@[blue][ru]^{1}\ar@{-->}@[red][rd]^{0}&&
\bullet\ar@[blue][rd]^{1}\ar@[green][ru]^{2}&&
 \\
f_0\ar@[blue][ru]^{1} &&
\bullet\ar@[green][ru]^{2}&&
\bullet\ar@[blue][ru]^{1}&&
\bullet\ar@{-->}@[red][ru]^{0}&&
\bullet\ar@[blue][ru]^{1}&&
\bullet &
\hspace{-20pt}\cdots
}$$\caption{The arrangement of negative root vectors  for $C_2^{(1)}$}
\label{F: the arrangement of negative root vectors  for C2(1)}
\end{figure}
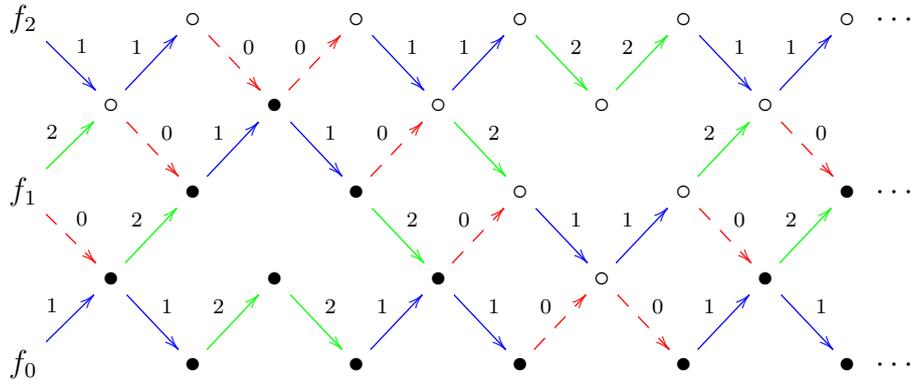

\subsection{The array of negative root vectors for $D_{l+1}^{(2)}$}

Since the crystal $\mathcal B_{B_l}(\lambda)$ parametrizes a weight basis of $\mathfrak g(B_l)$-module $L_{B_l}(\lambda)$, from  \eqref{E: affine g(Dl+1(2))}  we get (a parametrization of) a weight basis of ${\mathfrak g}(D_{l+1}^{(2)})$:
\begin{equation}\label{E: a basis of g(Dl+1(2))}
\{c, d\}\cup {\mathcal B}_{D_{l+1}^{(2)}}, \quad\text{where}\quad
{\mathcal B}_{D_{l+1}^{(2)}} =\bigcup_{j\in\mathbb Z}  {\mathcal B}_{B_l}(\theta)\otimes t^{2j} \ \cup \
\bigcup_{j\in\mathbb Z}  {\mathcal B}_{B_l}(\omega_1)\otimes t^{2j+1}.
\end{equation}
For $i\in\mathbb Z$, set $b_{2i}=b\otimes t^{2i}$ for $b\in {\mathcal B}_{B_l}(\theta)$ and 
$b_{2i+1}=b\otimes t^{2i+1}$ for $b\in {\mathcal B}_{B_l}(\omega_1)$. Note that these are weight vectors in the affine Lie algebra ${\mathfrak g}(D_{l+1}^{(2)})$ of weight 
$$
\wt(b\otimes t^j)=\wt_\mathfrak h (b)+j\delta.
$$
We can glue the catheti of the triangles ${\mathcal B}_{B_l}(\theta)\otimes t^{2i}$ to the lines ${\mathcal B}_{B_l}(\omega_1)\otimes t^{2i\pm1}$ in ${\mathcal B}_{D_{l+1}^{(2)}}$ along the congruent arrows; for $l=2$, this is shown in Figure \ref{F: gluing triangles with lines for D(3)(2)}, and, in general, in Figure \ref{F: gluing catheti of triangles with lines in D(l+1)(2)}. The 
dashed $0$-arrows represent the action of the Chevalley generator $f_0$ on the array of negative root vectors ${\mathcal B}_{D_{l+1}^{(2)}}$, which changes weights of the corresponding vectors by            
$$
-\alpha_0=\omega_1-\delta=\epsilon_1-\delta.
$$
We shall write the array of negative root vectors ${\mathcal B}_{D_{l+1}^{(2)}}^-$ rotated by $\pi/4$, starting with $f_0=1\otimes t^{-1}$, $f_1=2\bar{1}\otimes  t^0$, \dots,  $f_l=0\bar{l}\otimes  t^0$.
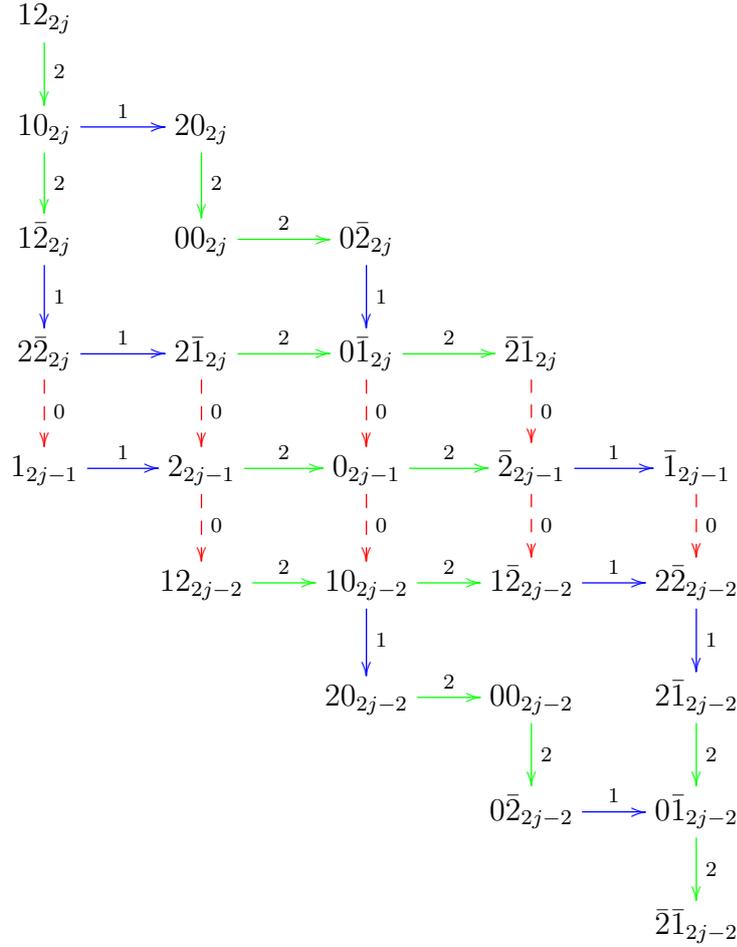
\begin{figure}[h!]
$$
\xymatrix@=2em {
12_{2j}\ar@[green][d]^{2} &
&&&& \\
10_{2j}\ar@[blue][r]^{1}\ar@[green][d]^{2} &
20_{2j}\ar@[green][d]^{2}&
&
&& \\
1\bar{2}_{2j}\ar@[blue][d]^{1} &
00_{2j}\ar@[green][r]^{2}& 
0\bar{2}_{2j}\ar@[blue][d]^{1} &
&& \\
2\bar{2}_{2j}\ar@[blue][r]^{1} \ar@{-->}@[red][d]^{0} &
2\bar{1}_{2j}\ar@[green][r]^{2} \ar@{-->}@[red][d]^{0}& 
0\bar{1}_{2j}\ar@[green][r]^{2}\ar@{-->}@[red][d]^{0} &
\bar{2}\bar{1}_{2j} \ar@{-->}@[red][d]^{0}&&   \\
1_{2j-1} \ar@[blue][r]^{1} & 
2_{2j-1}  \ar@[green][r]^{2} \ar@{-->}@[red][d]^{0}& 
0_{2j-1}  \ar@[green][r]^{2}\ar@{-->}@[red][d]^{0} & 
\bar{2}_{2j-1} \ar@[blue][r]^{1}\ar@{-->}@[red][d]^{0} & 
 \bar{1}_{2j-1}  \ar@{-->}@[red][d]^{0} & \\
& 12_{2j-2}\ar@[green][r]^{2} &
10_{2j-2}\ar@[green][r]^{2}\ar@[blue][d]^{1} & 
1\bar{2}_{2j-2}\ar@[blue][r]^{1} &
2\bar{2}_{2j-2}\ar@[blue][d]^{1} &   \\
& &
20_{2j-2}\ar@[green][r]^{2} & 
00_{2j-2}\ar@[green][d]^{2} &
2\bar{1}_{2j-2} \ar@[green][d]^{2}&   \\
& & & 
0\bar{2}_{2j-2}\ar@[blue][r]^{1} &
0\bar{1}_{2j-2} \ar@[green][d]^{2}&   \\
& & & &
\bar{2}\bar{1}_{2j-2} &   \\
}$$\caption{Gluing triangles to lines in ${\mathcal B}_{D_{3}^{(2)}}$}
\label{F: gluing triangles with lines for D(3)(2)}
\end{figure}
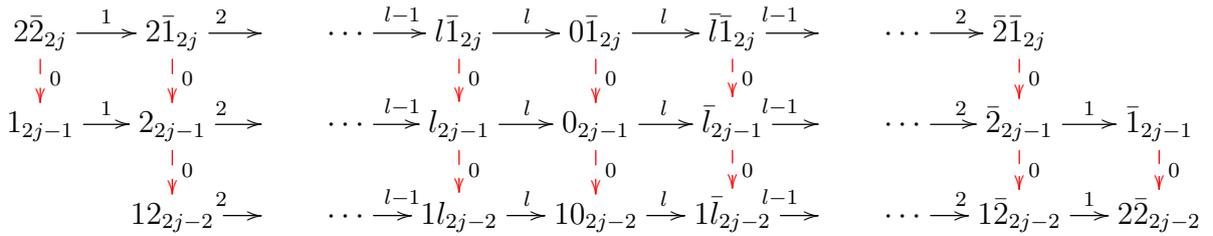
\begin{figure}[h!]
$$\xymatrix@=1.2em  {
2\bar{2}_{2j} \ar@[black][r]^{1} \ar@{-->}@[red][d]^{0} & 
2\bar{1}_{2j} \ar@[black][r]^{2} \ar@{-->}@[red][d]^{0} & 
&
\cdots \ar@[black][r]^{l-1} & 
l\bar{1}_{2j} \ar@[black][r]^{l} \ar@{-->}@[red][d]^{0} & 
0\bar{1}_{2j} \ar@[black][r]^{l}  \ar@{-->}@[red][d]^{0}& 
\bar{l}\bar{1}_{2j} \ar@[black][r]^{l-1} \ar@{-->}@[red][d]^{0} & 
&  
\cdots \ar@[black][r]^{2}  & 
\bar{2}\bar{1}_{2j} \ar@{-->}@[red][d]^{0} & 
  \\
1_{2j-1} \ar@[black][r]^{1} & 
2_{2j-1} \ar@[black][r]^{2}\ar@{-->}@[red][d]^{0} & 
&
\cdots \ar@[black][r]^{l-1} & 
l_{2j-1} \ar@[black][r]^{l} \ar@{-->}@[red][d]^{0}& 
0_{2j-1} \ar@[black][r]^{l} \ar@{-->}@[red][d]^{0}& 
\bar{l}_{2j-1} \ar@[black][r]^{l-1} \ar@{-->}@[red][d]^{0}& 
&  
\cdots \ar@[black][r]^{2} & 
\bar{2}_{2j-1}\ar@[black][r]^{1} \ar@{-->}@[red][d]^{0}& 
\bar{1}_{2j-1} \ar@{-->}@[red][d]^{0} \\
 & 
12_{2j-2} \ar@[black][r]^{2} & 
&
\cdots \ar@[black][r]^{l-1} & 
1l_{2j-2} \ar@[black][r]^{l} & 
10_{2j-2} \ar@[black][r]^{l} & 
1\bar{l}_{2j-2} \ar@[black][r]^{l-1} & 
&  
\cdots \ar@[black][r]^{2} & 
1\bar{2}_{2j-2}\ar@[black][r]^{1} & 
2\bar{2}_{2j-2}  
}$$\caption{Gluing catheti of triangles to lines in ${\mathcal B}_{D_{l+1}^{(2)}}$}
\label{F: gluing catheti of triangles with lines in D(l+1)(2)}
\end{figure}
\begin{figure}[h!]
$$\xymatrix@=0.15em {
 0\bar{2}_{0}\ar@[blue][dr]^{1}&&
\bar{2}\bar{1}_{0}\ar@{-->}@[red][rd]^{0}&&
\bar{1}_{-1}\ar@{-->}@[red][rd]^{0}&&
1_{-3}\ar@[blue][dr]^{1}&&
12_{-4}\ar@[green][dr]^{2}&&
20_{-4}\ar@[green][dr]^{2}&&
0\bar{2}_{-4}\ar@[blue][dr]^{1}&&&
\hspace{-10pt}\cdots\\
&
0\bar{1}_{0}\ar@{-->}@[red][rd]^{0}\ar@[green][ur]^{2}&&
\bar{2}_{-1}\ar@{-->}@[red][rd]^{0}\ar@[blue][ur]^{1}&&
2\bar{2}_{-2}\ar@[blue][dr]^{1}\ar@{-->}@[red][ru]^{0}&&
2_{-3}\ar@{-->}@[red][ru]^{0}\ar@[green][dr]^{2}&&
10_{-4}\ar@[green][dr]^{2}\ar@[blue][ur]^{1}&&
00_{-4}\ar@[green][ur]^{2}&&
0\bar{1}_{-4}&&
\\
2\bar{1}_{0}\ar@{-->}@[red][rd]^{0}\ar@[green][ur]^{2}&&
0_{-1}\ar@{-->}@[red][rd]^{0}\ar@[green][ur]^{2}&&
1\bar{2}_{-2}\ar@[blue][ur]^{1}&&
2\bar{1}_{-2}\ar@{-->}@[red][ru]^{0}\ar@[green][dr]^{2}&&
0_{-3}\ar@{-->}@[red][ru]^{0}\ar@[green][dr]^{2}&&
1\bar{2}_{-4}\ar@[blue][dr]^{1}&&
2\bar{1}_{-4}\ar@{-->}@[red][rd]^{0}\ar@[green][ur]^{2}&&&
\hspace{-10pt}\cdots\\
&
2_{-1}\ar@{-->}@[red][rd]^{0}\ar@[green][ur]^{2}&&
10_{-2}\ar@[green][ur]^{2}\ar@[blue][dr]^{1}&&
00_{-2}\ar@[green][dr]^{2}&&
0\bar{1}_{-2}\ar@{-->}@[red][ru]^{0}\ar@[green][dr]^{2}&&
\bar{2}_{-3}\ar@{-->}@[red][ru]^{0}\ar@[blue][dr]^{1}&&
2\bar{2}_{-4}\ar@[blue][ur]^{1}\ar@{-->}@[red][rd]^{0}&&
2_{-5}&&
\\
1_{-1}\ar@[blue][ur]^{1}&&
12_{-2}\ar@[green][ur]^{2}&&
20_{-2}\ar@[green][ur]^{2}&&
0\bar{2}_{-2}\ar@[blue][ur]^{1}&&
\bar{2}\bar{1}_{-2}\ar@{-->}@[red][ru]^{0}&&
\bar{1}_{-3}\ar@{-->}@[red][ru]^{0}&&
1_{-5}\ar@[blue][ur]^{1}&&&
\hspace{-10pt}\cdots
}$$\caption{Arrangement of negative root vectors with elements $b_i$ for $D_3^{(2)}$}\label{AD3}
\end{figure}
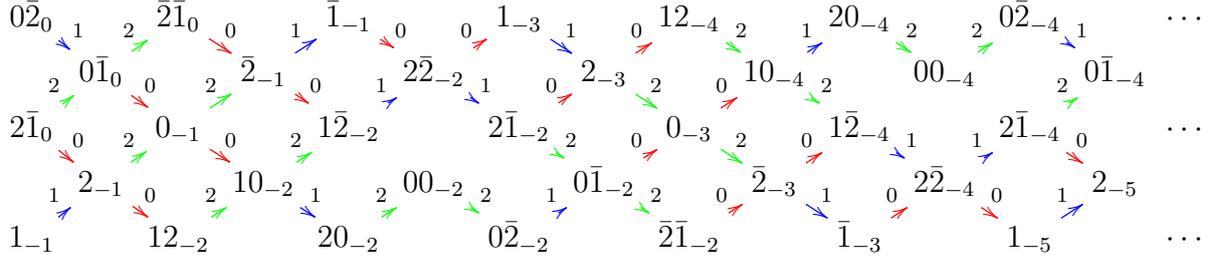
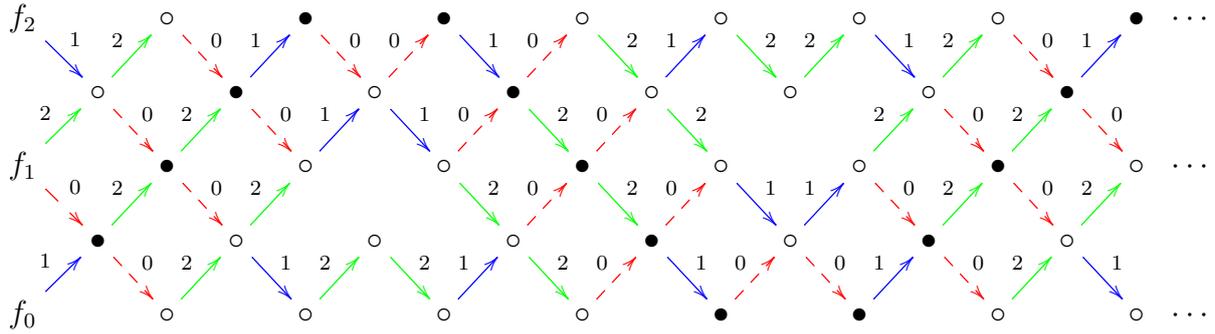
\begin{figure}[h!]
$$\xymatrix@=1.1em {
f_2\ar@[blue][dr]^{1}&
&
\circ\ar@{-->}@[red][rd]^{0}&
&
\bullet\ar@{-->}@[red][rd]^{0}&
&
\bullet\ar@[blue][dr]^{1}&
&
\circ\ar@[green][dr]^{2}&
&
\circ\ar@[green][dr]^{2}&
&
\circ\ar@[blue][dr]^{1}&
&
\circ\ar@{-->}@[red][rd]^{0}&
&
\bullet&\hspace{-10pt}\cdots\\
&
\circ\ar@{-->}@[red][rd]^{0}\ar@[green][ur]^{2}&
&
\bullet\ar@{-->}@[red][rd]^{0}\ar@[blue][ur]^{1}&
&
\circ\ar@[blue][dr]^{1}\ar@{-->}@[red][ru]^{0}&
&
\bullet\ar@{-->}@[red][ru]^{0}\ar@[green][dr]^{2}&
&
\circ\ar@[green][dr]^{2}\ar@[blue][ur]^{1}&
&
\circ\ar@[green][ur]^{2}&
&
\circ\ar@{-->}@[red][rd]^{0}\ar@[green][ur]^{2}&
&
\bullet\ar@{-->}@[red][rd]^{0}\ar@[blue][ur]^{1}\\
f_1\ar@{-->}@[red][rd]^{0}\ar@[green][ur]^{2}&
&
\bullet
\ar@{-->}@[red][rd]^{0}\ar@[green][ur]^{2}&
&
\circ\ar@[blue][ur]^{1}&
&
\circ\ar@{-->}@[red][ru]^{0}\ar@[green][dr]^{2}&
&
\bullet\ar@{-->}@[red][ru]^{0}\ar@[green][dr]^{2}&
&
\circ\ar@[blue][dr]^{1}&
&
\circ\ar@{-->}@[red][rd]^{0}\ar@[green][ur]^{2}&
&
\bullet\ar@{-->}@[red][rd]^{0}\ar@[green][ur]^{2}&
&
\circ&\hspace{-10pt}\cdots\\
&
\bullet\ar@{-->}@[red][rd]^{0}\ar@[green][ur]^{2}&
&
\circ\ar@[green][ur]^{2}\ar@[blue][dr]^{1}&
&
\circ\ar@[green][dr]^{2}&
&
\circ\ar@{-->}@[red][ru]^{0}\ar@[green][dr]^{2}&
&
\bullet\ar@{-->}@[red][ru]^{0}\ar@[blue][dr]^{1}&
&
\circ\ar@[blue][ur]^{1}\ar@{-->}@[red][rd]^{0}&
&
\bullet\ar@{-->}@[red][rd]^{0}\ar@[green][ur]^{2}&
&
\circ\ar@[green][ur]^{2}\ar@[blue][dr]^{1}\\
f_0\ar@[blue][ur]^{1}&
&
\circ\ar@[green][ur]^{2}&
&
\circ\ar@[green][ur]^{2}&
&
\circ\ar@[blue][ur]^{1}&
&
\circ\ar@{-->}@[red][ru]^{0}&
&
\bullet\ar@{-->}@[red][ru]^{0}&
&
\bullet\ar@[blue][ur]^{1}&
&
\circ\ar@[green][ur]^{2}&
&
\circ&\hspace{-10pt}\cdots
}$$\caption{Arrangement of negative root vectors  for $D_3^{(2)}$}
\label{F: arrangement of negative root vectors for D3(2)}
\end{figure} 
\begin{proposition}
The {\em array   of  negative root vectors for $D_{l+1}^{(2)}$}, denoted by $\mathcal{B}^{-}_{D_{l+1}^{(2)}}$, is a colored directed graph.
Its  nodes, which represent the basis vectors of $\mathfrak g(D_{l+1}^{(2)})^{-}$, are organized into $2l+1$ rows and two sequences of diagonals with $2l+1$ (or fewer) nodes. Its arrows  indicate the action of the Chevalley generators $f_0,f_1,\ldots,f_{l}$ on the negative root subspaces and they are colored by  $0,1,\ldots ,l$, respectively. Removing the arrows of color $0$, which correspond to the action of  $f_0=1\otimes t^{-1}$, the graph decomposes into an infinite union of connected subgraphs. We shall refer to these subgraphs as {\em triangles} and {\em diagonals}. The first triangle, positioned in the upper left corner,  corresponds to the root vector basis  of the  $l^2$-dimensional Lie subalgebra $\mathfrak{n}_{-}$ of the simple Lie algebra $\mathfrak{g}(B_l)$ of type $B_l$.  
The remaining triangles
are crystal graphs ${\mathcal B}_{B_l}(\theta)$ of the adjoint representation of ${\mathfrak g}(B_l)$ given over $L_{B_l}(\theta) \otimes t^{2i}$ with $i< 0$.
Finally,  the diagonals, which consist of $2l+1$ nodes, are crystal graphs 
${\mathcal B}_{B_l}(\omega_1)$ of the vector representation $L_{B_{l}}(\omega_1)$ given over $\mathbb{C}^{2l+1} \otimes t^{2i+1}$   with $i <0$. 
The weights of nodes in $\mathcal{B}^{-}_{D_{l+1}^{(2)}}$ are periodic with  period $2\delta$. We place the Chevalley generator $f_0$ in the bottom row.

The main property of  $\mathcal{B}^{-}_{D_{l+1}^{(2)}}$ is that the weights of the corresponding nodes on two adjacent diagonals differ by $-\alpha_i$ if there is an $i$-arrow between these two diagonals. The sequence of arrows between diagonals is determined by the sequence of arrows in ${\mathcal B}_{B_l}(\omega_1)$. 
\end{proposition}
\begin{example}
The array $\mathcal{B}^{-}_{D_{3}^{(2)}}$ is given by   Figure \ref{AD3}.  
The triangle with vertices $0\bar{2}_{0}$, $2\bar{1}_{0}$, $\bar{2}\bar{1}_{0}$  corresponds to the basis of the $4$-dimensional Lie subalgebra $\mathfrak{n}_{-}$ of the simple Lie algebra  of type $B_2$.  The remaining triangles $\mathcal B_{B_2}(\theta)\otimes t^{2i}$, $i<0$, possess catheti consisting of two $2$-arrows and one $1$-arrow, which represent the action of  $f_2$ and $f_1$, respectively. Finally, the diagonals $\mathcal B_{B_2}(\omega_1)\otimes t^{2i+1}$, $i<0$, consist of two $1$-arrows, at the beginning and at the  end, and two $2$-arrows in the middle.
The triangles and the diagonals are connected by  the action of $f_0 $ which is indicated by the dashed $0$-arrows.

As noted before, for our purposes it is not necessary to precisely specify the elements $b_i\in {\mathcal B}_{D_{l+1}^{(2)}}^-$; all the information needed is encoded in arrows and ``positions of triangles $B(\theta)\otimes t^{2j}$ and lines $B(\omega_1)\otimes t^{2j+1}$''  like in Figure \ref{F: arrangement of negative root vectors for D3(2)}.
\end{example}

\subsection{The array of negative root vectors for $A^{(2)}_{2l}$}

Since the crystal $\mathcal B_{B_l}(\lambda)$ parametrizes a weight basis of $\mathfrak g(B_l)$-module $L_{B_l}(\lambda)$, from  \eqref{E: affine g(A2l(2))}  we get (a parametrization of) a weight basis of ${\mathfrak g}(A^{(2)}_{2l})$:
\begin{equation}\label{E: a basis of g(Dl+1(2))}
\{c, d\}\cup {\mathcal B}_{A^{(2)}_{2l}}, \quad\text{where}\quad
{\mathcal B}_{A^{(2)}_{2l}} =\bigcup_{j\in\mathbb Z}  {\mathcal B}_{B_l}(\theta)\otimes t^{2j} \ \cup \
\bigcup_{j\in\mathbb Z}  {\mathcal B}_{B_l}(2\omega_1)\otimes t^{2j+1}.
\end{equation}
For $i\in\mathbb Z$, set $b_{2i}=b\otimes t^{2i}$ for $b\in {\mathcal B}_{B_l}(\theta)$ and 
$b_{2i+1}=b\otimes t^{2i+1}$ for $b\in {\mathcal B}_{B_l}(2\omega_1)$. Note that these are weight vectors in the affine Lie algebra ${\mathfrak g}(A^{(2)}_{2l})$ of weight 
$$
\wt(b\otimes t^j)=\wt_\mathfrak h (b)+j\delta.
$$
We can glue the catheti of the small triangles ${\mathcal B}_{B_l}(\theta)\otimes t^{2i}$ to the catheti of big triangles ${\mathcal B}_{B_l}(2\omega_1)\otimes t^{2i\pm1}$ in ${\mathcal B}_{A^{(2)}_{2l}}$ along the congruent arrows (see the small and big triangles in Figures \ref{F: crystal of the tensor square of vector representation for B2} and \ref{F: crystal of the tensor square of vector representation for B3} for $l=2$ and $l=3$); for $l=2$, this is shown in Figure \ref{F: gluing small and big triangles in for A4(2)}, and, in general, in Figure \ref{F: gluing catheti of small and big triangles in A2l(2)}. The dashed $0$-arrows represent the action of the Chevalley generator $f_0$ on the array of negative root vectors ${\mathcal B}_{A^{(2)}_{2l}}$, which changes weights of the corresponding vectors by            
$$
-\alpha_0=2\omega_1-\delta=2\epsilon_1-\delta.
$$
We shall usually write the array of negative root vectors ${\mathcal B}_{A^{(2)}_{2l}}^-$ rotated by $\pi/4$, starting with $f_0=0\bar{l}\otimes t^{0}$, $f_1=l\overline{l-1}\otimes  t^0$, \dots,  $f_{l-1}=2\bar{1}\otimes  t^0$ and $f_l= 11\otimes t^{-1}$.
\begin{figure}[h!]
$$\xymatrix@=1.1em {
12_{2j}\ar@[green][d]^{2} &
&&&& 
&&&\\
10_{2j}\ar@[blue][r]^{1}\ar@[green][d]^{2} &
20_{2j}\ar@[green][d]^{2}&
&
&& 
&&&\\
1\bar{2}_{2j}\ar@[blue][d]^{1} &
00_{2j}\ar@[green][r]^{2}& 
0\bar{2}_{2j}\ar@[blue][d]^{1} &
&& 
&&&\\
2\bar{2}_{2j}\ar@[blue][r]^{1}\ar@{-->}@[red][d]^{0} &
2\bar{1}_{2j}\ar@[green][r]^{2} \ar@{-->}@[red][d]^{0}& 
0\bar{1}_{2j}\ar@[green][r]^{2} \ar@{-->}@[red][d]^{0}&
\bar{2}\bar{1}_{2j} \ar@{-->}@[red][d]^{0}&& 
&&&\\
11_{2j-1}\ar@[blue][r]^{1} &
21_{2j-1}\ar@[blue][d]^{1}\ar@[green][r]^{2} & 
01_{2j-1}\ar@[green][r]^{2}\ar@[blue][d]^{1} &
\bar{2}1_{2j-1}\ar@[blue][r]^{1} &
\bar{1}1_{2j-1}\ar@[blue][d]^{1}&&&&
\\ 
&
22_{2j-1}\ar@[green][r]^{2} & 
02_{2j-1}\ar@[green][r]^{2} &
\bar{2}2_{2j-1}\ar@[green][d]^{2} &
\bar{1}2_{2j-1}\ar@[green][d]^{2}\ar@{-->}@[red][r]^{0}&
12_{2j-2}\ar@[green][d]^{2} &&& 
\\
&&&
\bar{2}0_{2j-1}\ar@[blue][r]^{1}\ar@[green][d]^{2}&
\bar{1}0_{2j-1}\ar@[green][d]^{2}\ar@{-->}@[red][r]^{0}&
10_{2j-2}\ar@[blue][r]^{1}\ar@[green][d]^{2} &
20_{2j-2}\ar@[green][d]^{2}&&
\\
&& &
\bar{2}\bar{2}_{2j-1}\ar@[blue][r]^{1} &
\bar{1}\bar{2}_{2j-1}\ar@[blue][d]^{1}\ar@{-->}@[red][r]^{0}&
1\bar{2}_{2j-2}\ar@[blue][d]^{1} &
00_{2j-2}\ar@[green][r]^{2}& 
0\bar{2}_{2j-2}\ar@[blue][d]^{1} & 
\\
&&&&
\bar{1}\bar{1}_{2j-1} \ar@{-->}@[red][r]^{0}& 
2\bar{2}_{2j-2}\ar@[blue][r]^{1} &
2\bar{1}_{2j-2}\ar@[green][r]^{2} & 
0\bar{1}_{2j-2}\ar@[green][r]^{2} &
\bar{2}\bar{1}_{2j-2}      
}$$\caption{Gluing small and big triangles in ${\mathcal B}_{A^{(2)}_{4}}$}
\label{F: gluing small and big triangles in for A4(2)}
\end{figure}
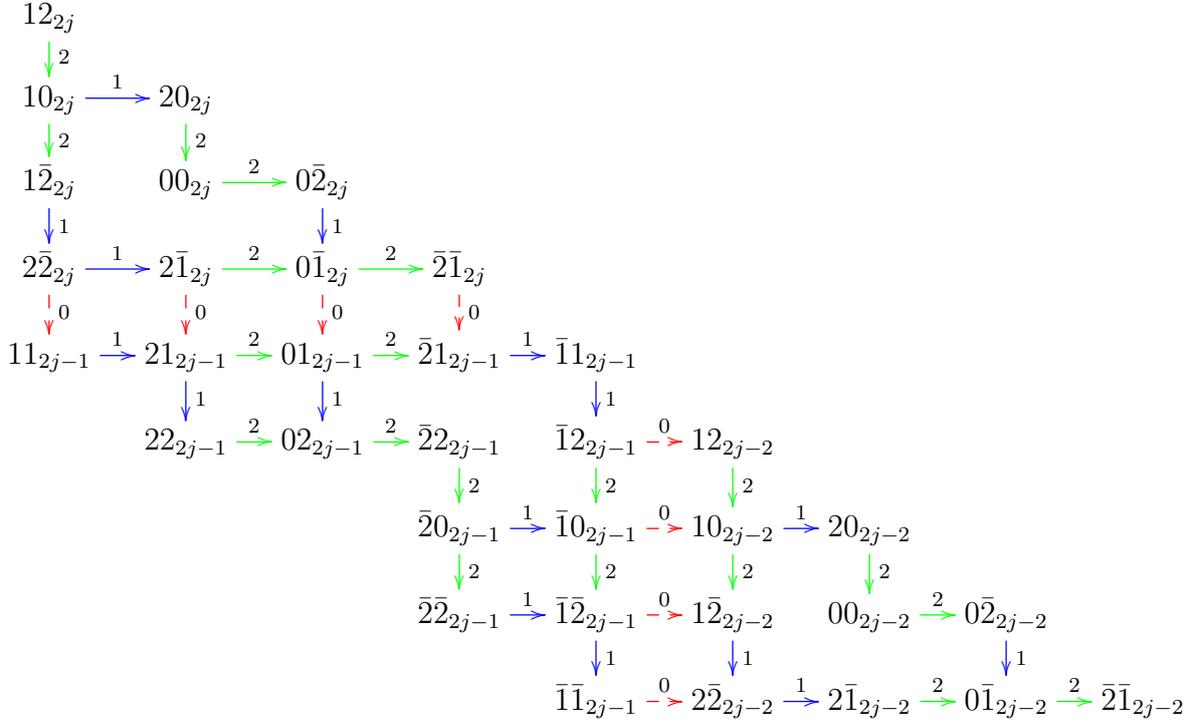
\begin{figure}[h!]
$$
\xymatrix@=1.1em  {
2\bar{2}_{2j} \ar@[black][r]^{1} \ar@{-->}@[red][d]^{0} & 
2\bar{1}_{2j} \ar@[black][r]^{2} \ar@{-->}@[red][d]^{0} & 
&
\cdots \ar@[black][r]^{l-1} & 
l\bar{1}_{2j} \ar@[black][r]^{l} \ar@{-->}@[red][d]^{0} & 
0\bar{1}_{2j} \ar@[black][r]^{l}  \ar@{-->}@[red][d]^{0}& 
\bar{l}\bar{1}_{2j} \ar@[black][r]^{l-1} \ar@{-->}@[red][d]^{0} & 
&  
\cdots \ar@[black][r]^{2}  & 
\bar{2}\bar{1}_{2j} \ar@{-->}@[red][d]^{0} & 
  \\
11_{2j-1} \ar@[black][r]^{1} & 
21_{2j-1} \ar@[black][r]^{2} & 
&
\cdots \ar@[black][r]^{l-1} & 
l1_{2j-1} \ar@[black][r]^{l} & 
01_{2j-1} \ar@[black][r]^{l} & 
\bar{l}1_{2j} \ar@[black][r]^{l-1} & 
&  
\cdots \ar@[black][r]^{2} & 
\bar{2}1_{2j}\ar@[black][r]^{1} & 
\bar{1}1_{2j}  \\
&&&&&&&&&&\\
\bar{1}1_{2j-1} \ar@[black][r]^{1} & 
\bar{1}2_{2j-1} \ar@[black][r]^{2}\ar@{-->}@[red][d]^{0} & 
&
\cdots \ar@[black][r]^{l-1} & 
\bar{1}l_{2j-1} \ar@[black][r]^{l} \ar@{-->}@[red][d]^{0}& 
\bar{1}0_{2j-1} \ar@[black][r]^{l} \ar@{-->}@[red][d]^{0}& 
\bar{1}\bar{l}_{2j} \ar@[black][r]^{l-1} \ar@{-->}@[red][d]^{0}& 
&  
\cdots \ar@[black][r]^{2} & 
\bar{1}\bar{2}_{2j}\ar@[black][r]^{1} \ar@{-->}@[red][d]^{0}& 
\bar{1}\bar{1}_{2j} \ar@{-->}@[red][d]^{0} \\
 & 
12_{2j-2} \ar@[black][r]^{2} & 
&
\cdots \ar@[black][r]^{l-1} & 
1l_{2j-2} \ar@[black][r]^{l} & 
10_{2j-2} \ar@[black][r]^{l} & 
1\bar{l}_{2j-2} \ar@[black][r]^{l-1} & 
&  
\cdots \ar@[black][r]^{2} & 
1\bar{2}_{2j-2}\ar@[black][r]^{1} & 
2\bar{2}_{2j-2}  
}$$\caption{Gluing catheti of small and big triangles in ${\mathcal B}_{A^{(2)}_{2l}}$}
\label{F: gluing catheti of small and big triangles in A2l(2)}
\end{figure}
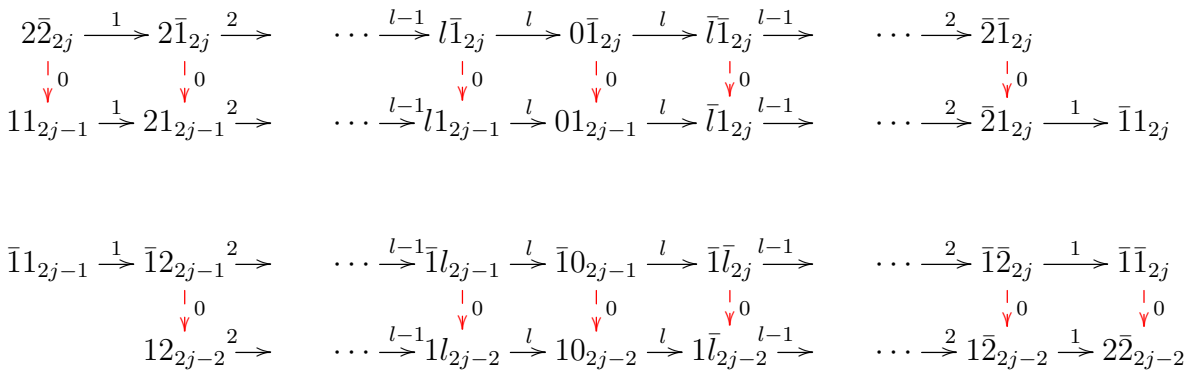
\begin{proposition}
The {\em array   of  negative root vectors for $ A^{(2)}_{2l} $}, denoted by $\mathcal{B}^{-}_{A^{(2)}_{2l}}$ is a colored directed graph.
Its  nodes, which represent the basis vectors of $\mathfrak g(A^{(2)}_{2l})^{-}$, are organized into $2l+1$ rows and two sequences of diagonals with $2l+1$ (or fewer) nodes. 
Its arrows  indicate the action of the Chevalley generators $f_0,f_1,\ldots,f_{l}$ on the negative root subspaces and they are colored by  $0,1,\ldots ,l$, respectively. Removing the arrows of color $l$, which correspond to the action of  $f_l=11\otimes t^{-1}$, the graph decomposes into an infinite union of connected subgraphs which we refer to  as {\em triangles}. The first triangle, positioned in the upper left corner,  corresponds to the root vector basis  of the  $l^2$-dimensional Lie subalgebra $\mathfrak{n}_{-}$ of the simple Lie algebra $\mathfrak{g}(B_l)$ of type $B_l$.  
The remaining triangles alternate between the crystals $\mathcal B_{B_l}(2\omega_1)$ of  $\mathfrak{g}(B_l)$ of the representation given over $L_{B_1}(2\omega_1) \otimes t^{2i+1}$ with $i< 0$, and the crystals $\mathcal B_{B_l}(\theta)$ of the adjoint representation given over $L_{B_1}(\theta) \otimes t^{2i}$ with $i< 0$. 
We shall refer to these subgraphs as {\em big triangles} and {\em small triangles}. The hypotenuses of the big triangles miss the midpoints. The weights of nodes in $\mathcal{B}^{-}_{A^{(2)}_{2l}}$ are periodic with  period $2\delta$. 
We place the hypotenuses of the big triangles on the bottom row; hence the Chevalley generator $f_0$ is in the top row.

The main property of  $\mathcal{B}^{-}_{A^{(2)}_{2l}}$ is that the weights of the corresponding nodes on two adjacent diagonals differ by $-\alpha_i$ if there is an $i$-arrow between these two diagonals. The sequence of arrows between diagonals is determined by the sequence of arrows in ${\mathcal B}_{B_l}(\omega_1)$.
\end{proposition}
\begin{example}
The array $\mathcal{B}^{-}_{A^{(2)}_{4}}$ is given by Figure \ref{AA2}, and by Figure \ref{F: arrangement of negative root vectors  for A4(2)} without specifying vectors $b_i$.
\end{example}
\begin{figure}[h!]
$$\xymatrix@=0.3em {
 0\bar{2}_{0}\ar@[blue][dr]^{1}&
&
\bar{2}\bar{1}_{0}\ar@{-->}@[red][rd]^{2}&
&
\bar{1}1_{-1}\ar@[blue][dr]^{1}&
&
12_{-2}\ar@[green][dr]^{0}&
&
20_{-2}\ar@[green][dr]^{0}&
&
0\bar{2}_{-2}\ar@[blue][dr]^{1}&
&
\bar{2}\bar{1}_{-2}&
\hspace{-20pt}\ \cdots \\
&
0\bar{1}_{0}\ar@{-->}@[red][rd]^{2}\ar@[green][ur]^{0}&
&
\bar{2}1_{-1}\ar@[blue][ur]^{1}&
&
\bar{1}2_{-1}\ar@[green][dr]^{0}\ar@{-->}@[red][ru]^{2}&
&
10_{-2}\ar@[green][dr]^{0}\ar@[blue][ur]^{1}&
&
00_{-2}\ar@[green][ur]^{0}&
&
0\bar{1}_{-2}\ar@[green][ur]^{0}\ar@{-->}@[red][rd]^{2}&\\
2\bar{1}_{0}\ar@{-->}@[red][rd]^{2}\ar@[green][ur]^{0}&
&
01_{-1}\ar@[green][ur]^{0}\ar@[blue][dr]^{1}&
&
\bar{2}2_{-1}\ar@[green][dr]^{0}&
&
\bar{1}0_{-1}\ar@{-->}@[red][ru]^{2}\ar@[green][dr]^{0}&
&
1\bar{2}_{-2}\ar@[blue][dr]^{1}&
&
2\bar{1}_{-2}\ar@{-->}@[red][rd]^{2}\ar@[green][ur]^{0}&
&
01_{-3}&
\hspace{-20pt}\cdots \\
&
21_{-1}\ar@[green][ur]^{0}\ar@[blue][dr]^{1}&
&
02_{-1}\ar@[green][ur]^{0}&
&
\bar{2}0_{-1}\ar@[green][dr]^{0}\ar@[blue][ur]^{1}&
&
\bar{1}\bar{2}_{-1}\ar@{-->}@[red][ru]^{2}\ar@[blue][dr]^{1}&
&
2\bar{2}_{-2}\ar@[blue][ur]^{1}\ar@{-->}@[red][rd]^{2}&
&
21_{-3}\ar@[green][ur]^{0}\ar@[blue][dr]^{1}\\
11_{-1}\ar@[blue][ur]^{1}&
&
22_{-1}\ar@[green][ur]^{0}&
&
&
&
\bar{2}\bar{2}_{-1}\ar@[blue][ur]^{1}&
&
\bar{1}\bar{1}_{-1}\ar@{-->}@[red][ru]^{2}&
&
11_{-3}\ar@[blue][ur]^{1}&
&
22_{-3}&
\hspace{-20pt}\cdots 
}$$\caption{Arrangement of negative root vectors with elements $b_i$ for $A_4^{(2)}$}\label{AA2}
\end{figure}
\begin{figure}[h!]
$$\xymatrix@=1.6em {
 f_0\ar@[blue][dr]^{1}&
&
\circ\ar@{-->}@[red][rd]^{2}&
&
\bullet\ar@[blue][dr]^{1}&
&
\circ\ar@[green][dr]^{0}&
&
\circ\ar@[green][dr]^{0}&
&
\circ\ar@[blue][dr]^{1}&
&
\circ&
\hspace{-20pt}\ \cdots \\
&
\circ\ar@{-->}@[red][rd]^{2}\ar@[green][ur]^{0}&
&
\bullet\ar@[blue][ur]^{1}&
&
\bullet\ar@[green][dr]^{0}\ar@{-->}@[red][ru]^{2}&
&
\circ\ar@[green][dr]^{0}\ar@[blue][ur]^{1}&
&
\circ\ar@[green][ur]^{0}&
&
\circ\ar@[green][ur]^{0}\ar@{-->}@[red][rd]^{2}&\\
f_1\ar@{-->}@[red][rd]^{2}\ar@[green][ur]^{0}&
&
\bullet\ar@[green][ur]^{0}\ar@[blue][dr]^{1}&
&
\bullet\ar@[green][dr]^{0}&
&
\bullet\ar@{-->}@[red][ru]^{2}\ar@[green][dr]^{0}&
&
\circ\ar@[blue][dr]^{1}&
&
\circ\ar@{-->}@[red][rd]^{2}\ar@[green][ur]^{0}&
&
\bullet&
\hspace{-20pt}\cdots \\
&
\bullet\ar@[green][ur]^{0}\ar@[blue][dr]^{1}&
&
\bullet\ar@[green][ur]^{0}&
&
\bullet\ar@[green][dr]^{0}\ar@[blue][ur]^{1}&
&
\bullet\ar@{-->}@[red][ru]^{2}\ar@[blue][dr]^{1}&
&
\circ\ar@[blue][ur]^{1}\ar@{-->}@[red][rd]^{2}&
&
\bullet\ar@[green][ur]^{0}\ar@[blue][dr]^{1}\\
f_2\ar@[blue][ur]^{1}&
&
\bullet\ar@[green][ur]^{0}&
&
&
&
\bullet\ar@[blue][ur]^{1}&
&
\bullet\ar@{-->}@[red][ru]^{2}&
&
\bullet\ar@[blue][ur]^{1}&
&
\bullet&
\hspace{-20pt}\cdots 
}$$\caption{Arrangement of negative root vectors  for $A_4^{(2)}$}
\label{F: arrangement of negative root vectors  for A4(2)}
\end{figure}

\section{Specialized arrays of negative root vectors}\label{section_06}

\subsection{Specialized arrays of negative root vectors for $ C_l^{(1)}$}

Let $s_0, s_1,\ldots, s_l \in \mathbb{N}$ be fixed. By assigning  degrees
$$  \deg f_0=s_0, \  \deg f_1=s_1, \ \ldots, \  \deg f_l=s_l$$
to the arrows in the array $\mathcal{B}^{-}_{C_l^{(1)}}$ we  obtain an array $\mathcal{N}_{C_l^{(1)}}^{(s_0,s_1,\ldots, s_l)}$ of positive integers, which we call the {\em $(s_0,s_1,\ldots, s_l)$-specialized weighted array for $C_l^{(1)}$}. The array $\mathcal{N}_{C_l^{(1)}}^{(s_0,s_1,\ldots, s_l)}$ is periodic with period $s_0+2\sum_{i=1}^{l-1}s_i+ s_l$.

For $s=(1,\dots,1)$ we have the so called {\it principal specialization}. In the principally specialized array of negative root vectors the labels of nodes  increase  by one  when we move one place to the upper right or lower right. Hence we have the following lemma:
\begin{lemma}\label{L: Q(1...1;C)}
For $l \geq 1$, the array $\mathcal{N}_{C_l^{(1)}}^{(1,\ldots,1)}$ consists of $l$ copies of positive integers and one copy of odd positive integers. Equivalently, 
$$
Q(1,\ldots,1; C_l^{(1)})=\prod_{i\in\mathbb N}(1-q^i)^ l\prod_{j\in\mathbb N}(1-q^{2j-1}).
$$
\end{lemma}
\begin{example}\label{eC1...1}
Figure \ref{fig:arrayC1} represents the array $\mathcal{N}_{C_4^{(1)}}^{(1,1,1,1,1)}$.  Note that   the labels corresponding to the different adjoint triangles    are written in italic and bold, so that they can be immediately detected in the image. 
\end{example}

\begin{figure}[htb!]
$$
\begin{array}{cccccccccccccccc}
{\textit{1} }& & {\textit{3}}& & {\textit{5}}& & {\textit{7}}& & \textit{9} & &\textit{11} &  &\textit{13} & &\textit{15} &  \\
&{\textit{2}} & &{\textit{4}} & &{\textit{6}} & &\textbf{8}& &\textit{10} & & \textit{12} &  &\textit{14} & &\textit{16}\\
{\textit{1} } & &{\textit{3}} & &{\textit{5}}& &\textbf{7} & &\textbf{9} &  &\textit{11} & &\textit{13} &  &\textit{15} &  \\
&{\textit{2}} & &{\textit{4}}  & &\textbf{6}& &\textbf{8} & &  \textbf{10}& &\textit{12} &  &\textit{14}&  &\textit{16} \\
{\textit{1} } & &{\textit{3}} & &\textbf{5} & &\textbf{7} & &\textbf{9} & & \textbf{11}& &\textit{13} &  &\textit{15} &  \\
&{\textit{2}}  & &\textbf{4}& &\textbf{6} & &\textbf{8} & &\textbf{10} & & \textbf{12}& &\textit{14} &  &\textit{16}\\
{\textit{1} }& &\textbf{3}& &\textbf{5} & &\textbf{7} & &\textbf{9} & &\textbf{11} & &  \textbf{13}& &\textit{15} &\\
&\textbf{2} & &\textbf{4} & &\textbf{6} & &\textbf{8} & &\textbf{10} & &\textbf{12} & & \textbf{14}& &\textit{16}\\
\textbf{1 }& &\textbf{3} & &\textbf{5} & &\textbf{7} & &\textbf{9} & &\textbf{11} & & \textbf{13} & &\textbf{15}&\\
\end{array}\quad\dots
$$
\caption{Array $\mathcal{N}_{C_4^{(1)}}^{(1,1,1,1,1)}$}
  \label{fig:arrayC1}
\end{figure}
\begin{example}\label{e21..12}
The array $\mathcal{N}_{C_4^{(1)}}^{(2,1,1,1,2)}$ consists of three copies of positive integers, one copy of even positive integers and one copy of positive integers congruent to $5$ modulo $10$. We shall now demonstrate how such a conclusion can be easily obtained   examining the  array in the case $l=4$.  First, note that the array $\mathcal{N}_{C_4^{(1)}}^{(2,1,1,1,2)}$ is periodic with period $10$, so it is enough to consider the first two triangles corresponding to $\mathfrak n_-\otimes t^0$ and $\mathfrak n_+\otimes t^{-1}$ (see Figure \ref{fig:arrayC2}). Then, as in Figure \ref{fig:arrayC1}, we divide the array  $\mathcal{N}_{C_4^{(1)}}^{(2,1,1,1,2)}$ into triangles; see Figure \ref{fig:arrayC2}. Next, suppose that all   triangles with italic labels slide one place down the catheti of the adjacent   triangles with  bold labels; see Figure \ref{fig:arrayC2pomaknut}.
\noindent  Finally, in Figure \ref{fig:arrayC2pomaknut}, we notice three copies of the positive integers (we indicate their positions in Figure \ref{fig:arrayC2pomaknut2} by the symbols  $    \textcolor{red}{\tcircled{1}},\textcolor{green}{\tcircled{2}},\textcolor{blue}{\tcircled{3}}$), one copy of the even positive  integers (indicated by $\textcolor{cyan}{\tcircled{E}}$ in Figure \ref{fig:arrayC2pomaknut2}) and one copy  of the positive integers congruent to $5$ modulo $10$ (indicated by $\textcolor{brown}{\tcircled{5}}$ in Figure \ref{fig:arrayC2pomaknut2}).
\end{example}

\begin{figure}[htb!]
\begin{minipage}{0.48\textwidth}
     \centering
$$
\begin{array}{cccccccccccccccc}
{\textit{2} }& & {\textit{4}}& & {\textit{6}}& & {\textit{8}} \\
&{\textit{3}} & &{\textit{5}} & &{\textit{7}} & &\textbf{10} \\
{\textit{1} } & &{\textit{4}} & &{\textit{6}}& &\textbf{9} \\
&{\textit{2}} & &{\textit{5}}  & &\textbf{8}& &\textbf{10}  \\
{\textit{1} } & &{\textit{3}} & &\textbf{7} & &\textbf{9}   \\
&{\textit{2}}  & &\textbf{5}& &\textbf{8} & &\textbf{10}  \\
{\textit{1} }& &\textbf{4}& &\textbf{6} & &\textbf{9}  \\
&\textbf{3} & &\textbf{5} & &\textbf{7} & &\textbf{10} \\
\textbf{2 }& &\textbf{4} & &\textbf{6} & &\textbf{8}   
\end{array}\quad\dots
$$
\caption{Array $\mathcal{N}_{C_4^{(1)}}^{(2,1,1,1,2)}$}
  \label{fig:arrayC2}
\end{minipage}\hfill
\begin{minipage}{0.48\textwidth}
     \centering
$$
\begin{array}{ccccccccccccccccccc}
{\textit{2} } & & {\textit{4}} & &{\textit{6}} & &{\textit{8}} & &\textbf{10}\\
& {\textit{3} } & &{\textit{5}} & &{\textit{7}}& &\textbf{9} \\
{\textit{1}} & &{\textit{4}}  & & {\textit{6}}  & &\textbf{8}& &\textbf{10}\\
& {\textit{2} } & &{\textit{5}} & &\textbf{7} & &\textbf{9} \\
{\textit{1}}  & & {\textit{3}}  & &\textbf{5}& &\textbf{8} & &\textbf{10}\\
& {\textit{2} }& &\textbf{4}& &\textbf{6} & &\textbf{9} \\
{\textit{1} }& &\textbf{3} & &\textbf{5} & &\textbf{7} & &\textbf{10} \\
& \textbf{2 }& &\textbf{4} & &\textbf{6} & &\textbf{8}  
\end{array}\quad\dots
$$
\caption{Translated triangles in the array $\mathcal{N}_{C_4^{(1)}}^{(2,1,1,1,2)}$}
  \label{fig:arrayC2pomaknut}
	\end{minipage}
\end{figure}

\begin{figure}[htb!]
$$
\begin{array}{ccccccccccc}
\textcolor{cyan}{\circled{E}} & & \textcolor{cyan}{\circled{E}} & &\textcolor{cyan}{\circled{E}} & &\textcolor{cyan}{\circled{E}}& &  \textcolor{cyan}{\circled{E}}\\
& \textcolor{blue}{\circled{3}} & &     \textcolor{green}{\circled{2}} & &     \textcolor{red}{\circled{1}} & &\textcolor{blue}{\circled{3}} &  \\
 \textcolor{red}{\circled{1}}  & &\textcolor{blue}{\circled{3}} & &     \textcolor{green}{\circled{2}}  & &        \textcolor{red}{\circled{1}} & & \textcolor{blue}{\circled{3}}   \\     
& \textcolor{red}{\circled{1}}  & &\textcolor{blue}{\circled{3}} & &      \textcolor{green}{\circled{2}}  & &       \textcolor{red}{\circled{1}}  & \\
 \textcolor{green}{\circled{2}}  & &      \textcolor{red}{\circled{1}}  & &\textcolor{brown}{\circled{5}}& &       \textcolor{green}{\circled{2}}  & &  \textcolor{red}{\circled{1}}\\ 
&      \textcolor{green}{\circled{2}} & &       \textcolor{red}{\circled{1}} & &\textcolor{blue}{\circled{3}} & &      \textcolor{green}{\circled{2}}  & \\
\textcolor{blue}{\circled{3}}& &     \textcolor{green}{\circled{2}}  & &  \textcolor{red}{\circled{1}}  & &\textcolor{blue}{\circled{3}} & &     \textcolor{green}{\circled{2}}  \\
& \textcolor{blue}{\circled{3}}& &      \textcolor{green}{\circled{2}}  & &      \textcolor{red}{\circled{1}}  & & \textcolor{blue}{\circled{3}} & 
\end{array}\quad\dots
$$
\caption{Three copies of the positive  integers ($\textcolor{red}{\tcircled{1}},\textcolor{green}{\tcircled{2}},\textcolor{blue}{\tcircled{3}}$), one  copy of the even positive integers ($\textcolor{cyan}{\tcircled{E}}$)  and one copy  of the positive  integers congruent to $5$ modulo $10$ ($\textcolor{brown}{\tcircled{5}}$) in the array $\mathcal{N}_{C_4^{(1)}}^{(2,1,1,1,2)}$}
 \label{fig:arrayC2pomaknut2}
\end{figure}
The previous argument can be generalized to the case of  arbitrary $l\geq3$:
\begin{lemma}\label{L: Q(21...1;C)}
The array $\mathcal{N}_{C_l^{(1)}}^{(2,1,\ldots ,1,2)}$ consists of $l-1$ copies of the positive integers, one copy of the even positive integers and one copy of the positive integers congruent to  $l+1$ modulo $2(l+1)$. Equivalently, 
$$
Q(2,1,\ldots,1,2; C_l^{(1)})=\prod_{i \in \mathbb{N}}(1-q^i)^{l-1}\prod_{j \in \mathbb{N}}(1-q^{2j})\prod_{r\equiv  (l+1) \mod 2(l+1)}(1-q^r).
$$
\end{lemma}

\begin{lemma}\label{L: Q(21...1;C)}
The array $\mathcal{N}_{C_l^{(1)}}^{(2,1,\ldots, 1)}$ consists of $l$ copies of the positive integers. Equivalently, 
$$
Q(2,1,\ldots,1; C_l^{(1)})=\prod_{i\in\mathbb N}(1-q^i)^ l.
$$
\end{lemma}
\begin{proof}
In each triangle the labels of nodes  increase  by one  when we move
one place to the upper right or lower right.  Moreover,  when we move one place from one triangle to  another one, the labels  increase by two. Figure \ref{fig:arrayC3} shows three adjacent triangles in the case   $l=4$. Now we can argue as before by sliding the upper triangles one place down the catheti of the adjacent lower triangles.
\end{proof}
\begin{remark}
It is worth noting that the array $\mathcal{N}_{C_l^{(1)}}^{(1,1,\ldots,1, 2)}$ is the transpose of the array $\mathcal{N}_{C_l^{(1)}}^{(2,1,\ldots ,1,1)}$; see  Figure  \ref{fig:arrayC4corr} for $l=4$.
\end{remark}
\begin{figure}[htb!]
$$
\begin{array}{cccccccccccccccc}
{\textit{1} }& & {\textit{3}}& & {\textit{5}}& & {\textit{7}}& & \textit{11} & &\textit{13} &  &\textit{15} & &\textit{17} &  \\
&{\textit{2}} & &{\textit{4}} & &{\textit{6}} & &\textbf{9}& &\textit{12} & & \textit{14} &  &\textit{16} & &\textit{18}\\
{\textit{1} } & &{\textit{3}} & &{\textit{5}}& &\textbf{8} & &\textbf{10} &  &\textit{13} & &\textit{15} &  &\textit{17} &  \\
&{\textit{2}} & &{\textit{4}}  & &\textbf{7}& &\textbf{9} & &  \textbf{11}& &\textit{14} &  &\textit{16}&  &\textit{18} \\
{\textit{1} } & &{\textit{3}} & &\textbf{6} & &\textbf{8} & &\textbf{10} & & \textbf{12}& &\textit{15} &  &\textit{17} &  \\
&{\textit{2}}  & &\textbf{5}& &\textbf{7} & &\textbf{9} & &\textbf{11} & & \textbf{13}& &\textit{16} &  &\textit{18}\\
{\textit{1} }& &\textbf{4}& &\textbf{6} & &\textbf{8} & &\textbf{10} & &\textbf{12} & &  \textbf{14}& &\textit{17} &\\
&\textbf{3} & &\textbf{5} & &\textbf{7} & &\textbf{9} & &\textbf{11} & &\textbf{13} & & \textbf{15}& &\textit{18}\\
\textbf{2 }& &\textbf{4} & &\textbf{6} & &\textbf{8} & &\textbf{10} & &\textbf{12} & & \textbf{14} & &\textbf{16}&\\
\end{array}\quad\dots
$$
\caption{Array $\mathcal{N}_{C_4^{(1)}}^{(2,1,1,1,1)}$}
  \label{fig:arrayC3}
\end{figure}
\begin{figure}[htb!]
$$
\begin{array}{cccccccccccccccc}
\textit{2 }& &\textit{4} & &\textit{6} & &\textit{8} & &\textit{10} & &\textit{12} & & \textit{14} & &\textit{16}&\\
&\textit{3} & &\textit{5} & &\textit{7} & &\textbf{9} & &\textit{11} & &\textit{13} & & \textit{15}& &\textit{18}\\
{\textit{1} }& &\textit{4}& &\textit{6} & &\textbf{8} & &\textbf{10} & &\textit{12} & &  \textit{14}& &\textit{17} &\\
&{\textit{2}}  & &\textit{5}& &\textbf{7} & &\textbf{9} & &\textbf{11} & & \textit{13}& &\textit{16} &  &\textit{18}\\
{\textit{1} } & &{\textit{3}} & &\textbf{6} & &\textbf{8} & &\textbf{10} & & \textbf{12}& &\textit{15} &  &\textit{17} &  \\
&{\textit{2}} & &{\textbf{4}}  & &\textbf{7}& &\textbf{9} & &  \textbf{11}& &\textbf{14} &  &\textit{16}&  &\textit{18} \\
{\textit{1} } & &{\textbf{3}} & &{\textbf{5}}& &\textbf{8} & &\textbf{10} &  &\textbf{13} & &\textbf{15} &  &\textit{17} &  \\
&{\textbf{2}} & &{\textbf{4}} & &{\textbf{6}} & &\textbf{9}& &\textbf{12} & & \textbf{14} &  &\textbf{16} & &\textit{18}\\
{\textbf{1} }& & {\textbf{3}}& & {\textbf{5}}& & {\textbf{7}}& & \textbf{11} & &\textbf{13} &  &\textbf{15} & &\textbf{17} &  \\
\end{array}\quad\dots
$$
\caption{Array $\mathcal{N}_{C_4^{(1)}}^{(1,1,1,1,2)}$}
  \label{fig:arrayC4corr}
\end{figure}
\begin{example} The labels of  nodes in the specialized array $\mathcal{N}_{C_4^{(1)}}^{(4,3,2,3,4)}$ are the  positive integers congruent to
 \begin{align*}
0,0,0,0, \pm 2, \pm 3,\pm 3, \pm 4,\pm 4,\pm 5, \pm 5,\pm 7,\pm 7,\pm 8, \pm 9,  \pm 9, \pm 10, \pm 10, 
\nonumber  \pm 12,\pm 12
\end{align*} 
modulo $24$; see Figure \ref{fig:cong3corr}.
\end{example}

\begin{figure}[htb!]
$$
\begin{array}{ccccccccccccccccc}
{\textit{4} }& & {\textit{10}}& & {\textit{14}}& & {\textit{20}}& & {\textit{28}} & &{\textit{34}} &  &{\textit{38}} & &{\textit{44}} &  \\
&{\textit{7}} & &{\textit{12}} & &{\textit{17}} & & {\textbf{24}} & &{\textit{31}} & & {\textit{36}} &  &{\textit{41}} & &{\textit{48}}\\
{\textit{3} } & &{\textit{9}} & &{\textit{15}}& &\textbf{21} & &\textbf{27} &  &{\textit{33}} & &{\textit{39}} &  &{\textit{45}} &  \\
&{\textit{5}} & &{\textit{12}}  & &\textbf{19}& &\textbf{24} & & \textbf{29}& &{\textit{36}} &  &{\textit{43}}&  &{\textit{48}} \\
{\textit{2} } & &{\textit{8}} & &\textbf{16} & &\textbf{22} & &\textbf{26} & & \textbf{32}& &{\textit{40}} &  &{\textit{46}} &  \\
&{\textit{5}}  & &\textbf{12}& &\textbf{19} & &\textbf{24} & &\textbf{29} & &  \textbf{36}& &{\textit{43}} &  &{\textit{48}}\\
{\textit{3} }& &\textbf{9}& &\textbf{15} & &\textbf{21} & &\textbf{27} & &\textbf{33} & &  \textbf{39}& &{\textit{45}} &\\
&\textbf{7} & &\textbf{12} & &\textbf{17} & &\textbf{24} & &\textbf{31} & &\textbf{36} & & \textbf{41}& &{\textit{48}}\\
\textbf{4 }& &\textbf{10} & &\textbf{14} & &\textbf{20} & &\textbf{28} & &\textbf{34} & & \textbf{38} & &\textbf{44}&\\
\end{array}\quad\dots
$$
\caption{Array $\mathcal{N}_{C_4^{(1)}}^{(4,3,2,3,4)}$}
  \label{fig:cong3corr}
\end{figure}

\subsection{Specialized arrays of negative roots for $D_{l+1}^{(2)}$}\label{subsect_52}

As in the previous case, fix  $s_0, s_1,\ldots, s_l \in \mathbb{N}$ and then  assign the  degrees
$$  \deg f_0=s_0, \  \deg f_1=s_1, \ \ldots, \  \deg f_l=s_l$$
to arrows in the array $\mathcal{B}^{-}_{D_{l+1}^{(2)}}$. Thus, we obtain an array $\mathcal{N}_{D_{l+1}^{(2)}}^{(s_0,s_1,\ldots, s_l)}$ of positive integers, which we call the {\em $(s_0,s_1,\ldots, s_l)$-specialized weighted array for $D_{l+1}^{(2)}$}. The array $\mathcal{N}_{D_{l+1}^{(2)}}^{(s_0,s_1,\ldots, s_l)}$ is periodic with period $2(s_0+s_1+\ldots +s_l)$.

In the principally specialized array of negative root vectors the labels of nodes  increase  by one  when we move one place to the upper right or lower right.  Hence we have the following:
\begin{lemma}\label{L: Q(1...1;D)}
The array $\mathcal{N}_{D_{l+1}^{(2)}}^{(1,\ldots, 1)}$ consists of $l$ copies of the positive integers and one copy of the odd positive integers. Equivalently, 
$$
Q(1,\ldots,1; D_{l+1}^{(2)})=\prod_{i\in\mathbb N}(1-q^i)^ l\prod_{j\in\mathbb N}(1-q^{2j-1}).
$$
\end{lemma}
\begin{example}\label{eD1...1}
In  Figure \ref{fig:arrayD1} we have the array $\mathcal{N}_{D_5^{(2)}}^{(1,1,1,1,1)}$. Its triangles (resp.  diagonals) are indicated by labels in italic (resp. bold). 
\end{example}

\begin{figure}[htb!]
$$
\begin{array}{cccccccccccccccccccc}
{\textit{1} }& & {\textit{3}}& & {\textit{5}}& & {\textit{7}}& & \textbf{9} & &\textbf{11} &  &\textit{13} & &\textit{15} &&\textit{17} & &\textit{19} & \\
&{\textit{2}} & &{\textit{4}} & &{\textit{6}} & &\textbf{8}& &\textit{10} & & \textbf{12} &  &\textit{14} & &\textit{16}&&\textit{18} &&\textit{20} \\
{\textit{1} } & &{\textit{3}} & &{\textit{5}}& &\textbf{7} & &\textit{9} &  &\textit{11} & &\textbf{13} &  &\textit{15} &&\textit{17} & &\textit{19}  \\
&{\textit{2}} & &{\textit{4}}  & &\textbf{6}& &\textit{8} & &  \textit{10}& &\textit{12} &  &\textbf{14}&  &\textit{16}&& \textit{18} &&\textit{20}\\
{\textit{1} } & &{\textit{3}} & &\textbf{5} & &\textit{7} & &\textit{9} & & \textit{11}& &\textit{13} &  &\textbf{15} &&  \textit{17} & &\textit{19} \\
&{\textit{2}}  & &\textbf{4}& &\textit{6} & &\textit{8} & &\textit{10} & & \textit{12}& &\textit{14} &  &\textbf{16}&&\textit{18} &&\textit{20} \\
{\textit{1} }& &\textbf{3}& &\textit{5} & &\textit{7} & &\textit{9} & &\textit{11} & &  \textit{13}& &\textit{15} &&\textbf{17}&&\textit{19}&\\
&\textbf{2} & &\textit{4} & &\textit{6} & &\textit{8} & &\textit{10} & &\textit{12} & & \textit{14}& &\textit{16}&&\textbf{18}&&\textit{20}\\
\textbf{1 }& &\textit{3} & &\textit{5} & &\textit{7} & &\textit{9} & &\textit{11} & & \textit{13} & &\textit{15}&&\textit{17}&&\textbf{19}&\\
\end{array}\quad\dots
$$
\caption{Array $\mathcal{N}_{D_5^{(2)}}^{(1,1,1,1,1)}$}
  \label{fig:arrayD1}
\end{figure}

\subsection{Specialized arrays of negative root vectors for $A^{(2)}_{2l}$}

As before, for fixed integers  $s_0, s_1,\ldots, s_l \in \mathbb{N}$ we  assign the  degrees
$$  \deg f_0=s_0, \  \deg f_1=s_1, \ \ldots, \  \deg f_l=s_l$$
to the arrows in the array $\mathcal{B}^{-}_{A_{2l}^{(2)}}$, thus getting the array $\mathcal{N}_{A_{2l}^{(2)}}^{(s_0,s_1,\ldots, s_l)}$ of positive integers, which we call the {\em $(s_0,s_1,\ldots, s_l)$-specialized weighted array for $A_{2l}^{(2)}$}. It is periodic with period $2(s_l+2\sum_{i=0}^{l-1}s_i)$.
By using the orientation-reversing automorphism of the Dynkin diagram of type $A_{2l}^{(2)}$ that
sends the node $i$ to   $l- i$, we obtain a weighted crystal graph of type $\mathcal{B}^{-}_{{A^{(2)}_{2l}}^T}$ and $(s_0,s_1,\ldots, s_l)$-specialized weighted crystal of negative roots for ${A^{(2)}_{2l}}^T$.
\begin{lemma}\label{L: Q(1...1;a)} 
For $l \geq 1$, the array $\mathcal{N}_{A_{2l}^{(2)}}^{(1,\ldots,1)}$ consists of $l$ copies of the  positive integers and one copy of the odd positive integers which are not congruent to $2l+1$ modulo $ 2(2l+1)$. Equivalently, 
$$
Q(1,\ldots,1; A_{2l}^{(2)})=\prod_{i\in\mathbb N}(1-q^i)^ l\prod_{j\not\equiv l+1\mod ( 2l+1 )}(1-q^{2j-1}).
$$
\end{lemma}
\begin{example}\label{eA1...1}
In Figure \ref{fig:arrayA1}, we have the array $\mathcal{N}_{A_8^{(2)}}^{(1,1,1,1,1)}$.  Its triangles are again indicated by italic and bold labels.
\end{example}
\begin{figure}[htb!]
$$\begin{array}{ccccccccccccccccccc}
\textit{1}& & \textit{3}& & \textit{5}& & \textit{7}& & \textbf{9}&  & \textit{11} & &\textit{13} &  &\textit{15} & &\textit{17} &  \\
&\textit{2} & &\textit{4} & &\textit{6} & &\textbf{8}& & \textbf{10} & & \textit{12} &  &\textit{14} & &\textit{16} &  &\textit{18}\\
\textit{1 } & &\textit{3} & &\textit{5}& &\textbf{7} & &\textbf{9} & & \textbf{11}& &\textit{13} &  &\textit{15} & &\textit{17} &   \\
&\textit{2} & &\textit{4}  & &\textbf{6}& &\textbf{8}& &\textbf{10}& &  \textbf{12}& &\textit{14} &  &\textit{16}&  &\textit{18} \\
\textit{1 } & &\textit{3} & &\textbf{5} & &\textbf{7} & &\textbf{9}&  &\textbf{11} & & \textbf{13}& &\textit{15} &  &\textit{17} &  \\
&\textit{2}  & &\textbf{4}& &\textbf{6}& &\textbf{8} & &\textbf{10}&  &\textbf{12} & &  \textbf{14}& &\textit{16} &  &\textit{18}\\
\textit{1 }& &\textbf{3}& &\textbf{5}& &\textbf{7} & &\textbf{9}& & \textbf{11} & &\textbf{13} & &  \textbf{15}& &\textit{17} &\\
&\textbf{2} & &\textbf{4} & &\textbf{6} & &\textbf{8} & &\textbf{10}&  &\textbf{12} & & \textbf{14} & &  \textbf{16}& &\textit{18}\\
\textbf{1}& &\textbf{3} & &\textbf{5} & &\textbf{7}& & &  &\textbf{11} & &  \textbf{13} & &\textbf{15} & &  \textbf{17}&\\
\end{array}\quad\dots
$$
\caption{Array $\mathcal{N}_{A_{8}^{(2)}}^{(1,1,1,1,1)}$}
  \label{fig:arrayA1}
\end{figure}

\begin{example}
The labels of nodes in the specialized array $\mathcal{N}_{A_{8}^{(2)}}^{(3,2,2,3,4)}$ are the positive integers congruent to
 \begin{eqnarray}\nonumber
0,0,0,0, \pm 2,\pm 2, \pm 3,\pm 3,\pm4, \pm 5, \pm 5,\pm 7,\pm 7,\pm 7,\pm 8, \pm 9,\pm 10, \pm 10, \pm 11,  \pm 12 
\end{eqnarray} 
modulo 24 and
the positive integers congruent to
$
\pm 4,  \pm 10, \pm 14,\pm 18 
$
modulo 48; see Figure \ref{fig:cong6}.
\end{example}
\begin{figure}[htb!]
$$
\begin{array}{ccccccccccccccccccc}
\textit{3 }& & \textit{8}& & \textit{12}& & \textit{17}& & \textbf{24}&  & \textit{31} & &\textit{36} &  &\textit{40} & &\textit{45} &  \\
&\textit{5} & &\textit{10} & &\textit{15} & &\textbf{21}& & \textbf{27} & & \textit{33} &  &\textit{38} & &\textit{43} &  &\textit{48}\\
\textit{2 } & &\textit{7} & &\textit{13}& &\textbf{19} & &\textbf{24} & & \textbf{29}& & \textit{35} &  &\textit{41} & &\textit{46} &   \\
&\textit{4} & &\textit{10}  & &\textbf{17}& &\textbf{22}& &\textbf{26}& &  \textbf{31}& &\textit{38} &  &\textit{44}&  &\textit{48} \\
\textit{2 } & &\textit{7} & &\textbf{14} & &\textbf{20} & &\textbf{24}&  &\textbf{28} & & \textbf{34}& &\textit{41} &  &\textit{46} &  \\
&\textit{5}  & &\textbf{11}& &\textbf{17}& &\textbf{22} & &\textbf{26}&  &\textbf{31} & &  \textbf{37}& &\textit{43} &  &\textit{48}\\
\textit{3 }& &\textbf{9}& &\textbf{14}& &\textbf{19} & &\textbf{24}& & \textbf{29} & &\textbf{34} & &  \textbf{39}& &\textit{45} &\\
&\textbf{7} & &\textbf{12} & &\textbf{16} & &\textbf{21} & &\textbf{27}&  &\textbf{32} & & \textbf{36} & &  \textbf{41}& &\textit{48}\\
\textbf{4}& &\textbf{10} & &\textbf{14} & &\textbf{18}& & &  &\textbf{30} & &  \textbf{34} & &\textbf{38} & &  \textbf{44}&\\
\end{array}\quad\dots
$$
\caption{Array $\mathcal{N}_{A_{8}^{(2)}}^{(3,2,2,3,4)}$}
  \label{fig:cong6}
\end{figure}

\begin{example}
The labels of nodes in the specialized array  $\mathcal{N}_{{A^{(2)}_{8}}^T}^{(6,2,2,3,2)}$ are the positive integers congruent to
$$
0,0,0,0, \pm 2,\pm 2, \pm 2,\pm 3,\pm4, \pm 5, \pm 5,\pm 7,\pm 7,\pm 7,\pm 8, \pm 9, \pm 9,\pm 10, \pm 11,\pm 12
$$
modulo 24 and the positive integers
congruent to
$\pm 6,  \pm 10, \pm 14,\pm 20 $ modulo 48;
see Figure \ref{fig:cong7}.
\end{example}
\begin{figure}[htb!]
$$
\begin{array}{ccccccccccccccccccc}
\textit{2 }& & \textit{7}& & \textit{12}& & \textit{16}& & \textbf{24}&  & \textit{32} & &\textit{36} &  &\textit{41} & &\textit{46} &  \\
&\textit{5} & &\textit{9} & &\textit{14} & &\textbf{22}& & \textbf{26} & & \textit{34} &  &\textit{39} & &\textit{43} &  &\textit{48}\\
\textit{3 } & &\textit{7} & &\textit{11}& &\textbf{20} & &\textbf{24} & & \textbf{28}& &\textit{37} &  &\textit{41} & &\textit{45} &   \\
&\textit{5} & &\textit{9}  & &\textbf{17}& &\textbf{22}& &\textbf{26}& &  \textbf{31}& &\textit{39} &  &\textit{43}&  &\textit{48} \\
\textit{2 } & &\textit{7} & &\textbf{15} & &\textbf{19} & &\textbf{24}&  &\textbf{29} & & \textbf{33}& &\textit{41} &  &\textit{46} &  \\
&\textit{4}  & &\textbf{13}& &\textbf{17}& &\textbf{21} & &\textbf{27}&  &\textbf{31} & &  \textbf{35}& &\textit{44} &  &\textit{48}\\
\textit{2 }& &\textbf{10}& &\textbf{15}& &\textbf{19} & &\textbf{24}& & \textbf{29} & &\textbf{33} & &  \textbf{38}& &\textit{46} &\\
&\textbf{8} & &\textbf{12} & &\textbf{17} & &\textbf{22} & &\textbf{26}&  &\textbf{31} & & \textbf{36} & &  \textbf{40}& &\textit{48}\\
\textbf{6}& &\textbf{10} & &\textbf{14} & &\textbf{20}& & &  &\textbf{28} & &  \textbf{34} & &\textbf{38} & &  \textbf{42}&\\
\end{array}\quad\dots
$$
\caption{Array $\mathcal{N}_{{A^{(2)}_{8}}^T}^{(6,2,2,3,2)}$}
  \label{fig:cong7}
\end{figure}

\section{Explicit versions of Lepowsky's and Wakimoto's product formulas for $C_l^{(1)}$}\label{section_07}

\subsection{Lepowsky's formula of type $(1,1, \ldots,1, 1; \ C_l^{(1)})$}
Let $l$ be a nonnegative integer, and let $(s_0, s_1, \ldots , s_l)$ be $(l+1)$-tuple of positive integers. To write Lepowsky's product formula for $C_l^{(1)}$ in terms of generating function of $(s_0, s_1, \ldots , s_l)$-admissible partitions, the authors in \cite{CMPP} introduced the {\em congruence triangle} as the multiset 
\begin{eqnarray}\nonumber
&\Delta(s_1, \ldots, s_l; D_{l+1}^{(2)})=D(s_1,  \ldots, s_l ; D_{l+1}^{(2)}) \cup D(s_2,  \ldots, s_l ; D_{l+1}^{(2)})\cup \cdots \cup D(s_l; D_{l+1}^{(2)}),
\end{eqnarray}
where for $0 \leq i \leq l$, $D(s_i, s_{i+1}, \ldots, s_l; D_{l+1}^{(2)})$ denotes the set of the $2(l-i)+1$  integers
\begin{align*} 
&s_i,\, s_i+s_{i+1}, \, \ldots,\,  s_i+s_{i+1}+ \ldots +s_{l-2}+s_{l-1}, \\
&s_i+s_{i+1}+ \ldots +s_{l-1}+ s_l,\,  s_i+s_{i+1}+ \ldots +s_{l-1}+ 2s_l, \\
&s_i+s_{i+1}+ \ldots +2s_{l-1}+ 2s_l,\, 
  \ldots,\,  s_i+2s_{i+1}+ \ldots +2s_{l-1}+ 2s_l .
\end{align*}
The elements of the multiset $\Delta(s_1, \ldots, s_l; D_{l+1}^{(2)})$ correspond to the elements of the upper left right-angled triangle with vertices  $s_1$, $s_l$ and $s_1+2\sum_{i=2}^ls_i$ in the $(s_0, s_1, \ldots , s_l)$-specialized array   $\mathcal{N}_{D_{l+1}^{(2)}}^{(s_0, s_1, \ldots , s_l)}$. 

\begin{example}\label{example_xy_1}
In  Figure \ref{fig:cong1}, the elements of the sets $D(s_i,\ldots ,s_4; D_{5}^{(2)})$ are denoted by
$\textcolor{red}{\tcircled{1}},\textcolor{green}{\tcircled{2}},\textcolor{blue}{\tcircled{3}},\textcolor{orange}{\tcircled{4}} $
for $i=1,2,3,4$, respectively. Hence, for example, the symbol $\textcolor{green}{\tcircled{2}}$ represents
the elements
\beq\label{elements_pr1}
s_2,\,s_2+s_3,\,s_2+s_3+s_4,\, s_2+s_3+2s_4,\, s_2+2s_3+2s_4.
\eeq
The remaining elements of the array which belong to a diagonal (resp. a triangle) are denoted by the symbol $\bullet$ (resp. $\circ$).
\end{example}

\begin{figure}[htb!]
$$
\begin{array}{cccccccccccccccccccc}
s_4  \,\textcolor{orange}{\circled{4}}& &  \textcolor{blue}{\circled{3}}& & \textcolor{green}{\circled{2}} & & \textcolor{red}{\textcolor{red}{\circled{1}}} & & {\bullet}&  & {\bullet} & &{\circ} &  &{\circ} & &{\circ} &  &{\circ} &\\
& \textcolor{blue}{\circled{3}} & &\textcolor{green}{\circled{2}}  & &\textcolor{red}{\textcolor{red}{\circled{1}}} & &{\bullet}& &{\circ} & & {\bullet}& &{\circ} &  &{\circ} & &{\circ} &  &{\circ}\\
 s_3\, \textcolor{blue}{\circled{3}} & &\textcolor{green}{\circled{2}}  & &\textcolor{red}{\textcolor{red}{\circled{1}}}& &{\bullet} & &{\circ} &  &{\circ} & & {\bullet}& &{\circ} &  &{\circ} & &{\circ} &   \\
&\textcolor{green}{\circled{2}}  & &\textcolor{red}{\textcolor{red}{\circled{1}}}  & &{\bullet}& &{\circ}& &{\circ} & &{\circ} & &  {\bullet}& &{\circ} &  &{\circ}&  &{\circ} \\
 s_2\,\textcolor{green}{\circled{2}}  & &\textcolor{red}{\textcolor{red}{\circled{1}}} & &{\bullet} & &{\circ} & &{\circ}&  &{\circ} & & {\circ} & & {\bullet}& &{\circ} &  &{\circ} &  \\
&\textcolor{red}{\textcolor{red}{\circled{1}}}  & &{\bullet}& &{\circ}& &{\circ} & &{\circ}&  &{\circ}&  &{\circ} & &  {\bullet}& &{\circ} &  &{\circ}\\
 {s_1\,\textcolor{red}{\circled{1}} }& &{\bullet}& &{\circ} & &{\circ} & &{\circ}& &  {\circ} & &{\circ}&  &{\circ} & &  {\bullet}& &{\circ} &\\
&{\bullet} & &{\circ} & &{\circ} & &{\circ} & &{\circ}&  &{\circ} & &  {\circ} & &{\circ} & &  {\bullet}& &{\circ}\\
{s_0 \,\bullet}& &{\circ} & &{\circ} & &{\circ}& &{\circ} &  &{\circ}&  &{\circ} & &  {\circ} & &{\circ} & &  {\bullet}&\\
\end{array}\quad\dots
$$
\caption{Congruence triangle $\Delta(s_1, s_2,s_3, s_4; D_{5}^{(2)})$}
  \label{fig:cong1}
\end{figure}


In general, as we demonstrated in Subsection \ref{subsect_52}, the elements of
the array $\mathcal{N}_{D_{l+1}^{(2)}}^{(s_0,\ldots, s_l)}$
can be organized into a disjoint union of triangles and diagonals. Suppose that all triangles, with an exception of the  grey triangle $\Delta(s_1,\ldots, s_l; D_{l+1}^{(2)})$, are divided into two right-angled triangles; see Figure \ref{fig:cong101a}. 
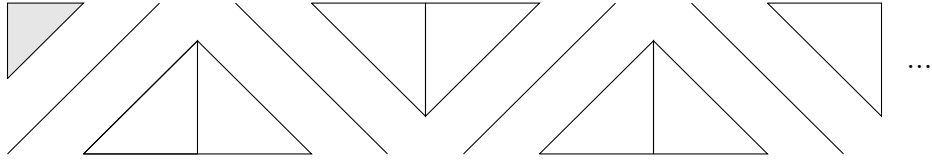
\begin{figure}[htb!]\begin{tikzpicture}
        \coordinate (a) at (0,0);
        \coordinate (b) at (0,1);
        \coordinate (c) at (1,1);
      \filldraw[draw=black, fill=gray!20] (a) -- (b) -- (c) -- (a); 
\coordinate (d) at (0,-1);
\coordinate (e) at (2,1);
\draw (d) -- (e); 
\coordinate (f) at (1,-1);
\coordinate (g1) at (2,-0.06);
\coordinate (g) at (2.5,0.5);
\coordinate (h) at (2.5,-1);
\coordinate (h1) at (2,-1);
\draw (f) -- (g) -- (h) -- (f); 
\coordinate (i) at (1,-1);
\coordinate (j) at (2.5,0.5);
\coordinate (k) at (4,-1);
        \draw (i) -- (j) -- (k) -- (i); 
\coordinate (l) at (5,-1);
\coordinate (m) at (3,1);
 \draw (l) -- (m); 
\coordinate (n) at (4,1);
\coordinate (o) at (5.5,-0.5);
\coordinate (p) at (5.5,1);
        \draw (n) -- (o) -- (p) -- (n); 
\coordinate (q) at (5.5,1);
\coordinate (r) at (5.5,-0.5);
\coordinate (s) at (7,1);
        \draw (q) -- (r) -- (s) -- (q); 
\coordinate (t) at (6,-1);
\coordinate (u) at (8,1);
\draw (t) -- (u); 
\coordinate (v) at (7,-1);
\coordinate (w) at (8.5,-1);
\coordinate (x) at (8.5,0.5);
        \draw (v) -- (w) -- (x) -- (v); 
\coordinate (y) at (10,-1);
       \draw (w) -- (x) -- (y) -- (w); 
\coordinate (z) at (9,1);
\coordinate (1) at (11,-1);
\draw (z) -- (1); 
\coordinate (2) at (10,1);
\coordinate (3) at (11.5,1);
\coordinate (4) at (11.5,-0.5);
        \draw (2) -- (3) -- (4) -- (2); 
\draw (12,0)  node[anchor=south] {...};
\end{tikzpicture}
\caption{Array $\mathcal{N}_{D_{l+1}^{(2)}}^{(s_0, s_1, \ldots , s_l)}$}
  \label{fig:cong101a}\end{figure}
Let us consider Figure \ref{fig:cong11a}, where the common  catheti of these right-angled triangles are denoted by $IJ$. 
Then the grey triangles $ABC$ are the specialization  of $\mathfrak n_-\otimes t^{2j}\subset \mathcal B_{D_{l+1}^{(2)}}^-$, $j\leq0$, the black triangles $DEF$ are the specialization  of $\mathfrak n_+\otimes t^{2j}$, $j<0$, and common catheti $IJ$ are the specialization  of $\mathfrak h\otimes t^{2j}$, $j<0$; where $\mathfrak n_-+\mathfrak h+\mathfrak n_+$ is the triangular decomposition of $\mathfrak g(B_l)$. On the other side, the lines $HG$ are the specialization of $\mathcal B_{B_{l}}(\omega_1)\otimes t^{2j-1} $, $j\leq0$.
By closely examining the weights
of  $\mathcal{B}_{D_{l+1}^{(2)}}^{-}$, one observes that the elements   of the congruence triangle $\Delta(s_1, \ldots, s_l; D_{l+1}^{(2)})$ are congruent modulo $\pi=2\sum_{i=0}^ls_i$ to the corresponding elements of other grey right-angled triangles $ABC$.  The same observation holds true for the corresponding elements of the black triangles $DEF$, lines $IJ$ and the diagonals $GH$, as indicated by the labels of their vertices. In particular,  the elements of the  common  catheti $IJ$ 
  are multiples of $2\pi$. Furthermore, the   black triangles $DEF$   consist  of   modular additive inverses modulo $\pi$ of the corresponding elements of the grey triangles $ABC$.

\begin{figure}[htb!]\begin{tikzpicture}
        \coordinate (a) at (0,0);
        \coordinate (b) at (0,1);
        \coordinate (c) at (1,1);
      \filldraw[draw=black, fill=gray!20] (a) -- (b) -- (c) -- (a); 
\fill[fill=gray!20] (a)--(b)--(c)--(a);
\coordinate (d) at (0,-1);
\coordinate (e) at (2,1);
\draw (d) -- (e); 
\coordinate (f) at (1,-1);
\coordinate (g1) at (2,-0.01);
\coordinate (g) at (2.5,0.5);
\coordinate (h) at (2.5,-1);
\coordinate (h1) at (2,-1);
\filldraw[draw=black, fill=black] (f) -- (g1) -- (h1) -- (f); 
\draw (g) -- (h); 
\coordinate (i) at (1,-1);
\coordinate (j) at (2.5,0.5);
\coordinate (j1) at (3,-1);
\coordinate (k) at (4,-1);
\coordinate (k1) at (3,-0.01);
  \filldraw[draw=black, fill=gray!20] (k) -- (j1) -- (k1) -- (k); 
\coordinate (l) at (5,-1);
\coordinate (m) at (3,1);
 \draw (l) -- (m); 
\coordinate (n) at (4,1);
\coordinate (o) at (5.5,-0.5);
\coordinate (p) at (5.5,1);
\coordinate (o1) at (5,0.01);
\coordinate (p1) at (5,1);
\filldraw[draw=black, fill=black] (n) -- (o1) -- (p1) -- (n); 
        \draw  (o) -- (p); 
\coordinate (q) at (5.5,1);
\coordinate (r) at (5.5,-0.5);
\coordinate (s) at (7,1);
\coordinate (r1) at (6,1);
\coordinate (s1) at (6,0.01);
  \filldraw[draw=black, fill=gray!20] (s) -- (r1) -- (s1) -- (s); 
\coordinate (t) at (6,-1);
\coordinate (u) at (8,1);
\draw (t) -- (u); 
\coordinate (v) at (7,-1);
\coordinate (w) at (8.5,-1);
\coordinate (x) at (8.5,0.5);
\coordinate (w2) at (8,-1);
\coordinate (x2) at (8,-0.01);
 \filldraw[draw=black, fill=black] (v) -- (w2) -- (x2) -- (v); 
\coordinate (y) at (10,-1);
       \draw (w) -- (x); 
\coordinate (x1) at (9,-1);
\coordinate (w1) at (9,-0.01);
  \filldraw[draw=black, fill=gray!20] (y) -- (x1) -- (w1) -- (y); 
 \coordinate (z) at (9,1);
\coordinate (1) at (11,-1);
\draw (z) -- (1); 
\coordinate (2) at (10,1);
\coordinate (3) at (11.5,1);
\coordinate (4) at (11.5,-0.5);
\coordinate (31) at (11,1);
\coordinate (41) at (11,0.01);
  \filldraw[draw=black, fill=black] (2) -- (31) -- (41) -- (2); 
       \draw (3) -- (4); 
\draw (12,0)  node[anchor=south] {...};
\node[draw=none] at (0.01,-0.2) {A};
\node[draw=none] at (0,1.2) {C};
\node[draw=none] at (1,1.2) {B};
\node[draw=none] at (1,-1.2) {D};
\node[draw=none] at (2,-1.2) {E};
\node[draw=none] at (2,0.2) {F};
\node[draw=none] at (2.5,0.7) {I};
\node[draw=none] at (2.5,-1.2) {J};
\node[draw=none] at (3.14,0.2) {A};
\node[draw=none] at (3,-1.2) {C};
\node[draw=none] at (4,-1.2) {B};
\node[draw=none] at (4,1.2) {D};
\node[draw=none] at (5,1.2) {E};
\node[draw=none] at (5,-0.2) {F};
\node[draw=none] at (5.5,-0.7) {I};
\node[draw=none] at (5.5,1.2) {J};
\node[draw=none] at (6.01,-0.2) {A};
\node[draw=none] at (6,1.2) {C};
\node[draw=none] at (7,1.2) {B};
\node[draw=none] at (7,-1.2) {D};
\node[draw=none] at (8,-1.2) {E};
\node[draw=none] at (8,0.2) {F};
\node[draw=none] at (8.5,0.7) {I};
\node[draw=none] at (8.5,-1.2) {J};
\node[draw=none] at (9.14,0.2) {A};
\node[draw=none] at (9,-1.2) {C};
\node[draw=none] at (10,-1.2) {B};
\node[draw=none] at (10,1.2) {D};
\node[draw=none] at (11,1.2) {E};
\node[draw=none] at (11,-0.2) {F};
\node[draw=none] at (11.5,-0.7) {I};
\node[draw=none] at (11.5,1.2) {J};
\node[draw=none] at (0,-1.2) {G};
\node[draw=none] at (2,1.2) {H};
\node[draw=none] at (3,1.2) {G};
\node[draw=none] at (5,-1.2) {H};
\node[draw=none] at (6,-1.2) {G};
\node[draw=none] at (8,1.2) {H};
\node[draw=none] at (9,1.2) {G};
\node[draw=none] at (11,-1.2) {H};
\end{tikzpicture}
\caption{Positions of congruent elements in the array ${\mathcal N}_{D_{l+1}^{(2)}}^{(s_0, s_1, \ldots , s_l)}$}
  \label{fig:cong11a}\end{figure}
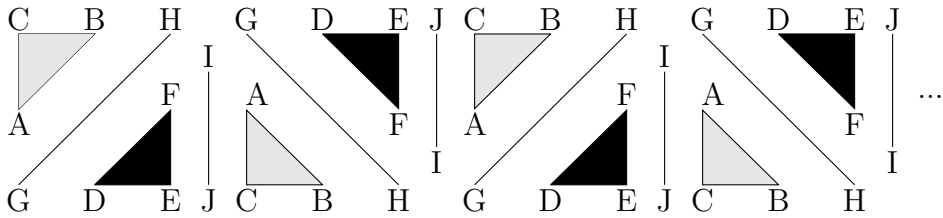

\begin{remark}
By a  reasoning similar to above we can write formulas for $Q(s_0, s_1, \dots. s_l; C_{l}^{(1)})$ and $Q(s_0, s_1, \dots. s_l; A_{2l}^{(2)})$ needed for  Theorems \ref{T: Wakimotos formula for 21...12} and \ref{T: Wakimotos formula for 1...12} below. 
Positions of congruent elements in the array ${\mathcal N}_{C_{l}^{(1)}}^{(s_0, s_1, \ldots , s_l)}$ look like Figure \ref{fig:cong11a} without lines $GH$.

On the other side, the array $\mathcal{B}_{A_{2l}^{(2)}}^{-}$ consists of $\mathfrak n_-\otimes t^0$, small triangles $L_{B_l}(\theta)\otimes t^{2j}$ and big triangles $L_{B_l}(2\theta)\otimes t^{2j+1}$, $j<0$. After specialization $\mathfrak n_-\otimes t^0$ becomes the gray triangle in the upper left corner of  $\mathcal N_{A_{2l}^{(1)}}^{(s_0, s_1, \dots, s_1)}$,
and the small triangle becomes the union of the black triangle $DEF$, the vertical line $IJ$ and the gray triangle $ABC$ with the hypotenuse in the top row.  Finally, after specialization the big triangle has its hypotenuse in the bottom row, and the remaining part is the union of the black triangle $DEF$, the vertical line $IJ$ and the gray triangle $ABC$. For $l=2$ the last statement is obvious if we put the (rectified) small triangle in the lower left corner of Figure \ref{F: crystal of the tensor square of vector representation for B2} and the big triangle without hypotenuse in the upper left corner: one is the transpose of the other.
\end{remark}

\begin{example}
We return to the setting of Example \ref{example_xy_1} to  illustrate the preceding discussion. 
Consider Figure \ref{fig:cong11}. It shows the array $\mathcal{N}_{D_{5}^{(2)}}^{(s_0, s_1, s_2 , s_3,s_4)}$ modulo $\pi =2\sum_{i=0}^4s_i$, where its triangle (resp. diagonal) elements are labeled by the upper case letters, lower case letters and zeros $\textbf{0} $ (resp. bullets $\bullet_i$ for $i=1, \dots, 9$). Note that we use the upper case letters for the elements of the congruence triangle  $\Delta(s_1, s_2,s_3, s_4; D_{5}^{(2)})$ and that the common catheti $IJ$ from Figure \ref{fig:cong11a} consists of zeros $\textbf{0}, \textbf{0}, \textbf{0}, \textbf{0}$.

The triangle in Figure \ref{fig:cong11} with the vertices $\textbf{0}$, $\textit{d}$, $\textit{D}$ (which are underlined in the figure) and hypotenuse $\textit{d}\textit{D}$ in the bottom row is the ``$(s_0, s_1, s_2 , s_3,s_4)\mod\pi$''-specialization of the triangle 
$L_{B_l}(\theta)\otimes t^2\subset \mathcal{B}_{D_{5}^{(2)}}^-$.
From Figures \ref{F: crystal of the tensor square of vector representation for B2} and \ref{F: crystal of the tensor square of vector representation for B3} we see that the positions of root vectors for roots $\alpha$ and $-\alpha$ are symmetric with respect to the ``line corresponding to $\mathfrak h$'' (i.e. $IJ$ from Figure \ref{fig:cong11a}). Hence $A+a=B+b=\dots=P+p=0$.
\end{example}

\begin{figure}[htb!]
$$
\begin{array}{cccccccccccccccccccc}
\textit{A}& &  \textit{B}& & \textit{C} & & \textit{D} & & \bullet_{9}&  &\bullet_{1} & &\textit{d} &  &\textit{c} & &\textit{b} &  &\textit{a} &\\
& \textit{E} & &\textit{F}  & &\textit{G} & &\bullet_{8}& &\underline{\textbf{0}} & &\bullet_{2}& &\textit{g} &  &\textit{f} & &\textit{e} &  &\textbf{0}\\
 \textit{H} & &\textit{I} & &\textit{J}& &\bullet_{7}& &\textit{p} &  &\textit{P} & & \bullet_{3}& &\textit{j} &  &\textit{i} & &\textit{h} &   \\
&\textit{K}  & &\textit{L}  & &\bullet_{6}& &\textit{o} & &\textbf{0} & &\textit{O}& & \bullet_{4}& &\textit{l} &  &\textit{k}&  &\textbf{0} \\
 \textit{M}  & &\textit{N} & &\bullet_{5} & &\textit{n} & &\textit{m}&  &\textit{M} & & \textit{N} & & \bullet_{5}& &\textit{n} & &\textit{m} &  \\
&\textit{O}  & &\bullet_{4}& &\textit{l} &  &\textit{k}&  &\textbf{0}&  &\textit{K}&  &\textit{L} & &\bullet_{6}& &\textit{o} & &\textbf{0}\\
\textit{P}& &\bullet_{3}& &\textit{j} &  &\textit{i} & &\textit{h}& &  \textit{H} & &\textit{I}&  &\textit{J} & &  \bullet_{7}& &\textit{p} &\\
&\bullet_{2} & &\textit{g} & &\textit{f} & &\textit{e} & &\textbf{0}&  &\textit{E} & &  \textit{F} & &\textit{G} & & \bullet_{8}& &\textbf{0}\\
\bullet_{1}& &\underline{\textit{d}} & &\textit{c} & &\textit{b}& &\textit{a} &  &\textit{A}&  &\textit{B} & &  \textit{C} & &\underline{\textit{D}} & & \bullet_{9}& 
\end{array}\quad\dots
$$
\caption{Congruent elements in the array $\mathcal{N}_{D_{5}^{(2)}}^{(s_0, s_1, s_2 , s_3,s_4)}$}
  \label{fig:cong11}
\end{figure}

By using the above analysis of the array $\mathcal{N}_{D_{l+1}^{(2)}}^{(s_0,s_1,\ldots,s_l)}$, one obtains Lepowsky's product formulas for the principally specialized character of a standard $C^{(1)}_l$-modules:


\begin{theorem}\label{T: Lepowskys formula}
Lepowsky's product formula for the principally specialized character of the standard $C^{(1)}_l$-module $L_{C_l^{(1)}}(\Lambda)$ of highest weight $\Lambda=\sum_{i=0}^l k_i\Lambda_i$ can be written as
\begin{align}\label{f01}
&\ch^{(1,\ldots,1; \ C_l^{(1)})}L(k_0, k_1, \ldots, k_{l-1}, k_l)\nonumber\\ 
=& \frac{\prod_{i \equiv a, \pm b \mod  2(k+l+1), \  a\in  \{0\}^l \cup D(k_0+1,k_1+1,\ldots ,k_l+1; D_{l+1}^{(2)}), \ b \in \Delta(k_1+1,k_2+1, \ldots, k_{l-1}+1, k_l+1; D_{l+1}^{(2)})}(1-q^i)}{\prod_{j \in \mathbb{N}}(1-q^j)^{l}\prod_{j\in\mathbb{N}}(1-q^{2j-1})},\end{align}
where $\{0\}^l$ denotes the multiset consisting of $l$ copies of $0$. 
\end{theorem}
\begin{proof}
The principal specialization of the character of $L(\Lambda)$ is given by
\beq\label{Lep1}
\ch^{(1,\ldots,1; \ C_l^{(1)})}L(k_0, k_1, \ldots, k_{l-1}, k_l)=\frac{Q(k_0+1,k_1+1,\ldots,k_l+1; D_{l+1}^{(2)})}{Q(1,\ldots,1; C_l^{(1)})}.
\eeq
 Lemma \ref{L: Q(1...1;C)} implies that the denominator $Q(1,\ldots,1; C_l^{(1)})$ of  \eqref{Lep1}  is equal to $$\prod_{j \in \mathbb{N}}(1-q^j)^{l}\prod_{j\in\mathbb{N}}(1-q^{2j-1}).$$

The elements in the array $\mathcal{N}_{D_{l+1}^{(2)}}^{(k_0+1,k_1+1,\ldots,k_l+1)}$ are congruent modulo $2(k+l+1)$. By using the direct calculations on the elements of the array $\mathcal{N}_{D_{l+1}^{(2)}}^{(k_0+1,k_1+1,\ldots,k_l+1)}$, follows that the numerator $Q(k_0+1,k_1+1,\ldots,k_l+1; D_{l+1}^{(2)})$ of   \eqref{Lep1}  can be written as 
$$
\prod_{ i \equiv a, \pm b \mod  2(k+l+1), \,  a\in  \{0\}^l \cup D(k_0+1,k_1+1,\ldots ,k_l+1; D_{l+1}^{(2)}), \, b \in \Delta(k_1+1,k_2+1, \ldots, k_{l-1}+1, k_l+1; D_{l+1}^{(2)}) }(1-q^i).\qedhere
$$
\end{proof}
\begin{remark}
Lepowsky's product formula is written in this way in \cite{CMPP}.  In the same vein we define three types of sets $D(s_i, s_{i+1}, \ldots, s_l; X_l^{(r)})$ and corresponding types of congruence triangles, where $X_l^{(r)}=C_l^{(1)}, \ A^{(2)}_{2l}$ or ${A^{(2)}_{2l}}^T$.
\end{remark}

\subsection{Wakimoto's formulas of type $(2,1, \ldots,1, 2; \ C_l^{(1)})$}

For   $1 \leq i \leq l$, we define the set  $D(s_i, s_{i+1}, \ldots, s_l; C_l^{(1)})$   of cardinality $2(l-i)+1$ which consists of the integers
\begin{align*} 
 & s_i,\, s_i+s_{i+1},\, \ldots, \, s_i+s_{i+1}+ \ldots +s_{l-2}+s_{l-1},  \,
 s_i+s_{i+1}+ \ldots +s_{l-1}+ s_l,    \\
 & s_i+s_{i+1}+ \ldots +2s_{l-1}+ s_l,\,\ldots,\, 2s_i+2s_{i+1}+ \ldots +2s_{l-1}+ s_l .
\end{align*}
Moreover, we introduce the {\em congruence triangle} $\Delta(s_1, \ldots, s_l; C_l^{(1)})$ as 
\begin{equation*}
\Delta(s_1, \ldots, s_l; C_l^{(1)})= 
D(s_1,  \ldots, s_l; C_l^{(1)} ) \cup D(s_2,  \ldots, s_l; C_l^{(1)})\cup \cdots \cup D(s_l; C_l^{(1)}). 
\end{equation*}
The elements of the multiset $\Delta(s_1, \ldots, s_l; C_l^{(1)})$ correspond to  the elements of the upper left right-angled triangle with vertices $s_1$, $s_l$ and $s_l+2\sum_{i=1}^{l-1}s_i$ in the $(s_0, s_1, \ldots , s_l)$--specialized array of negative roots $\mathcal{N}_{C_l^{(1)}}^{(s_0, s_1, \ldots , s_l)}$; see Figure \ref{fig:cong11b}. In   Figure \ref{fig:cong2}, the elements of $D(s_i,\ldots ,s_4; C_{4}^{(1)}) $ are denoted by
$\textcolor{red}{\tcircled{1}},\textcolor{green}{\tcircled{2}},\textcolor{blue}{\tcircled{3}},\textcolor{orange}{\tcircled{4}},$
for $i=1,2,3,4$, respectively, e.g., the symbol $\textcolor{green}{\tcircled{2}}$ represents
the elements
$$
s_2,\,s_2+s_3,\,s_2+s_3+s_4,\, s_2+2s_3+s_4,\,2 s_2+2s_3+s_4.
$$  
The remaining elements of the array   are denoted by the symbol  $\circ$.
\begin{figure}[htb!]
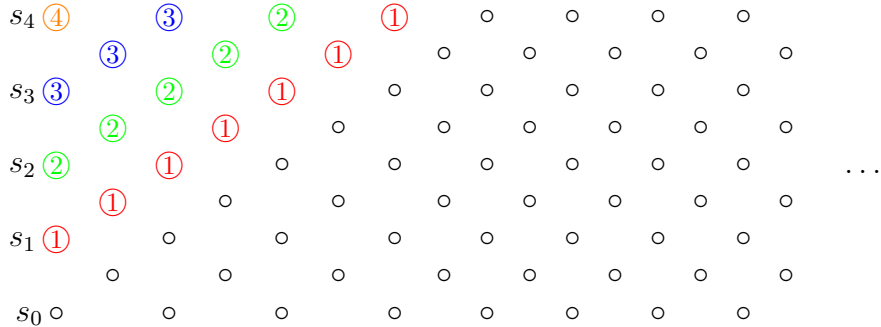

$$
\begin{array}{cccccccccccccccc}
 s_4\,\textcolor{orange}{\circled{4}}& &   \textcolor{blue}{\circled{3}}& & \textcolor{green}{\circled{2}}& &  \textcolor{red}{\circled{1}}& & {\circ} & &{\circ} &  &{\circ} & &{\circ} &  \\
&\textcolor{blue}{\circled{3}} & & \textcolor{green}{\circled{2}} & &\textcolor{red}{\circled{1}}& &{\circ}& &{\circ} & & {\circ} &  &{\circ} & &{\circ}\\
 s_3\,\textcolor{blue}{\circled{3}} & & \textcolor{green}{\circled{2}} & &\textcolor{red}{\circled{1}}& &{\circ} & &{\circ} &  &{\circ} & &{\circ} &  &{\circ} &  \\
& \textcolor{green}{\circled{2}} & &\textcolor{red}{\circled{1}}   & &{\circ}& &{\circ} & &  {\circ}& &{\circ} &  &{\circ}&  &{\circ} \\
 s_2\, \textcolor{green}{\circled{2}} & &\textcolor{red}{\circled{1}} & &{\circ} & &{\circ} & &{\circ} & & {\circ}& &{\circ} &  &{\circ} &  \\
&\textcolor{red}{\circled{1}}    & &{\circ}& &{\circ} & &{\circ} & &{\circ} & &  {\circ}& &{\circ} &  &{\circ}\\
 s_1\,\textcolor{red}{\circled{1}}& &{\circ}& &{\circ} & &{\circ} & &{\circ} & &{\circ} & &  {\circ}& &{\circ} &\\
&{\circ} & &{\circ} & &{\circ} & &{\circ} & &{\circ} & &{\circ} & & {\circ}& &{\circ}\\
{s_0 \,\circ }& &{\circ} & &{\circ} & &{\circ} & &{\circ} & &{\circ} & & {\circ} & &{\circ}&\\
\end{array}\quad\dots
$$
\caption{Congruence triangle $\Delta(s_1, s_2,s_3, s_4; C_{4}^{(1)})$}
  \label{fig:cong2}
\end{figure}

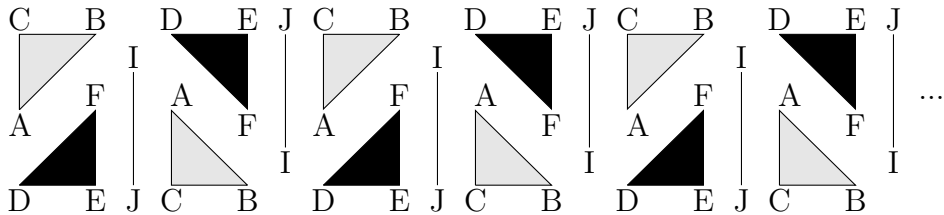
\begin{figure}[htb!]\begin{tikzpicture}
        \coordinate (a) at (0,0);
        \coordinate (b) at (0,1);
        \coordinate (c) at (1,1);
        \draw (a) -- (b);
      \filldraw[draw=black, fill=gray!20] (a) -- (b) -- (c) -- (a); 
\coordinate (f) at (0,-1);
\coordinate (g1) at (1,-0.01);
\coordinate (g) at (1.5,0.5);
\coordinate (h) at (1.5,-1);
\coordinate (h1) at (1,-1);
\filldraw[draw=black, fill=black] (f) -- (g1) -- (h1) -- (f); 
\draw (g) -- (h); 
\coordinate (i) at (1,-1);
\coordinate (j) at (2.5,0.5);
\coordinate (j1) at (2,-1);
\coordinate (k) at (3,-1);
\coordinate (k1) at (2,-0.01);
  \filldraw[draw=black, fill=gray!20] (k) -- (j1) -- (k1) -- (k); 
\coordinate (n) at (2,1);
\coordinate (o) at (3.5,-0.5);
\coordinate (p) at (3.5,1);
\coordinate (o1) at (3,0.01);
\coordinate (p1) at (3,1);
\filldraw[draw=black, fill=black] (n) -- (o1) -- (p1) -- (n); 
        \draw  (o) -- (p); 
\coordinate (q) at (3.5,1);
\coordinate (r) at (3.5,-0.5);
\coordinate (s) at (5,1);
\coordinate (r1) at (4,1);
\coordinate (s1) at (4,0.01);
  \filldraw[draw=black, fill=gray!20] (s) -- (r1) -- (s1) -- (s); 
\coordinate (v) at (4,-1);
\coordinate (w) at (5.5,-1);
\coordinate (x) at (5.5,0.5);
\coordinate (w2) at (5,-1);
\coordinate (x2) at (5,-0.01);
 \filldraw[draw=black, fill=black] (v) -- (w2) -- (x2) -- (v); 
\coordinate (y) at (6,-1);
       \draw (w) -- (x); 
\coordinate (x1) at (7,-1);
\coordinate (w1) at (6,-0.01);
  \filldraw[draw=black, fill=gray!20] (y) -- (x1) -- (w1) -- (y); 
\coordinate (2) at (6,1);
\coordinate (3) at (7.5,1);
\coordinate (4) at (7.5,-0.5);
\coordinate (31) at (7,1);
\coordinate (41) at (7,0.01);
  \filldraw[draw=black, fill=black] (2) -- (31) -- (41) -- (2); 
       \draw (3) -- (4); 
\coordinate (5) at (8,1);
\coordinate (6) at (8,0.01);
\coordinate (7) at (9,1);
  \filldraw[draw=black, fill=gray!20] (5) -- (6) -- (7) -- (5); 
\coordinate (8) at (8,-1);
\coordinate (9) at (9.5,-1);
\coordinate (10) at (9.5,0.5);
\coordinate (81) at (9,-1);
\coordinate (91) at (9,-0.01);
 \filldraw[draw=black, fill=black] (8) -- (81) -- (91) -- (8); 
\draw (9)--(10);
\coordinate (11) at (10,1);
\coordinate (12) at (11.5,1);
\coordinate (13) at (11.5,-0.5);
\coordinate (111) at (11,1);
\coordinate (121) at (11,0.01);
 \filldraw[draw=black, fill=black] (11) -- (111) -- (121) -- (11); 
\draw (12)--(13);
\coordinate (14) at (10,-1);
\coordinate (141) at (11,-1);
\coordinate (151) at (10,-0.01);
 \filldraw[draw=black, fill=gray!20] (14) -- (141) -- (151) -- (14); 
\draw (12,0)  node[anchor=south] {...};
\node[draw=none] at (0.01,-0.2) {A};
\node[draw=none] at (0,1.2) {C};
\node[draw=none] at (1,1.2) {B};
\node[draw=none] at (0,-1.2) {D};
\node[draw=none] at (1,-1.2) {E};
\node[draw=none] at (1,0.2) {F};
\node[draw=none] at (1.5,0.7) {I};
\node[draw=none] at (1.5,-1.2) {J};
\node[draw=none] at (2.14,0.2) {A};
\node[draw=none] at (2,-1.2) {C};
\node[draw=none] at (3,-1.2) {B};
\node[draw=none] at (2,1.2) {D};
\node[draw=none] at (3,1.2) {E};
\node[draw=none] at (3,-0.2) {F};
\node[draw=none] at (3.5,-0.7) {I};
\node[draw=none] at (3.5,1.2) {J};
\node[draw=none] at (4.01,-0.2) {A};
\node[draw=none] at (4,1.2) {C};
\node[draw=none] at (5,1.2) {B};
\node[draw=none] at (4,-1.2) {D};
\node[draw=none] at (5,-1.2) {E};
\node[draw=none] at (5,0.2) {F};
\node[draw=none] at (5.5,0.7) {I};
\node[draw=none] at (5.5,-1.2) {J};
\node[draw=none] at (6.14,0.2) {A};
\node[draw=none] at (6,-1.2) {C};
\node[draw=none] at (7,-1.2) {B};
\node[draw=none] at (6,1.2) {D};
\node[draw=none] at (7,1.2) {E};
\node[draw=none] at (7,-0.2) {F};
\node[draw=none] at (7.5,-0.7) {I};
\node[draw=none] at (7.5,1.2) {J};
\node[draw=none] at (8.14,-0.2) {A};
\node[draw=none] at (8,1.2) {C};
\node[draw=none] at (9,1.2) {B};
\node[draw=none] at (8,-1.2) {D};
\node[draw=none] at (9,-1.2) {E};
\node[draw=none] at (9,0.2) {F};
\node[draw=none] at (9.5,0.7) {I};
\node[draw=none] at (9.5,-1.2) {J};
\node[draw=none] at (10.14,0.2) {A};
\node[draw=none] at (10,-1.2) {C};
\node[draw=none] at (11,-1.2) {B};
\node[draw=none] at (10,1.2) {D};
\node[draw=none] at (11,1.2) {E};
\node[draw=none] at (11,-0.2) {F};
\node[draw=none] at (11.5,-0.7) {I};
\node[draw=none] at (11.5,1.2) {J};
\end{tikzpicture}
\caption{Positions of congruent elements in the array ${\mathcal N}_{C_l^{(1)}}^{(s_0, s_1, \ldots , s_l)}$}
  \label{fig:cong11b}\end{figure}

\begin{example}\label{e1}
For $s=(4,3,2,3,4)$, we have
\begin{align*}
D(3,2,3,4; C_{4}^{(1)}) &= \{3,5,8,12,15,17,20\},\\
D(2,3,4; C_{4}^{(1)}) &= \{2,5,9,12,14\},\\
D(3,4; C_{4}^{(1)})   &= \{3,7,10\},\\
D(4; C_{4}^{(1)})&= \{4\}.
\end{align*}
The elements of these sets correspond to the italic nodes in the upper left triangle of the array $\mathcal{N}_{C_4^{(1)}}^{(4,3,2,3,4)}$ on Figure \ref{fig:cong3corr}.  
\end{example}
\begin{theorem}\label{T: Wakimotos formula for 21...12}
Wakimoto's product formula for the $(2,1, \ldots,1,2)$-specialized character of the standard module $L_{C_l^{(1)}}(\Lambda)$ of highest weight $\Lambda=\sum_{i=0}^l k_i\Lambda_i$    \eqref{W83_3}  can be written as
\begin{align}\label{f1}
&\ch^{(2,1, \ldots,1,2; \ C_l^{(1)})}L(k_0, k_1, \ldots, k_{l-1}, k_l)\nonumber\\ 
=& \frac{\prod_{i \equiv a,\pm b \mod  2(k+l+1), \ a \in \{0\}^l,
\ b \in \Delta(k_1+1,k_2+1, \ldots, k_{l-1}+1, 2(k_l+1); \, C_l^{(1)})}(1-q^i)}{\prod_{i \in \mathbb{N}}(1-q^i)^{l-1}\prod_{j \in \mathbb{N}}(1-q^{2j})\prod_{r\equiv  (l+1) \mod 2(l+1)}(1-q^r)}.\end{align}
\end{theorem}
 
\begin{remark}
We note that the congruence triangle does not depend on $s_0=2(k_0+1)$, but the product formula depends on $k_0$ via the modulus in the congruence condition $\mod 2(k+l+1)$.
\end{remark} 
 
 \begin{example}\label{example_79a} 
From Examples \ref{e1} and \ref{e21..12} follows   the specialized character formula 
\begin{align*}
\ch^{(2,1, 1,1,2; \ C_4^{(1)})}L(1,2,1,2,1)=&\prod_{j \equiv \pm 1, \pm 1, \pm 2, \pm 6, \pm 6, \pm 8, \pm 11, \pm 11 \mod 24} (1-q^j)^{-1}\\
&\times \prod_{\substack{i \in \mathbb{N}\\ i \not\equiv 0, 12 \mod 24}}(1-q^i)^{-1}\prod_{\substack{s \in 2\mathbb{N}\\ s \not\equiv 0, 12 \mod 24}}(1-q^s)^{-1}\\
& \times \prod_{r \equiv 5  \mod 10} (1-q^r)^{-1}.\end{align*}
\end{example} 

\subsection{Wakimoto's formulas of types $(1,\ldots,1,2; \ C_l^{(1)})$ and $(2,1, \ldots,1; \ C_l^{(1)})$}

 For   $0 \leq i \leq l-1$, we define the set  $D(s_i, s_{i-1}, \ldots, s_0; {A^{(2)}_{2l}})$   of cardinality $2i+1$ which consists of the integers
\begin{align*}
& s_{i}, \,
s_{i}+s_{i-1},\, 
\ldots,\,
 s_{i}+s_{i-1}+ \ldots +s_{1}+s_{0}, \\
& s_{i}+s_{i-1}+ \ldots +s_{1}+2s_{0},\,
 s_{i}+s_{i-1}+ \ldots  +2s_{1}+2s_{0},
\\
&    \ldots,\, s_{i}+2s_{i-1}+ \ldots +2s_{1}+2s_{0}.
\end{align*} 
Moreover, we introduce the {\em congruence triangle}
\begin{align*}\nonumber
\Delta(s_{l-1}, \ldots, s_0; A^{(2)}_{2l})=D(s_{l-1},  \ldots, s_0; {A^{(2)}_{2l}}) \cup D(s_{l-2},  \ldots, s_0;{A^{(2)}_{2l}})\cup \cdots \cup D(s_0; {A^{(2)}_{2l}}).
\end{align*}
Its  elements belong to  the upper left right-angled triangle with vertices $s_{l-1}$, $s_0$ and $s_{l-1}+2\sum_{i=0}^{l-2}s_i$ of the $(s_0, s_1, \ldots , s_l)$-specialized array of negative roots $\mathcal{N}_{A^{(2)}_{2l}}^{(s_0,s_1,\ldots,s_l)}$. In   Figure \ref{fig:cong5}, the elements of $D(s_3, s_2,s_1,s_0; A_{8}^{(2)})$ are denoted by $\textcolor{blue}{\tcircled{3}},\textcolor{green}{\tcircled{2}},\textcolor{red}{\tcircled{1}},\textcolor{cyan}{\tcircled{0}},$ for $i=3,2,1,0$, respectively. For example,  the symbol $\textcolor{green}{\tcircled{2}}$ represents the elements
$$
s_2,\,s_2+s_1,\,s_2+s_1+s_0,\, s_2+s_1+2s_0,\, s_2+2s_1+2s_0.
$$
The remaining elements of the array   are denoted by  the symbols $\circ$ and $\bullet$.
\begin{figure}[htb!]
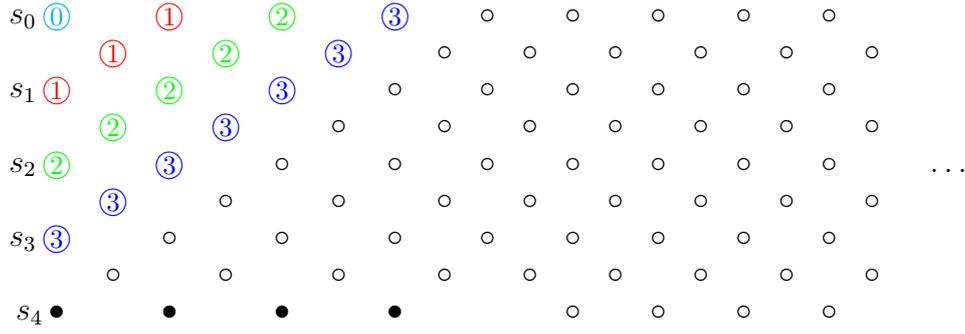

$$
\begin{array}{ccccccccccccccccccc}
  s_0 \,\textcolor{cyan}{\circled{0}} & & \textcolor{red}{\circled{1}}& & \textcolor{green}{\circled{2}}& & \textcolor{blue}{\circled{3}}& & {\circ}&  & {\circ} & &{\circ} &  &{\circ} & &{\circ} &  \\
&\textcolor{red}{\circled{1}} & &\textcolor{green}{\circled{2}} & & \textcolor{blue}{\circled{3}} & &{\circ}& & {\circ} & & {\circ} &  &{\circ} & &{\circ} &  &{\circ}\\
 s_1\,\textcolor{red}{\circled{1}} & &\textcolor{green}{\circled{2}} & & \textcolor{blue}{\circled{3}}& &{\circ} & &{\circ} & & {\circ}& &{\circ} &  &{\circ} & &{\circ} &   \\
&\textcolor{green}{\circled{2}} & & \textcolor{blue}{\circled{3}}  & &{\circ}& &{\circ}& &{\circ}& &  {\circ}& &{\circ} &  &{\circ}&  &{\circ} \\
 s_2\,\textcolor{green}{\circled{2}}& & \textcolor{blue}{\circled{3}} & &{\circ} & &{\circ} & &{\circ}&  &{\circ} & & {\circ}& &{\circ} &  &{\circ} &  \\
& \textcolor{blue}{\circled{3}}  & &{\circ}& &{\circ}& &{\circ} & &{\circ}&  &{\circ} & &  {\circ}& &{\circ} &  &{\circ}\\
 s_3\, \textcolor{blue}{\circled{3}}& &{\circ}& &{\circ}& &{\circ} & &{\circ}& & {\circ} & &{\circ} & &  {\circ}& &{\circ} &\\
&{\circ} & &{\circ} & &{\circ} & &{\circ} & &{\circ}&  &{\circ} & & {\circ} & &  {\circ}& &{\circ}\\
 s_4 \,\bullet& &\bullet & &\bullet & &\bullet& & &  &{\circ} & &  {\circ} & &{\circ} & &  {\circ}&\\
\end{array}\quad\dots
$$
\caption{Congruence triangle $\Delta(s_3, s_2,s_1,s_0; A_{8}^{(2)})$}
  \label{fig:cong5}
\end{figure}

To write Wakimoto's formulas for the specialized character of type $(1,1, \ldots,1,2)$ (see  \eqref{W83_2} ), we introduce the set 
\beq\nonumber
S(s_l, s_{l-1}, \ldots, s_1; {A^{(2)}_{2l}})=\left\{s_l, s_l+2s_{l-1}, \ldots, s_l+2s_{l-1}+\cdots +2s_1\right\}
\eeq
of cardinality $l$. In   Figure \ref{fig:cong5}, the elements of   $S(s_4,s_3,s_2,s_1; A_{8}^{(2)})$,  which are
$$s_4,\, s_4+2s_3,\,s_4+2s_3+2s_2,\,s_4+2s_3+2s_2+2s_1,$$
are indicated by the symbol $\bullet$.

\begin{example}\label{e3}
For $s=(3,2,2,3,4)$, we have
\begin{align*}
D(3,2,2,3; A_{8}^{(2)}) &=  \{3,5,7,10,13,15,17\},\\
D(2,2,3; A_{8}^{(2)}) &=  \{2,4,7,10, 12\},\\
D(2,3; A_{8}^{(2)})   &=  \{2,5,8\},\\
D(3; A_{8}^{(2)})&=  \{3\},\\
S(4,3,2,2; A_{8}^{(2)})&=  \{4,10,14,18\}.\\
\end{align*}
\end{example}
\begin{theorem}\label{T: Wakimotos formula for 1...12}
The Wakimoto product formula for the $(1,1, \ldots,1,2)$-specialized character of the standard module $L_{C_l^{(1)}}(\Lambda)$ of highest weight $\Lambda=\sum_{i=0}^l k_i\Lambda_i$ can be written as
\begin{align}
&\ch^{(1,1, \ldots,1,2; \ C_l^{(1)})}L(k_0, k_1, \ldots, k_{l-1}, k_l)\nonumber\\
=& \prod_{i \in \mathbb{N}}(1-q^i)^{-l}
\prod_{i \equiv \pm a, b \mod 2(k+l+1), \ a \in \Delta(k_{l-1}+1, \ldots, k_0+1; A^{(2)}_{2l}),  b \in \{0\}^l} (1-q^i)\nonumber\\
&\times \prod_{j\equiv \pm c \mod 4(k+l+1), \  c \in S(2(k_l+1), k_{l-1}+1, \ldots, k_{1}+1; {A^{(2)}_{2l}})}(1-q^j).\label{f3}
\end{align}
\end{theorem}
 
 \begin{example}\label{example_712a} 
From  Example \ref{e3} follows  the specialized character formula  
\begin{align*} 
\ch^{(1,1, 1,1,2; \ C_4^{(1)})}L(2,1,1,2,1)=&\prod_{\substack{i \in \mathbb{N}\\ i \not\equiv  0 \mod 24}}(1-q^i)^{-1}\prod_{k \equiv \pm 6,\pm 20  \mod 48} (1-q^k)^{-1}\\
&\hspace{-30pt}\times\prod_{j \equiv \pm 1,\pm 1, \pm 1, \pm 2, \pm 3, \pm 4,  \pm 5, \pm 6, \pm 6, \pm 8, \pm 8, \pm 9, \pm 9, \pm 11,\pm 11,12 \mod 24} (1-q^j)^{-1}.
\end{align*}
\end{example} 

For  $1 \leq i \leq l$,   define  the {\em congruence triangles} as multisets
\begin{align}\nonumber
&\Delta(s_1, \ldots, s_l; {A^{(2)}_{2l}}^T)=D(s_1,  \ldots, s_l; {A^{(2)}_{2l}}^T) \cup D(s_2,  \ldots, s_l;{A^{(2)}_{2l}}^T)\cup \cdots \cup D(s_l; {A^{(2)}_{2l}}^T),
\end{align}
where the sets $D(s_i, s_{i+1}, \ldots, s_l; {A^{(2)}_{2l}}^T )$ consist of integers
\begin{align}\nonumber
 &s_i,\, s_i+s_{i+1}, \,\ldots, \,s_i+ \ldots +s_{l-1}+ s_l, \\
 \nonumber
&s_i+s_{i+1}+ \ldots +s_{l-1}+2 s_l, \,\ldots,\, s_i+2s_{i+1}+ \ldots +2s_{l-1}+2 s_l.
\end{align}
Furthermore, let
\beq\nonumber
S(s_0,  \ldots, s_{l-1}; {A^{(2)}_{2l}}^T )=\left\{s_0, s_0+2s_1, \ldots, s_0+2s_1+\cdots +2s_{l-1}\right\}.
\eeq
\begin{example}\label{e4}
For $s=(2,3,2,2,6)$, we have
\begin{align*}
D(2,2,3,2;{A^{(2)}_{8}}^T) &= \{2,4,7,9,11,14,16\},\\
D(2,3,2;{A^{(2)}_{8}}^T) &= \{2,5,7,9, 12\},\\
D(3,2; {A^{(2)}_{8}}^T)   &= \{3,5,7\},\\
D(2; {A^{(2)}_{8}}^T)&= \{2\},\\
S(6,2,2,3; {A^{(2)}_{8}}^T)&=\{6,10,14,20\}.
\end{align*}
\end{example}

\begin{theorem}\label{T: Wakimotos formula for 21...1}
The Wakimoto formula for the specialized character of type $(2,1, \ldots,1,1)$ of the standard module of highest weight $\Lambda=\sum_{i=0}^l k_i\Lambda_i$  can be written as
\begin{align}
\ch^{(2,1, \ldots,1,1; \ C_l^{(1)})}L(k_0, k_1, \ldots, k_{l-1}, k_l)
=&  \prod_{i \in \mathbb{N}}(1-q^i)^{-l}
\prod_{\substack{ i \equiv \pm a, b \mod 2(k+l+1), \\ a \in \Delta(k_1+1, \ldots, k_l+1; {A^{(2)}_{2l}}^T),\, b \in \{0\}^l}} (1-q^i)\nonumber\\
&\times \prod_{\substack{j\equiv \pm c \mod 4(k+l+1), \\  c \in S(2(k_0+1), k_1+1 \ldots, k_{l-1}+1; {A^{(2)}_{2l}}^T )}} (1-q^j).\label{f2}
\end{align}
\end{theorem}

\begin{example}\label{example_715a} 
Using Example \ref{e4} we obtain   the specialized character formula 
\begin{align}\nonumber
\ch^{(2,1, 1,1,1; \ C_4^{(1)})}L(2,1,1,2,1)= &\prod_{\substack{i \in \mathbb{N}\\i \not\equiv 0\mod 24}}(1-q^i)^{-1}\prod_{k \equiv \pm 4,\pm 18 \mod 48} (1-q^k)^{-1}\\
\nonumber
&\hspace{-40pt}\times\prod_{j \equiv \pm 1,\pm 1, \pm 1, \pm 3,  \pm 3, \pm 4, \pm 5, \pm 6, \pm 6, \pm 8, \pm 8, \pm 9, \pm 10, \pm 11,\pm 11,12 \mod 24} (1-q^j)^{-1}. 
\end{align}
\end{example} 
 \begin{example}\label{example_716a}
The specialized characters 
$$\ch^{(2,1, 1,1,1; \ C_4^{(1)})}L(1,2,1,2,1)\qquad\text{and}\qquad \ch^{(1,1, 1,1,2; \ C_4^{(1)})}L(1,2,1,2,1)$$
coincide and they are equal to
$$\prod_{\substack{i \in \mathbb{N}\\ i \not\equiv 0,12 \mod 24}}(1-q^i)^{-1} \prod_{j \equiv \pm 1,\pm 1, \pm 1, \pm 2, \pm 3, \pm 4, \pm 4, \pm 6, \pm 6, \pm 6, \pm 8, \pm 8, \pm 9, \pm 10, \pm 11,\pm 11,\pm 11  \mod 24} (1-q^j)^{-1}.$$
\end{example} 

\subsection{Wakimoto's formulas of type $(s_0,s_1, \ldots,s_l; \ C_l^{(1)})$}

From 
\eqref{W83_4}--\eqref{W83_6}, we have the following three sets of Wakimoto's formulas of type $(s_0,s_1, \ldots,s_l; \ C_l^{(1)})$. Using notation from previous subsections we will write these formulas and give some examples.
\begin{theorem}\label{T: Wakimotos formula for 2n-1,...,2n-1,n-1}
The first set of Wakimoto's formulas of type $(s_0,s_1, \ldots,s_l; \ C_l^{(1)})$ is
\begin{align}\label{ps1}
&\ch^{(s_0,s_1, \ldots,s_l; \ C_l^{(1)})}L(2n-1, 2n-1, \ldots,2n-1,n-1)\\
\nonumber
=& \frac{\prod_{i \equiv \pm a, b \mod  n(2\sum_{j=0}^ls_j+2\sum_{j=1}^{l-1}s_j), \ a \in \Delta(2ns_{1}, ,\ldots, 2ns_{l-1}, ns_l;{A_{2l}^{(2)}}^T), \, b \in \{0\}^l} (1-q^i)}{\prod_{i \equiv \pm a, b \mod (\sum_{j=0}^ls_j+\sum_{j=1}^{l-1}s_j), \ a \in \Delta(s_1,s_2, \ldots, s_l;C_l^{(1)}), \, b \in \{0\}^l}(1-q^i)}\\
\nonumber
&\times \prod_{j\equiv \pm c \mod 2n(2\sum_{m=0}^ls_m+2\sum_{m=1}^{l-1}s_m), \  c \in S(2ns_0, 2ns_{1}, \ldots,  2ns_{l-1}; {A^{(2)}_{2l}}^T)}(1-q^j).
\end{align}
\end{theorem}
\begin{example}\label{example_718a}
 For $l=2$, $s=(3,1,1)$ and $n=1$, we have $$\ch^{(3,1,1; \ C_2^{(1)})}L(1,1,0)=\frac{\prod_{i \equiv \pm 5 \mod 12}  (1+q^i)}{\prod_{j\in  \mathbb{N} } (1-q^{2j-1})\prod_{r \equiv \pm  6  \mod 24} (1-q^r)}.$$ 
For $l=2$, $s=(3,1,1)$ and $n=2$, we have $$\ch^{(3,1,1; \ C_2^{(1)})}L(3,3,1)=\frac{\prod_{i \equiv \pm 10 \mod 24}  (1+q^i)}{\prod_{j\in  \mathbb{N} } (1-q^{2j-1})^2\prod_{r \equiv \pm 6, 12  \mod 24} (1-q^r)}.$$
 For $l=2$, $s=(3,1,1)$ and $n=3$, we have 
\begin{align*}
\ch^{(3,1,1; \ C_2^{(1)})}L(5,5,2)=
 &\prod_{i \in \mathbb{N}, i \not\equiv 0,\pm 30, 36 \mod 72} (1-q^i)^{-1}\\
&\times \prod_{r \equiv \pm 1, \pm 5, \pm 7, \pm 11, \pm 13,\pm 15, \pm 17  \mod 36} (1-q^r)^{-1} .
\end{align*} 
\end{example}

\begin{theorem}\label{T: Wakimotos formula for n-1,2n-1,...,2n-1}
The second set of Wakimoto's formulas of type $(s_0,s_1, \ldots,s_l; \ C_l^{(1)})$ is
\begin{align}\label{ps2}
&\ch^{(s_0,s_1, \ldots,s_l; \ C_l^{(1)})}L(n-1, 2n-1, \ldots,2n-1,2n-1)\\
\nonumber
=& \frac{\prod_{i \equiv \pm a, b \mod n(2\sum_{j=0}^ls_j+2\sum_{j=1}^{l-1}s_j), \ a \in \Delta(2ns_{l-1}, ,\ldots, 2ns_{1}, ns_0;A_{2l}^{(2)}),\, b \in\{0\}^l} (1-q^i)}{\prod_{i \equiv \pm a, b \mod (\sum_{j=0}^ls_j+\sum_{j=1}^{l-1}s_j), \ a \in \Delta(s_1,s_2, \ldots, s_l;C_l^{(1)}),\,  b \in \{0\}^l}(1-q^i)}\\
\nonumber
&\times \prod_{j\equiv \pm c \mod 2n(2\sum_{m=0}^ls_m+2\sum_{m=1}^{l-1}s_m), \  c \in S(2ns_l, 2ns_{l-1}, \ldots,  2ns_{1}; {A^{(2)}_{2l}})}(1-q^j).\end{align}
\end{theorem}
\begin{example}\label{example_720a}
 For $l=2$, $s=(1,3,1)$ and $n=1$, we have 
$$\ch^{(1,3,1; \ C_2^{(1)})}L(0,1,1)=  \frac{\prod_{i \in\mathbb{N}}(1+q^{2i-1})}{ \prod_{j \equiv \pm 4 \mod 16} (1-q^j)^2}.$$ 
 For $l=2$, $s=(1,3,1)$ and $n=2$, we have $$\ch^{(1,3,1; \ C_2^{(1)})}L(1,3,3)  =\frac{\prod_{i \equiv \pm1,\pm 7,\pm 9, \pm 15 \mod 32}(1+q^i)}{\prod_{j\in\mathbb{N}} (1-q^{2j-1})\prod_{r \equiv \pm 4, \pm 8, \pm 8, \pm 12 \mod 32} (1-q^r)}.$$  
\end{example}
\begin{theorem}\label{T: Wakimotos formula for n-1,2n-1,...,2n-1,n-1}
The third set of Wakimoto's formulas of type $(s_0,s_1, \ldots,s_l; \ C_l^{(1)})$ is
\begin{align}\label{ps3}
&\ch^{(s_0,s_1, \ldots,s_l; \ C_l^{(1)})}L(n-1, 2n-1, \ldots,2n-1,n-1)=\\
\nonumber
=&  \prod_{i \equiv \pm a, b \mod (\sum_{j=0}^ls_j+\sum_{j=1}^{l-1}s_j), \ a \in \Delta(s_1,s_2, \ldots, s_l;C_l^{(1)}), \,b \in \{0\}^l}(1-q^i)^{-1} \\
\nonumber
&\times \prod_{\substack{i \equiv a, \pm b \mod 2n(\sum_{j=0}^ls_j+\sum_{j=1}^{l-1}s_j),\\ a \in \Delta(ns_0, 2ns_1,\ldots,2ns_{l-1}, ns_l;D_{l+1}^{(2)}) \cup \{0\}^l,\, b\in\Delta(2ns_1,\ldots,2ns_{l-1}, ns_l;D_{l+1}^{(2)})} } (1-q^i).\end{align}
\end{theorem}
\begin{example}\label{example_722a}
For $l=2$, $s=(3,1,1)$ and $n=1$, we have 
$$\ch^{(3,1,1; \ C_2^{(1)})}L(0,1,0)=\prod_{i \equiv \pm 1 \mod 6} (1-q^i)^{-1}\prod_{j \equiv 6  \mod 12} (1-q^j)^{-1}.$$
For $l=2$, $s=(3,1,1)$ and $n=2$, we have 
$$\ch^{(3,1,1; \ C_2^{(1)})}L(1,3,1)=\prod_{i\in\mathbb{N}} (1-q^{2i-1})^{-2}\prod_{j \equiv 12  \mod 24} (1-q^j)^{-1}.$$
  For $l=2$, $s=(3,1,1)$ and $n=3$, we have 
\begin{align*}
\ch^{(3,1,1; \ C_2^{(1)})}L(2,5,2)=&
\prod_{\substack{i \in \mathbb{N}\\ i \not\equiv 0, \pm 9, \pm 27, 36 \mod 72}} (1-q^i)^{-1} \prod_{j \equiv \pm 1 \mod 6} (1-q^j)^{-1}.
\end{align*} 
\end{example}

\section{Borcea's correspondence for $C_l^{(1)}$ and $A_{2l}^{(2)}$}\label{section_08}

Let $k \in \mathbb{N}$ be fixed. The number of level $k=\sum_{i=0}^lk_i$ standard $C_l^{(1)}$-modules of highest weight $\Lambda=\sum_{i=0}^lk_i\Lambda_i$ is $\binom{k+l}{l}$. This number is the same as the number of level $2k+1$ standard $A_{2l}^{(2)}$-modules of highest weight $(2k_0+1)\Lambda_0+\sum_{i=1}^lk_i\Lambda_i$. 

It follows directly from the Lepowsky--Wakimoto product formula for the specialized characters of the Weyl--Kac character formulas of types $(2,1,\ldots, 1,1;C_l^{(1)})$ and $(1,1,\ldots, 1; A_{2l}^{(2)})$ (see \cite[Theorem 1]{W1} and \cite[Corollary 2.2.8]{W2}) and Lemmas \ref{L: Q(21...1;C)} and \ref{L: Q(1...1;a)} that
\begin{theorem} 
Let $\ch^{(1,1,\ldots, 1; A_{2l}^{(2)})}L(2k_0+1,k_1, \ldots, k_l)$ denote the principally specialized character of level $2k+1$ standard module $L(\Lambda)$ of the affine Lie algebra of type $A_{2l}^{(2)}$ with highest weight $\Lambda=(2k_0+1)\Lambda_0+\sum_{i=1}^lk_i\Lambda_i$. We have
\beq\label{id1}\ch^{(2,1,\ldots, 1; C_l^{(1)})} L(k_0,k_1 \ldots, k_l)=F(q) \cdot \ch^{(1,1,\ldots, 1; A_{2l}^{(2)})}L(2k_0+1,k_1, \ldots, k_l),\eeq
where 
$$F(q)
=\prod_{i\equiv \pm 1, \pm 3, \ldots , \pm(2l-1) \mod 2(2l+1)} (1-q^i).$$
\end{theorem}
\begin{proof}
From \cite[Theorem 1]{W1} (see also \cite{W2}) we have
\beqn\ch^{(2,1,\ldots, 1; C_l^{(1)})} L(k_0,k_1 \ldots, k_l)=\frac{Q(k_l+1,\ldots,k_1+1, 2(k_0+1); {A_{2l}^{(2)}})}{Q(2,1,\ldots,1;C_l^{(1)})},\eeqn
and 
\beqn \ch^{(1,1,\ldots, 1; A_{2l}^{(2)})}L(2k_0+1,k_1, \ldots, k_l)=\frac{Q(k_l+1,\ldots,k_1+1, 2(k_0+1); {A_{2l}^{(2)}})}{Q(1,1,\ldots,1;A_{2l}^{(2)})}.\eeqn
The theorem now follows from  Lemmas \ref{L: Q(21...1;C)} and \ref{L: Q(1...1;a)}.
\end{proof}

In this work we are interested in affine Lie algebras of type $ C_l^{(1)}$,  $\l\geq 2$, but most of our considerations hold as well for $l=1$ when $ C_1^{(1)}\cong A_1^{(1)}$. 
From that point of view the identity in \eqref{id1} is a generalization of the identity obtained in \cite{B} in the case when $l=1$. In the case of $l=1$ and $k=1$, we have two identities \eqref{id1} 
\begin{eqnarray}\nonumber
\ch^{(2,1; A_1^{(1)})} L(1,0)&=&\prod_{i\equiv \pm 1 \mod 6} (1-q^i)\cdot \ch^{(1,1; A_{2}^{(2)})}L(3,0)\\
\nonumber
&=&\frac{1}{ \prod_{i\equiv \pm 2,\pm 3 \mod 12} (1-q^{i})},\end{eqnarray}
and
\begin{eqnarray}\nonumber
\ch^{(2,1; A_1^{(1)})} L(0,1)&=&\prod_{i\equiv \pm 1 \mod 6} (1-q^i)\cdot \ch^{(1,1; A_{2}^{(2)})}L(1,1)\\
\nonumber
&=&\frac{\prod_{i\equiv \pm 2 \mod 12} (1-q^i)}{\prod_{i\in \mathbb{N}} (1-q^{2i-1})},\end{eqnarray}
and the corresponding products appear in \cite{MP} as product sides of two combinatorial identities for level $1$ standard $A_1^{(1)}$-modules and, on the other side, the same products appear in \cite{Capp} as product sides of Capparelli's identities for level $3$ standard $A_2^{(2)}$-modules. Moreover, by some sort of coincidence, the difference conditions for partitions coincide in both cases.

\section{Combinatorial identities}\label{section_09}

Let $l\geq2$. A downward path $\mathcal Z$ in the array $ \mathcal{B}_{C_l^{(1)}}$ is a sequence (or a subset) with an element in the top row followed by an adjacent element in the second row, and so on all the way to an element in the bottom row. The vertices denoted by $\bullet$'s in Figure \ref{F: extended arrangement of negative root vectors  for C2(1)} gives an example of a downward path.

A monomial
\begin{equation}\label{E:colored partition pi}
\pi=\prod_{b(j)\in \mathcal{B}_{C_l^{(1)}}^{-}}b(j)\sp{m_{b(j)}}\in S({\mathfrak g}(C_l^{(1)}))
\end{equation}
can be interpreted as a colored partition $\pi$: for $m_{b(j)}>0$ we say that $b(j)$ is a part of degree $|b(j)|=j$ and color $b\in \mathcal{B}_{C_l}(\theta)$ which appears in the partition $m_{b(j)}$ times. We define the degree, weight and length of a colored partition $\pi$ as
\begin{equation}\label{E:degree and length of pi}
|\pi| = \sum_{b(j)\in \bar{B}_\ell^{<0}}j\cdot{m_{b(j)}},\quad
\wt_{\mathfrak{h}}(\pi)
 = \sum_{b(j)\in \bar{B}_\ell^{<0}}\wt_{\mathfrak{h}}(b)\cdot{m_{b(j)}},\quad 
\ell(\pi)=\sum_{b(j)\in \bar{B}_\ell^{<0}}{m_{b(j)}}.
\end{equation}
We interpret $1$ as an empty partition $\emptyset$ with no parts of degree $0$ and length $0$.
A colored partition $\pi$ can be also interpreted as a function 
$$
\pi\colon  \mathcal{B}_{C_l^{(1)}}^-\to \mathbb Z_{\geq0},\qquad b(n)\mapsto m_{b(n)}
$$
with finite support.

Let $k$ be a positive integer. We say that a {\it colored partition $\pi$ satisfies the level $k$-difference conditions} if
\begin{equation}\label{E:level k difference conditions}
\sum_{b(j)\in \mathcal Z}m_{b(j)}\leq k\quad\text{for all downward paths}
\ \mathcal Z\subset \mathcal{B}_{C_l^{(1)}}^{-}.
\end{equation}
To formulate the {\it initial conditions} for $\pi$ we need to extend $\mathcal{B}_{C_l^{(1)}}^{-}$ with Chevalley generators associated to simple roots and simple coroots in $\mathfrak h$, i.e.
$$
\mathcal{B}_{C_l^{(1)}}^{-}\subset\mathcal{B}_{C_l^{(1)}}^{-e}=
\{e_0, e_1, \dots, e_l\}\cup\{h_1, \dots, h_l\}\cup
\mathcal{B}_{C_l^{(1)}}^{-}\subset\mathcal{B}_{C_l^{(1)}}.
$$
For example, for $l=2$ we have the extended array of root vectors on Figure \ref{F: extended arrangement of negative root vectors  for C2(1)} (with omitted arrows---compare with Figure \ref{F: the arrangement of negative root vectors  for C2(1)}).
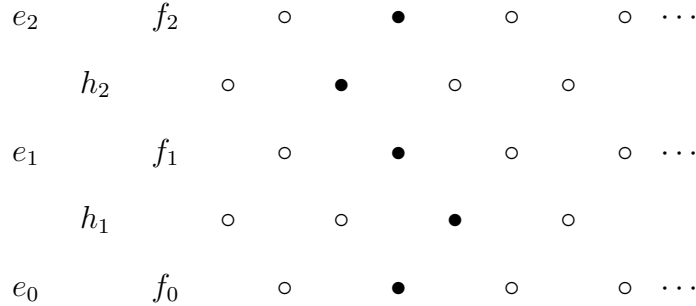
\begin{figure}[h!]
$$
\xymatrix@=0.7em {
e_2 &&
f_2&&
\circ&&
\bullet&&
\circ&&
\circ&
\hspace{-5pt}\cdots
\\
&h_2&&
\circ&&
\bullet&&
\circ&&
\circ&&
\\
e_1& &
f_1&&
\circ&&
\bullet&&
\circ&&
\circ&
\hspace{-5pt}\cdots
\\
&h_1&&
\circ&&
\circ&&
\bullet&&
\circ&&
 \\
{e_0}&&
{f_0}&&
\circ&&
\bullet&&
\circ&&
\circ &
\hspace{-5pt}\cdots
}$$\caption{The extended arrangement of negative root vectors  for $C_2^{(1)}$ and a downward path}
\label{F: extended arrangement of negative root vectors  for C2(1)}
\end{figure}
For nonnegative integers $k_0, k_1, \dots, k_l$, $k_0+ k_1+ \dots+ k_l=k$ and dominant  integral $\Lambda=k_0\Lambda_0+ k_1\Lambda_1+ \dots+ k_l\Lambda_l$ we say that {\it $\pi$ is $\Lambda$-admissible if the monomial}
$$
e_0^{k_0} e_1^{k_1} \dots  e_l^{k_l}h_1^0\dots  h_l^0\cdot\pi
$$
satisfies the level $k$ difference conditions on the extended array $\mathcal{B}_{C_l^{(1)}}^{-e}$. The generating function for $\Lambda$-admissible colored partitions is
\begin{equation}\label{E:generating function for admissible partitions}
\sum_{\pi\text{ is } \Lambda\text{-admissible}} e^{|\pi|\delta}
e^{\wt_{\mathfrak{h}}(\pi)}
\end{equation}
In \cite{CMPP} it was (implicitly) conjectured:
\begin{conjecture}\label{C: conjecture}
The generating function for $\Lambda$-admissible colored partitions is $\operatorname{ch}e^{-\Lambda}L_{C_l^{(1)}}(\Lambda) $ for $l, k\geq 2$.
\end{conjecture}
\begin{remark}
For $l=1$ the statement above is proved in \cite{MP, FKLMM} for all $\Lambda$ and in \cite{F} for $\Lambda=k\Lambda_0$, $k\geq1$. For $l\geq2$ and $k=1$ the statement above is proved in \cite{DK, R} for all $\Lambda$ and in \cite{PS1} for $\Lambda=\Lambda_0$.
\end{remark}

The $(s_0, s_1, \dots, s_l)$-specialization defines a map 
$$
\mathcal{B}_{C_l^{(1)}}^{-}\to \mathcal{N}_{C_l^{(1)}}^{(s_0, s_1, \dots, s_l)},\qquad
b(j)\mapsto a_{b(j)}=\Vert b(j)\Vert,
$$
which extends to a map
$$
\pi=\prod_{b(j)\in \mathcal{B}_{C_l^{(1)}}^{-}}b(j)\sp{m_{b(j)}}\mapsto
\Vert\pi\Vert
=\sum_{a_{b(j)}\in \mathcal{N}_{C_l^{(1)}}^{(s_0, s_1, \dots, s_l)}}a_{b(j)} \cdot{m_{b(j)}}.
$$
In other words, the $(s_0, s_1, \dots, s_l)$-specialization of a colored partition $\pi$ gives a partition $\Vert\pi\Vert$ defined on the array $\mathcal{N}_{C_l^{(1)}}^{(s_0, s_1, \dots, s_l)}$, and Conjecture \ref{C: conjecture} implies that the generating function 
$$
\sum_{\pi\text{ is } \Lambda\text{-admissible}} q^{-\Vert\pi\Vert}
$$
for $\Lambda$-admissible colored partitions $\Vert\pi\Vert$ is the specialized character $\operatorname{ch}^{(s_0, s_1, \dots, s_l; \,C_l^{(1)} )}L(\Lambda)$. 
If the specialized character can be written as an infinite periodic product, the conjecture gives a Rogers--Ramanujan-type combinatorial identity. By using Lepowsky's product formula for the principal specialization $(1,1,\dots,1)$, the conjectured identities are formulated in \cite{CMPP}. By using Wakimoto's product formulas we get other Rogers--Ramanujan-type combinatorial identities.

\begin{example}\label{Ex: 21...1 identities} For $l=2$ and $(2,1,1)$-specialization we have the array $\mathcal N_{C_2^{(1)}}^{(2, 1, 1)}$. For $\Lambda= k\Lambda_0$, $k\geq2$, we can apply Theorem \ref{T: Wakimotos formula for 21...1} and get the conjectured identity:

The generating function for partitions on the array
$$
\begin{matrix}
& & & & \textbf{\upshape 7}& & \textbf{\upshape 9}& & \textbf{\upshape 11}&   \\
& & &\textit{\textcolor{red}{5}} & &\textbf{\upshape 8}& &\textbf{\upshape 10} & & \textbf{\upshape 12}\\
& &\textit{\textcolor{red}{4}} & &\textit{\textcolor{red}{6}} & &\textbf{\upshape 9} & &\textbf{\upshape 11} & \\
&\textit{\textcolor{red}{3}} & &\textit{\textcolor{red}{5}} & &\textit{\textcolor{red}{7}} & &\textbf{\upshape 10} & &\textit{\textcolor{red}{13}} \\
\textit{\textcolor{red}{2}}& &\textit{\textcolor{red}{4}} & &\textit{\textcolor{red}{6}} & &\textit{\textcolor{red}{8}} & &\textit{\textcolor{red}{12}} &
\end{matrix}
$$
satisfying level $k$ difference conditions is
\begin{align*}
&  \prod_{i \in \mathbb{N}}(1-q^i)^{-2}
\prod_{\substack{i \equiv \pm a, b \mod 2(k+3), \\ a \in \{1,1,2,3\}, \, b \in \{0\}^2}} (1-q^i) 
 \prod_{\substack{j\equiv \pm c \mod 4(k+3), \\  c \in \{2(k+1), 2(k+2)\}}} (1-q^j).
\end{align*}
For $k=1$ this is a theorem due to results in \cite{PS1, DK, R}, and $k>1$ due to \cite{PT}.
\end{example}

\begin{remark}
The product formula in Example \ref{example_79a} is the conjectured generating function for $(1, 2, 1, 2, 1)$-admissible partitions
$$ 
n=\sum_{a\in\mathcal N}m_a\cdot a
$$
on the array $\mathcal N=\mathcal N_{C_4^{(1)}}^{(2,1,1,1,2)}$ that consists of $3$ copies of the positive integers,  one copy of the even positive integers  and one copy of the positive integers congruent to $5$ modulo $10$.

The product formulas in Examples \ref{example_712a}  and \ref{example_715a} are the conjectured generating functions for $(2, 1, 1, 2, 1)$-admissible partitions on two different arrays (see Figures \ref{fig:arrayC3} and \ref{fig:arrayC4corr}) that consist of $4$ copies of the positive integers.
The product formula in Example \ref{example_716a} is another conjectured generating function for admissible partitions, also satisfying level $7$ difference conditions.

For specializations $(2,1,\dots,1,2)$, $(2,1,\dots,1,1)$ and $(1,1,\dots,1,2)$, it seems that we have the conjectured Rogers--Ramanujan-type  partition identity for all $\Lambda$-admissible partitions. On the other side, Theorems \ref{T: Wakimotos formula for 2n-1,...,2n-1,n-1}, \ref{T: Wakimotos formula for n-1,2n-1,...,2n-1} and \ref{T: Wakimotos formula for n-1,2n-1,...,2n-1,n-1} correspond to Wakimoto's unspecialized product formulas, but only for specific highest weights $\Lambda$. The product formulas in Examples \ref{example_718a}, \ref{example_720a} and \ref{example_722a} have  combinatorial interpretations, and Conjecture \ref{C: conjecture} gives
the corresponding Rogers--Ramanujan-type  partition identities; it is the theorem in level $1$ case due to \cite{DK}, and the conjectures in all the other cases. It seems that unspecialized Wakimoto's product formulas should provide many new (conjectured) Rogers--Ramanujan-type  partition identities for different specializations.
\end{remark}

 \section*{Acknowledgment}
 The authors would like to thank Thomas Edlund for  producing the images of Dynkin diagrams in Figure \ref{Dynkin diagrams}  and for his proofreading of the paper.  
 This work was supported by the project "Implementation of cutting-edge research and its application as part of the Scientific Center of Excellence for Quantum and Complex Systems, and Representations of Lie Algebras", PK.1.1.02, European Union, European Regional Development Fund. Moreover, this work has been  supported  by Croatian Science Foundation under the project UIP-2019-04-8488.  Also, this work has been partially supported by the Croatian Science Foundation under the project IP-2022-10-9006.
\linespread{1.0}


\begin{thebibliography}{9}
\bibitem[A1]{A1}
G. E. Andrews,
{\em An analytic generalization of the Rogers-Ramanujan identities for odd moduli},
Proc. Natl. Acad. Sci. U.S.A. \textbf{71} (1974), 4082--4085.

\bibitem[A2]{A2}
G. E. Andrews,
{\em The theory of partitions}, 
Encyclopedia of Mathematics and Its Applications, Vol. 2, Addison-Wesley, 1976.

\bibitem[ABF]{ABF}
G. E. Andrews, R. J. Baxter and P. J. Forrester, 
{\em Eight-vertex SOS model and generalized Rogers--Ramanujan-type identities}, 
J. Stat. Phys. \textbf{35} (1984), 193--266.

\bibitem[AsI]{AI}
R. Askey and M. E.-H. Ismail, 
{\em A generalization of ultraspherical polynomials}, 
in {\em Studies in pure mathematics},   55--78, Birkh\"{a}user, Basel, 1983.

\bibitem[B]{Bax}
R. J. Baxter, 
{\em Rogers--Ramanujan identities in the hard hexagon model}, 
J. Stat. Phys. \textbf{26} (1981), 427--452.

\bibitem [Bor]{B}    
J. Borcea,
{\em Dualities and vertex operator algebras of affine type}, J. Algebra {\bf 258} (2002), 389--441.


\bibitem [Bou1]{Bou1}    
N. Bourbaki, Groupes et alg\`ebres de Lie, Chapitres IV--VI,
Hermann, Paris, 1968.

\bibitem [Bou2]{Bou2}    
N. Bourbaki, Groupes et alg\`ebres de Lie, Chapitres VII--VIII,
Hermann, Paris, 1975.

\bibitem[Br1]{Br1}
D. M. Bressoud, 
{\em A generalization of the Rogers--Ramanujan identities for all moduli},
J. Combin. Theory Ser. A \textbf{27} (1979), 64--68.

\bibitem[Br2]{Br2}
D. M. Bressoud, 
{\em An analytic generalization of the Rogers--Ramanujan identities with interpretation}, 
Quart. J. Math. Oxford  \textbf{31} (1980), 385--399.

\bibitem[Br3]{Br}
D. M. Bressoud, 
{\em On partitions, orthogonal polynomials and the expansion of certain infinite products},
Proc. London Math. Soc. \textbf{42} (1981), 478--500.

\bibitem[BCFK]{BCFK}
K. Bringmann, C. Calinescu, A. Folsom  and S. Kimport, 
{\em Graded dimensions of principal subspaces and modular Andrews--Gordon series}, Commun. Contemp. Math. \textbf{16} (2014), 1350050.

\bibitem[BMS]{BMS}
 C. Bruschek, H. Mourtada and J. Schepers, 
{\em Arc spaces and the Rogers--Ramanujan identities}, 
Ramanujan J. \textbf{30} (2013), 9--38.

\bibitem [C]{Capp}    
S. Capparelli, {\em On some representations of twisted affine Lie algebras and combinatorial identities},
J. Algebra {\bf 154} (1993), 335--355.

\bibitem[CMPP]{CMPP}   
S. Capparelli, A. Meurman, A. Primc and M. Primc, 
{\em  New partition identities from $C_l^{(1)}$-modules},  
Glas. Mat. Ser. III  \textbf{57}, No. 2 (2022), 161--184;
\href{https://arxiv.org/abs/2106.06262}{arXiv:2106.06262 [math.RT]}.

\bibitem[Car]{Car}
R. W. Carter, Lie Algebras of Finite and Affine Type, Cambridge University Press, Cambridge, 2005.

\bibitem[CF]{CF}
I. Cherednik and B. Feigin, 
{\em Rogers--Ramanujan type identities and Nil-DAHA}, 
Adv. Math. \textbf{248} (2013), 1050--1088;
\href{https://arxiv.org/abs/1209.1978}{arXiv:1209.1978 [math.QA]}.

\bibitem [DK]{DK} J. Dousse, I. Konan, \textit{Characters of level $1$ standard modules of $C_n^{(1)}$ as generating functions for generalised
partitions}, \href{https://arxiv.org/abs/2212.12728}{arXiv:2212.12728 [math.CO]}.

\bibitem [FKLMM]{FKLMM} B. Feigin, R. Kedem, S. Loktev, T. Miwa and E. Mukhin,
 \textit{Combinatorics of the $\widehat{\mathfrak sl}_2$ spaces of coinvariants},
Transform. Groups  \textbf{6} (2001), 25--52;
\href{https://arxiv.org/abs/math-ph/9908003}{arXiv:math-ph/9908003}.

\bibitem [F]{F} E. Feigin, \textit{The PBW filtration},  Represent. Theory {\bf 13} (2009), 165--181; 
\href{https://arxiv.org/abs/math/0702797}{arXiv:math/0702797 [math.QA]}.

\bibitem[Go]{Go}
B. Gordon,
{\em A combinatorial generalization of the Rogers-Ramanujan identities},
Amer. J. Math. \textbf{83} (1961), 393--399.

\bibitem[GOR]{GOR}
E. Gorsky, A. Oblomkov and J. Rasmussen, 
{\em On stable Khovanov homology of torus knots}, 
Exp. Math. \textbf{22} (2013), 265--281;
\href{https://arxiv.org/abs/1206.2226}{arXiv:1206.2226 [math.GT]}.

\bibitem[GOW]{GOW}
 	M. J. Griffin, K. Ono and S. O. Warnaar,
{\em A framework of Rogers--Ramanujan identities and their arithmetic properties},
 Duke Math. J.  \textbf{165} (2016), 1475--1527. 
\href{https://arxiv.org/abs/1401.7718}{arXiv:1401.7718 [math.NT]}.

\bibitem[HK]{HK} 
J. Hong and S. Kang, 
Introduction to Quantum Groups and Crystal Bases,  
Graduate Studies in Mathematics Vol. 42, American Mathematical Society, 2002.

\bibitem[Hum]{Hum} 
J. E. Humphreys, Introduction to Lie Algebras and Representation theory, Springer, New York, 1972.

\bibitem[Kac]{Kac} 
V. G. Kac, 
Infinite Dimensional Lie Algebras, 
Cambridge University Press, Cambridge, 3rd edition, 1990.

\bibitem[K1]{K90} 
M. Kashiwara,  
{\em Crystalizing the $q$-analogue of universal enveloping algebras}, 
Comm. Math. Phys. {\bf 133} (1990), 249--260.

\bibitem[K2]{K91} 
M. Kashiwara,  
{\em On crystal bases of the $q$-analogue of universal enveloping algebras}, 
 Duke Math. J. {\bf 63} (1990), 465--516.

\bibitem[KKMMNN]{KKMMNN} 
S.-J. Kang, M. Kashiwara, K. C. Misra, T. Miwa, T. Nakashima and A. Nakayashiki, 
{\em Affine crystals and vertex models}, 
Internat. J. Modern Phys. A \textbf{7} (1992), 449--484.

\bibitem[L]{Lepowsky} J. Lepowsky, {\em Applications of the numerator formula to $k$-rowed plane partitions}, Adv. Math. {\bf 35} (1980), 179--194.

\bibitem[LM]{LM}
J. Lepowsky, S. Milne,
{\em Lie algebraic approaches to classical partition identities},
Adv. Math. \textbf{29} (1978), 15--59.

\bibitem[LW1]{LW1}
J. Lepowsky, R. L. Wilson, 
{\em Construction of the affine Lie algebra $A_{1}^{(1)}$},
Comm. Math. Phys. \textbf{62} (1978), 43--53.

\bibitem[LW2]{LW2}
J. Lepowsky, R. L. Wilson, 
{\em A new family of algebras underlying the Rogers-Ramanujan identities and generalizations}, 
Proc. Nat. Acad. Sci. U.S.A. \textbf{78} (1981), 7254--7258.

\bibitem[LW3]{LW3}
J. Lepowsky, R. L. Wilson, 
{\em A Lie theoretic interpretation and proof of the Rogers-Ramanujan identities},
Adv. Math. \textbf{45} (1982), 21--72.

\bibitem[LW4]{LW4}
J. Lepowsky, R. L. Wilson, 
{\em The structure of standard modules, \textbf{I}: Universal algebras and the Rogers-Ramanujan identities}, 
Invent. Math. \textbf{77} (1984), 199--290; 
{\em \textbf{II}, The case $A_1$, principal gradation},
Invent. Math. \textbf{79} (1985), 417--442.


\bibitem [MP]{MP}    
A. Meurman, M. Primc, {\em Annihilating fields of standard modules of $ \mathfrak{sl}(2,\mathbb{C})^\sim$ and combinatorial identities}, Memoirs of the Amer. Math. Soc.
\textbf{137}, No. 652 (1999).
 


\bibitem[P\v{S}1]{PS1} 
M. Primc and T. \v Siki\' c, 
{\em Combinatorial bases of basic modules for affine Lie algebras $C_n^{(1)}$}, 
J. Math. Phys. {\bf 57}  (2016), 091701 (19pp);
\href{https://arxiv.org/abs/1603.04399}{arXiv:1603.04399 [math.QA]}.

\bibitem[P\v{S}2]{PS2}  
M. Primc and T. \v Siki\' c, 
{\em Leading terms of relations for standard modules of  $C_{n}\sp{(1)}$}, 
 Ramanujan J. \textbf{48} (2019), 509--543;
\href{https://arxiv.org/abs/1506.05026}{arXiv:1506.05026 [math.QA]}.

\bibitem[PT]{PT}
M. Primc and G. Trup\v cevi\' c,
{\em Linear independence for $C_\ell^{(1)}$ by using $C_{2\ell}^{(1)}$},
preprint.

\bibitem [R]{R} M. C. Russell, 
{\em Companions to the Andrews--Gordon and Andrews--Bressoud identities, and recent conjectures of Capparelli, Meurman, Primc, and Primc}, 
\href{https://arxiv.org/abs/2306.16251}{arXiv:2306.16251 [math.CO]}.

\bibitem[W1]{W1} 
M. Wakimoto, 
{\em Two formulae for specialized characters of Kac--Moody Lie algebras}, 
preprint.

\bibitem[W2]{W2} 
M. Wakimoto, 
Lectures on infinite--dimensional Lie algebra,  
World Scientific, 2001.

\bibitem[W]{W}
S. O. Warnaar,
{\em The $A_2$ Andrews--Gordon identities and cylindric partitions},
 Trans. Amer. Math. Soc. Ser. B  \textbf{10} (2023), 715--765;
\href{https://arxiv.org/abs/2111.07550}{arXiv:2111.07550 [math.CO]}.

 


\end{thebibliography}
\end{document}